\documentclass[11pt]{amsart}
\usepackage{cite}
\usepackage{pinlabel}
\usepackage{amsmath} 
\usepackage{amsthm} 

\usepackage{graphicx} 
\setkeys{Gin}{keepaspectratio}
\usepackage{hyperref}
\usepackage{url}
\usepackage{bm}
\usepackage{mathabx}
\usepackage[left=1.25in,top=1in,right=1.25in,bottom=1in,head=.1in]{geometry}
\usepackage{xcolor}
\usepackage{tikz}
\usetikzlibrary{cd}
\usepackage{adjustbox}
\usepackage[normalem]{ulem} % 
\usepackage{enumitem}

\usepackage{verbatim}
\usepackage{xspace}
\newcommand{\items}{\begin{itemize}[leftmargin=25pt,rightmargin=15pt]
  \setlength\itemsep{2pt}}
\newcommand{\stopitems}{\end{itemize}}

\usepackage[textsize=scriptsize]{todonotes}

\usepackage{leftidx}% http://ctan.org/pkg/leftidx
\usepackage{fancyhdr}
\newtheorem{theorem}{Theorem}[section] % numbered theorems, lemmas, etc
\newtheorem*{theorem*}{Theorem}
\newtheorem{lemma}[theorem]{Lemma}

\newtheorem*{conjecture*}{Conjecture}

\newtheorem*{question*}{Question}
\newtheorem*{lemma*}{Lemma}
\newtheorem{proposition}[theorem]{Proposition}
\newtheorem{corollary}[theorem]{Corollary}
\newtheorem*{corollary*}{Corollary}
\theoremstyle{definition}
\newtheorem{definition}[theorem]{Definition}
\newtheorem{remark}[theorem]{Remark}

\newtheorem*{remark*}{Remark}

\newtheorem{construction}[theorem]{Construction}
\newtheorem*{example*}{Example}
\newtheorem*{remarks*}{Remarks}
\newtheorem*{addenda*}{Addenda}
\newtheorem*{construction*}{Construction}

\newcommand{\Tc}{\check{T}}
\newcommand{\data}{\varpi} %% for (g,\eta)  %% Dave prefers \Delta

\newcommand{\cc}{{\bf c}}  %%% replace with line below if we prefer eta
\newcommand{\bu}{{\bf U}}

\newcommand{\bT}{{\bf T}}

\newcommand{\bW}{{\bf W}}
\newcommand{\bX}{{\bf X}}

\newcommand{\bZ}{{\bf Z}}

\newcommand{\balpha}{\boldsymbol\alpha}
\newcommand{\bbeta}{\boldsymbol\beta}

\newcommand{\sss}{S^2\mkern-1.5mu \times \mkern-1.25mu S^2} % \ss is German double-s

\newcommand{\SSS}{S^1\mkern-1.5mu \times \mkern-1.25mu S^2} % \ss is German double-s

\newcommand{\zz}{\bZ} % notation for specific manifolds Z whose higher D's we calculate

\DeclareMathOperator{\psc}{psc}
\DeclareMathOperator{\NS}{NS} % for NonSpin
\newcommand{\zzn}{{\bZ}_{\NS}} % nonspin version
\newcommand{\zzs }{{\bZ}_S} % spin version
\newcommand{\bztilde}{\widetilde{\bZ}} % auxiliary manifold for families of diffeomorphisms
\newcommand{\bzptilde}{\widetilde{\bZ}^p}
\DeclareMathOperator{\Ric}{Ric}
\DeclareMathOperator{\diff}{Diff}
\DeclareMathOperator{\bdiff}{BDiff}
\DeclareMathOperator{\tdiff}{TDiff} % Torelli group
\DeclareMathOperator{\diffb}{Diff^{\,\flat}}
\DeclareMathOperator{\btdiff}{B\tdiff}
\DeclareMathOperator{\bhomeo}{{BHomeo}^0}
\DeclareMathOperator{\bhome}{BHomeo}
\DeclareMathOperator{\bdiffo}{{BDiff}^{\hspace{.07em}0\hspace{-.05em}}}

\DeclareMathOperator{\Emb}{Emb}
\DeclareMathOperator{\Embtop}{Emb^{\Top}}
\DeclareMathOperator{\home}{Homeo}
\DeclareMathOperator{\diffo}{{Diff}^{\hspace{.07em}0\hspace{-.05em}}}
\DeclareMathOperator{\homeo}{{Homeo}^0}
\DeclareMathOperator{\Top}{Top}  %%% or if we prefer change to C^0 superscript
 
\newcommand{\R}{\mathbb R}
\newcommand{\Z}{\mathbb Z}
\newcommand{\Q}{\mathbb Q}
\newcommand{\bc}{\mathbb C}

\newcommand{\bq}{\mathbb Q}

\newcommand{\calb}{\mathcal B}

\newcommand{\calh}{\mathcal H}
\newcommand{\calm}{\mathcal M}

\newcommand{\calo}{\mathcal O}
\newcommand{\calp}{\mathcal P}

\newcommand{\calq}{\mathcal Q}

\newcommand{\calz}{\mathcal Z}

\newcommand{\calrp}{{\mathcal R}^+}

\newcommand\Shat{\widehat{\mathcal{S}}}
\newcommand\Htilde{\widetilde{H}}

\def \co{\colon\!} % for maps f \co X \to Y

\newcommand{\vphi}{\varphi}
\DeclareMathOperator{\diffp}{Diff^+}

\DeclareMathOperator{\BSO}{BSO}

\newcommand{\M}{\mathcal M}
\DeclareMathOperator{\met}{Met}

\DeclareMathOperator{\Aut}{Aut}

\newcommand{\sts}{S^2\widetilde\times S^2}

\newcommand{\cs}{\mathbin{\#}}
\newcommand{\cpone}{\bc P^1}

\newcommand{\cptwo}{\bc P^2}

\newcommand{\cptwobar}{\smash{\overline{\bc P}^2}}
\newcommand{\interior}{\textup{int}}

\newcommand{\Spinc}{\operatorname{Spin^c}}

\newcommand{\vdim}{\operatorname{v-dim}}
\newcommand{\spincs}{\mathfrak{s}}

\newcommand{\unred}[1]{ \ignorespaces}  %% to omit that red text
\newcommand{\0}{\textup{0}}%\footnotesize O}}
\newcommand{\sw}{\operatorname{SW}}
\newcommand{\swu}{\underline{\operatorname{SW}}}
\newcommand{\swc}{\operatorname{\mathbb{SW}}} %matching Konno's notation
\newcommand{\fsw}{\operatorname{FSW}}
\newcommand{\SW}{Seiberg-Witten\xspace}

\newcommand{\Bigwedge}{\mathord{\adjustbox{valign=B,totalheight=.6\baselineskip}{$\bigwedge$}}}

\title{Families of diffeomorphisms, embeddings, and positive scalar curvature metrics via Seiberg-Witten theory}
\author[Dave Auckly]{Dave Auckly}
\address{Department of Mathematics\newline\indent Kansas State University\newline\indent  Manhattan,
Kansas 66506}
\email{dav@math.ksu.edu}
\author[Daniel Ruberman]{Daniel Ruberman}
\address{Department of Mathematics, MS 050\newline\indent Brandeis
University \newline\indent Waltham, MA 02454}
\email{ruberman@brandeis.edu}
\thanks{The authors were partially supported by NSF Grant DMS-1928930 while in residence at the Simons Laufer Mathematical Sciences Institute (formerly known as MSRI). The second author thanks Kyoto University and the University of Tokyo for their hospitality while portions of this work were carried out.
The first author was partially supported by Simons Foundation grant 585139 and NSF grant DMS-1952755. The second author was partially supported by NSF grant DMS-1952790.}

\begin{document}
\setlength{\headheight}{12.0pt}.
\begin{abstract}
We construct infinite rank summands isomorphic to $\Z^\infty$ in the higher homotopy and homology groups of the diffeomorphism groups of certain $4$-manifolds. These spherical families become trivial in the homotopy and homology groups of the homeomorphism group; an infinite rank subgroup becomes trivial after a single stabilization by connected sum with $\sss$. The stabilization result gives rise to an inductive construction, starting from non-isotopic but pseudoisotopic diffeomorphisms constructed by the second author in 1998. The spherical families give $\Z^\infty$ summands in the homology of the classifying spaces of specific subgroups of those diffeomorphism groups. 

The non-triviality is shown by computations with family Seiberg-Witten invariants, including a gluing theorem adapted to our inductive construction. As applications, we obtain infinite generation for higher homotopy and homology groups of spaces of embeddings of surfaces and $3$-manifolds in various $4$-manifolds, and for the space of positive scalar curvature metrics on standard PSC $4$-manifolds.  
\end{abstract}
\maketitle
\section{Introduction}
The twin revolutions~\cite{donaldson,freedman} of the early 1980s revealed dramatic differences between the smooth and topological categories in dimension four. Subsequent developments allowed these differences to be quantified; for example the smooth homology cobordism group $\Theta^3_\Z$ is infinitely generated~\cite{furuta:cobordism,fs:instanton} and indeed has an infinite rank summand~\cite{dai-hom-stoffregen-truong:summand} while the topological version is trivial~\cite{freedman}. Similarly, the kernel of the natural map from the smooth concordance group to the topological concordance group has an infinite rank summand~\cite{dai-hom-stoffregen-truong:more,ozs:upsilon}. In this paper, we compare the homotopy and homology groups of the diffeomorphism and homeomorphism groups of $4$-manifolds, and show that there are manifolds so that (in a given degree or range of degrees) the kernels in homotopy and homology have an infinite rank summand. There are similar results for the corresponding homology groups of the classifying spaces ${\rm BG}$ where $G$ is one of a number of subgroups of $\diff(Z)$. (For this paper, we use the notation $\diff(Z)$ for \emph{orientation-preserving} diffeomorphisms, instead of the more typical $\diffp(Z)$.) Using a connection between diffeomorphisms and embeddings (compare~\cite{auckly-ruberman:emb,auckly:internal,lin-mukherjee:surfaces}), we also produce topologically trivial infinite rank summands of the homotopy and homology groups of spaces of embeddings of spheres and $3$-manifolds in related $4$-manifolds, Similarly, we find infinite rank summands in the homotopy and homology groups of spaces of metrics of positive scalar curvature in dimension $4$.

We summarize the general phenomenon by saying that $Z$ admits exotic families of diffeomorphisms or embeddings; families parameterized by a sphere will be called \emph{spherical}. In a companion paper~\cite{auckly-ruberman:diffym} (whose results predate the present work by several years) we construct and detect exotic spherical families using parameterized Yang-Mills gauge theory, as introduced in the first author's paper~\cite{ruberman:isotopy}. In particular~\cite{auckly-ruberman:diffym} gives examples of simply connected $4$-manifolds $Z$ for which $\pi_k(\diffo(Z))$ has arbitrarily high rank free abelian summands that lie in the kernel of the natural map to $\pi_k(\homeo(Z))$. Here the superscript in $\diffo(Z)$ indicates the identity component of $\diff(Z)$, while $\diff(Z,\spincs,\calo)$ denotes diffeomorphisms that preserve a $\Spinc$ structure and orientation $\calo$ of the \SW moduli space, and $\tdiff(Z)$ is the Torelli subgroup consisting of diffeomorphisms that act as the identity on $H_2(Z)$.

 The main result of this paper is that for any $k > 0$ there is a manifold $Z$ for which $\ker[\pi_k(\diffo(Z))  \to \pi_k(\homeo(Z))]$ has $\Z^\infty$ summands that survive in homology.  A number of elaborations and variations of the theorem are discussed below in Remark~\ref{R:main}.  
\begin{theorem}\label{T:infgen}
For any $p > 0$ there are $4$-manifolds $\zz^p$
such that for all $0< j \leq p$ with $j \equiv p \pmod{2}$ the following groups contain $\Z^\infty$ summands:
\begin{enumerate}[topsep=2pt,parsep=2pt,partopsep=1pt]
    \item $\pi_j(\diffo(\zz^p))$ and $\ker\left[\pi_j(\diffo(\zz^p)) \to \pi_j(\homeo(\zz^p))]\right]$\label{homotopy}
    \item $H_j(\diffo(\zz^p))$ and $\ker\left[H_j(\diffo(\zz^p)) \to H_j(\homeo(\zz^p))]\right]$\label{homology}
    \item $H_{j+1}(\btdiff(\zz^p))$ and $\ker\left[H_{j+1}(\btdiff(\zz^p))\to H_{j+1}(\bhome(\zz^p))\right]$.\label{top-triv}
  \end{enumerate}
For $j = p=0$ the groups $\diffo(\zz^p)$ in items (1) and (2) should be replaced by $\diff(\zz^p,\spincs,\calo)$ for a particular choice of $\Spinc$ structure $\spincs$ and homology orientation $\calo$, or indeed by $\tdiff(\zz^p)$. 
The conclusion is that $\pi_0$ and $H_0$ of this group are infinitely generated, and that $H_{1}(\bdiff(\zz^p,\spincs,\calo))$ and $\ker\left[H_{1}(\bdiff(\zz^p,\spincs,\calo))\to H_{1}(\bhome(\zz^p,\spincs,\calo))\right]$ have $\Z^\infty$ summands, with similar statements for $\tdiff(\zz^p)$.
\end{theorem}

Some partial results along these lines were previously known. Torelli group, the second author~\cite{ruberman:polyisotopy} used a parameterized version of Donaldson's polynomial invariants to show infinite generation of $\ker[\pi_0(\tdiff) \to \pi_0(\home(Z))]$. The technique here reproduces that result and extends it to some larger subgroups of $\diff$, in particular to $\diff(Z,\spincs,\calo)$. It is an interesting question whether such results on higher homotopy and homology groups could be obtained using Donaldson theory. 

There are many other recent papers~\cite{kato-konno-nakamura:families,konno-lin:instability,konno-taniguchi:boundary,kronheimer-mrowka:dehntwist,lin:dehn} showing that gauge-theoretic methods reveal a good deal about homotopical properties of the diffeomorphism group of a $4$-manifold.  Closest to the current work, Smirnov~\cite{smirnov:loops} constructed loops of infinite order in $\pi_1(\diff(X))$ for many complex surfaces,
Baraglia~\cite{baraglia:k3} showed that $\pi_1(\diff(K3))$ is infinitely generated,  and Lin~\cite{lin:loops} showed the same for $4$-manifolds containing many embedded $2$-spheres of self-intersection $-1$ or $-2$. The generators of the $\Z^\infty$ subgroup constructed by these authors are non-trivial in $\pi_1$ of the homeomorphism group as they are detected by Pontrjagin classes of the vertical tangent bundle. However, this argument does not show their linear independence.  As pointed out by Lin~\cite{lin:loops}, our results show that the result of Bustamante-Krannich-Kupers~\cite{bustamante-krannich-kupers:finiteness} that the higher homotopy groups $\pi_i(\diffo(M^{2n}))$ of the identity component $\diffo(M)$ of $\diff(M)$ are finitely generated when $2n \geq 6$ fails dramatically in dimension $4$.  In addition to papers of the present authors, there are also many other works~\cite{baraglia:surfaces,drouin:embeddings,konno-mallick-taniguchi:knotted-donaldson,konno-mukherjee-taniguchi:codim1,konno-taniguchi:boundary,lin-mukherjee:surfaces} showing exotic phenomena amongst embeddings of surfaces and $3$-manifolds in $4$-manifolds.

The main new inputs that go into the proof of our main result are a recursive construction of spherical families and the computation of their Seiberg-Witten invariants.  The recursive scheme for creating families of diffeomorphisms parameterized by spheres of arbitrary dimensions starts with exotic diffeomorphisms similar to those constructed in~\cite{ruberman:isotopy,ruberman:polyisotopy} as well as~\cite{baraglia:mcg,konno:moduli}.  First, we take advantage of the fact that certain manifolds, distinguished by Seiberg-Witten invariants, become diffeomorphic under a single stabilization by $\sss$. This is used to create non-trivial diffeomorphisms that 
generate a surjection $\pi_0(\diff(\bZ^0,\spincs,\calo)) \to \Z^\infty$ for a particular manifold $\bZ^0$ and $\Spinc$ structure $\spincs$.  The $1$-stabilization result of~\cite{auckly-kim-melvin-ruberman:isotopy,auckly-kim-melvin-ruberman-schwartz:1-stable} is used repeatedly to recursively create elements in the higher homotopy groups in Theorem~\ref{T:infgen}. We used a similar recursive construction in our work on parameterized Donaldson theory~\cite{auckly-ruberman:diffym}. 

To show the independence of these generators, we make use of integer-valued \SW invariants $\sw^{\pi_k}$ and $\sw^{H_k}$ (related to the characteristic classes of~\cite{konno:classes}) and their mod $2$ versions $\sw^{\pi_k,\Z_2}$ and $\sw^{H_k,\Z_2}$. We compute them via a gluing formula that works well with the inductive construction. (Such a strategy works for the Donaldson-type invariants used in~\cite{auckly-ruberman:diffym}.) The gluing result, the Parameterized Irreducible-Reducible Gluing Theorem~\ref{pIRglue}, is a common generalization of the usual blow-up formula~\cite{nicolaescu:swbook} and the combination wall-crossing/blow-up formulas from~\cite{ruberman:isotopy,ruberman:swpos}; see also~\cite{baraglia-konno:gluing}. The theorem computes the mod $2$ family \SW invariants of a connected sum $N_1 \cs N_2$ (with appropriate $\Spinc$ structures and generic data) in the special case where we are gluing a $k$-dimensional family on $N_1$ with only isolated irreducible solutions to an $\ell$-dimensional family on $N_2$ with only isolated reducible solutions.  

The Seiberg-Witten invariants to which the gluing formula applies are only well-defined mod $2$, which would not seem strong enough to produce  $\Z^\infty$ subgroups or summands.  A fortunate coincidence resolves this dilemma: the $\Z_2$ invariants behave well with respect to compositions, and integer-valued \SW invariants are often defined when one composes two diffeomorphisms that have only $\Z_2$ invariants. This coincidence (Lemma~\ref{L:spheredef}) allows us to conclude that there are non-vanishing integer-valued invariants for certain families. To show the infinite generation claimed in Theorem~\ref{T:infgen}, we use a simple vanishing result (Lemma~\ref{L:finite}) for families corresponding to the fact that any $4$-manifold with $b_2^+ > 1$ has a finite number of basic classes.

An intrinsic aspect of our construction is its behavior under stabilization: the families we construct become trivial after a single connected sum with $\sss$ or $\sts = \cptwo \cs \cptwobar$. (For this connected sum to make sense, we need to make sure that all of our families are the identity on a ball.) Results of Quinn~\cite{quinn:isotopy} (see~\cite{gabai-etal:pseudoisotopies} for a corrected version) show that homotopic diffeomorphisms of simply connected manifolds become isotopic after some number of stabilizations, but this is not at all clear for higher-parameter families, even topologically trivial ones. (In particular, the generators of the $\Z^\infty$ in $\pi_1(\diffo(X))$ constructed in~\cite{baraglia:k3,lin:loops,smirnov:loops} remain topologically non-trivial after one stabilization, but the behavior of the full subgroups under stabilization is not known.) The inductive nature of our construction nonetheless allows us to build a stable contraction of our spherical families. 

Similarly, there are no general theorems that can be used to contract higher-parameter families topologically. By~\cite{kreck:isotopy,quinn:isotopy,gabai-etal:pseudoisotopies} elements in $\tdiff(Z)$ are pseudoisotopic to the identity, and are therefore topologically isotopic to the identity.  As for stabilizations, the inductive construction similarly allows us to pass from topological isotopies for a single diffeomorphism to topological contractions of the spherical families we construct. Our inductive construction also allows us to show in Corollary~\ref{C:pseudo} that our families are homotopically trivial in the \emph{block diffeomorphism group}; this is the higher-parameter version of the pseudoisotopy statement for a single diffeomorphism.

A second focus of the present paper
is to expand the range of manifolds that admit exotic spherical families of diffeomorphisms and PSC metrics.  The manifolds appearing in~\cite{auckly-ruberman:diffym} supporting spherical families of diffeomorphisms are not spin, nor are the manifolds in~\cite{konno:family,ruberman:swpos} with exotic PSC metrics. The methods used here allow us to find spin examples. 

Before turning to corollaries and applications, we make a few remarks on the statement of Theorem~\ref{T:infgen}.
\begin{remark}\label{R:main} 
\leavevmode
\begin{enumerate}
    \item The homology elements in item (2) of the theorem are simply the image of the spherical families constructed to prove item (1). In particular, they are in the image of the Hurewicz map. 
\item Simultaneous papers of Baraglia~\cite{baraglia:mcg} and Konno~\cite{konno:moduli} show that in fact the mapping class group $\pi_0(\diff(X^4))$ is infinitely generated for appropriate simply connected $4$-manifolds $X$. Examples of non-simply connected $4$-manifolds with this property were previously exhibited by Budney-Gabai~\cite{budney-gabai:3-balls} and Watanabe~\cite{watanabe:theta}.
\item The manifolds $\zz^p$ may be chosen to be spin or non-spin.  We may take the non-spin manifolds $\bZ^p$ to be diffeomorphic to a connected sum of copies of $\cptwo$ and $\cptwobar$. In the spin case, $\bZ^p$ may be taken to be diffeomorphic to a connected sum of copies of the $K3$ surface and $\sss$. Hence we may take $\bZ^p$ to admit a metric of positive scalar curvature when it is not spin, or if it is spin with signature zero. 
\item As we will discuss in Remark~\ref{R:def-comments}, all of the results discussed in the introduction apply to a broader class of manifolds. In particular, there are manifolds with boundary or with arbitrary fundamental group whose diffeomorphism groups have the same homotopical properties described in our main theorems.
\end{enumerate}
\end{remark}

Fundamental to the recursive construction in  Theorem~\ref{T:infgen} is the fact that the constructed families become trivial after stabilization. Let us formalize this in a definition.
\begin{definition}\label{D:stabilize}
    A spherical family of diffeomorphisms $\alpha:S^k \to \diff(X,B^4)$ is $n$-stably trivial if $\alpha \cs 1_{\cs^n (\sss)}$ is trivial in $\pi_k(\diff(X\cs^n (\sss)))$. If $\alpha$ is $n$-stably trivial for some $n$, we say it is stably trivial.
\end{definition}
A theorem of Quinn~\cite{quinn:isotopy}, with proof corrected in~\cite{gabai-etal:pseudoisotopies}, states that diffeomorphisms of simply connected $4$-manifolds that act trivially on the second homology are stably trivial. However, to our knowledge, there are no general result proving stable triviality of families with $k \geq 1$. In contrast, we have: 
\begin{corollary}\label{C:stabilize}
    The spherical families in Theorem~\ref{T:infgen} are in the image of $\diff(\bZ^p,B^4)$ and a $\Z^\infty$ subgroup of the claimed summand consists of $1$-stably trivial elements. 
\end{corollary}
%     Both parts follow from the formulas given in Definition~\ref{D:families}. The initial choice of diffeomorphism $\balpha^0$ can be assumed to be the identity on a ball $B^4$. As shown in Proposition~\ref{P:Rstab}, the stable isotopy $F^0$ can be assumed to be the identity on that same $B^4$. Note that a commutator as described in Section~\ref{S:comm} of maps that are the identity on $B^4$ will also be the identity on $B^4$.  By induction, it follows that for any choices of $V$ and $\bu$, the families $\balpha^{p}[q]$ and stable contractions $F^{p+1}[q]$ are the identity on $B^4$. In particular, the $F^{p+1}$s show  any combination of $\balpha^p[q]$ will be stably trivial. 
% \end{proof}
The study of diffeomorphism groups in higher dimensions replaces isotopy by pseudoisotopy (or concordance). The diffeomorphisms showing that $\pi_0(\tdiff(\bZ^0))$ is infinitely generated act trivially on the identity, and so are pseudoisotopic to the identity~\cite{kreck:isotopy,quinn:isotopy}. As observed in~\cite{ruberman:isotopy}, this implies that Cerf's landmark pseudoisotopy theorem~\cite{cerf:stratification} does not hold in dimension $4$.

To study higher-parameter families, one replaces the diffeomorphism group with the \emph{block diffeomorphism group}~\cite{antonelli-burghelea-kahn} $\diffb(Z)$ (using the notation of~\cite{kupers:diffbook}) the geometric realization of the simplicial group whose $n$-simplices are diffeomorphisms $\Delta^n \times Z \to \Delta^n \times Z$. An element $\alpha \in \pi_k(\diff(Z))$ determines a self-diffeomorphism 
\[
\widehat{\alpha}(\theta,x) = (\theta,\alpha(\theta)(x))
\]
of $S^k \times Z$, which in turn is an element of $\pi_k(\diffb(Z))$. This element is trivial if $\widehat{\alpha}$ extends to a self-diffeomorphism of $D^{k+1} \times Z$. The assignment $\alpha \to \widehat{\alpha}$ defines a natural homomorphism $\diff(Z) \to \diffb(Z)$. 

Note that to any $\alpha \in \pi_k(\diff(Z))$, the clutching construction \eqref{E:clutch} builds a bundle $S^{k+1} \times_\alpha Z$. 
\begin{corollary}\label{C:pseudo}
Each $\alpha$ constructed in proving Theorem~\ref{T:infgen}, the corresponding $\widehat\alpha \in \pi_p(\diffb(Z))$ is trivial.
\end{corollary}
% \begin{proof}
%     An extension of $\widehat{\alpha}$ is provided by the diffeomorphism $K^p$ from Definition~\ref{D:families}, where the case $p=0$ is the pseudoisotopy in Equation~\ref{E:FGK0}.
% \end{proof}
As in~\cite{kato-konno-nakamura:families}, this leads to interesting bundles over a sphere.
\begin{corollary}\label{C:bundle}
     For each non-trivial $\alpha$ constructed in proving Theorem~\ref{T:infgen}, the associated bundle $S^{p+1} \times_\alpha Z$ satisfies the following:
     \begin{enumerate}
        \item the bundle is smoothly non-trivial;
        \item the bundle is topologically trivial;
        \item the total space of the bundle is diffeomorphic to a product.  
     \end{enumerate}
\end{corollary}

Using methods of algebraic topology, we can deduce some additional results about the homology of the diffeomorphism group directly from Theorem~\ref{T:infgen}. There are two versions, the first giving infinitely generated torsion subgroups in the homology, and the second giving infinitely generated non-torsion subgroups. Via the main theorem from Browder's study~\cite{browder:torsion} of the algebraic topology of H-spaces, we obtain
\begin{corollary}\label{C:Hinf}
For the manifolds $\zz^p$ with $p$ even appearing in Theorem~\ref{T:infgen}, the homology $H_*(\diffo(\zz^p))$ is infinitely generated. 
\end{corollary}
Browder uses the Hopf algebra structure of $H_*(\diffo(\zz_r^p))$ to construct non-trivial torsion elements in $H_*(\diffo(\zz_r^p))$ in infinitely many degrees, and so the classes he constructs are different from the homology classes that we are detecting. It seems likely that a more strenuous use of the methods from~\cite{browder:torsion} will show that in fact the kernel of the map $H_*(\diffo(\zz^p)) \to H_*(\homeo(\zz^p))$ is infinitely generated.  

A rational version of Corollary~\ref{C:Hinf} was pointed out to us by Jianfeng Lin. He observed that methods of rational homotopy theory~\cite{felix-halperin-thomas:RHT} can be used to deduce the following statement, which is a dramatic demonstration  of the difference between the world of homeomorphisms and diffeomorphisms in dimension four.
\begin{corollary}\label{C:HQinf}
    For the manifolds $\zz^p$ with $p$ even appearing in Theorem~\ref{T:infgen}, 
    \[
    \ker\left[H_*(\diffo(\zz^p);\bq) \to H_*(\homeo(\zz^p);\bq)\right]
    \]
    is infinitely generated in every even degree. 
\end{corollary}

Corollaries~\ref{C:stabilize}-\ref{C:bundle} will be proved in Section~\ref{S:corollaries}; Corollaries~\ref{C:Hinf} and~\ref{C:HQinf} will be proved in Section~\ref{S:algtop}.

\begin{remark}\label{R:highdim}
It is interesting to speculate on the relation between the results of this paper and those obtained in higher dimensions using smoothing and handlebody theory. Fundamental theorems of high dimensional topology (see for example the surveys~\cite{hatcher:50,randal-williams:mit} and responses to~\cite{ruberman:diffMO}) translate problems about the homotopy properties of diffeomorphism and homeomorphism groups into lifting problems relating the topological and smooth tangent bundles. 
The results on $\pi_0(\diff)$ from~\cite{ruberman:isotopy,ruberman:polyisotopy,ruberman:swpos} are in sharp contrast to what one would expect in higher dimensions, but this is less clear for the higher homotopy groups, because the homotopy groups of $PL_4$ and $TOP_4$ (which govern the tangential information) are unknown.

In a different direction, the fact that the homotopy group elements $\alpha$ become trivial after one stabilization suggests a contrast with the remarkable homological stability results of Galatius and Randal-Williams~\cite{galatius-randal-williams:I,galatius-randal-williams:II}. They show for $n>2$ that for a simply-connected $2n$-manifold with $g$ connected summands $S^n \times S^n$, stabilization induces an isomorphism on the $k\textsuperscript{th}$ homology groups of the classifying spaces of  the diffeomorphism groups as long as $2k \leq g -3$. Theorem~\ref{T:infgen}, combined with Proposition~\ref{P:Rstab}, says that this stability does not hold for the integral or rational homology groups of $\bdiffo$ or $\btdiff$ in dimension $4$.  A recent paper of Konno and Lin~\cite{konno-lin:instability} shows that homology stability does not hold for $\bdiff$. This failure of homology stability is demonstrated by finding elements in the kernel and cokernel of the stabilization map; these elements are all $2$-torsion. Hence it would be of interest to show that homology stability for $\bdiff$ fails with rational coefficients.

The correct context for structural questions about the homology of the Torelli group in other dimensions seems to be \emph{representation stability}, as in~\cite{church-farb:stability}, and it is an interesting question as to how our results fit in to this framework.   

\end{remark}
A direct consequence of Theorem~\ref{T:infgen} is that Konno's characteristic classes~\cite{konno:classes} $\swc\in H^*(\btdiff)$ have non-trivial evaluations on the families we construct. This answers, in the case of the Torelli group, a question raised by Cushing-Moore-Ro\v{c}ek-Sakena in Section 9 of~\cite{cushing-moore-rocek-sakena:diff}. They ask if the generalized Miller-Morita-Mumford classes (or tautological classes, or MMM classes) ~\cite{miller:mcg,morita:characteristic,mumford:classes} in $H^*(\bdiff(X);\Q)$ generate all of the cohomology, as holds in dimension $2$ by the solution to Mumford's conjecture~\cite{madsen-weiss:mumford}. It is natural to ask if this holds true for the pullback of the MMM classes to $H^*(\btdiff(X^n);\Q)$ induced by the inclusion $\tdiff \to \diff$. 
\begin{corollary}\label{C:MMM}
The MMM classes do not generate all of $H^*(\btdiff(X^n);\Q)$. 
\end{corollary}

\vspace*{1ex}\noindent

As in~\cite{auckly-ruberman:diffym,auckly-ruberman:emb} the construction of our exotic families gives rise to families of embeddings of surfaces and $3$-manifolds. Let us write $\Emb(\Sigma,X)$ and $\Embtop(\Sigma,X)$ for the space of smooth embeddings of a manifold $\Sigma$ in $X$ and locally flat topological embeddings, respectively. Note that there is a forgetful map $\Emb(\Sigma,X) \to \Embtop(\Sigma,X)$. If we want to keep track of the homology class $A$ represented by $\Sigma$, we will write $\Emb_A(\Sigma,X)$. There are several different notions of stabilization for embedded surfaces and families of embedded surfaces. External stabilization may be summarized by $(X,\Sigma)\cs (\sss,\emptyset)$. 
This operation makes sense on the family level for families that agree with a fixed embedding of a disk in $D^4\subset X$.

The next two theorems describe large families of embedded 2 or 3-spheres in manifolds $\bzptilde$. These are simply stabilizations of the manifolds $\zz^p$ appearing in Theorem~\ref{T:infgen}. 
\begin{theorem}\label{T:emb2} 
For any $p \geq 0$ and for all $j \leq p$ with $j \equiv p \pmod{2}$ there are manifolds $\bzptilde$ such that the following groups have $\Z^\infty$ summands:
\begin{enumerate}[topsep=2pt,parsep=2pt,partopsep=1pt]
    \item $\pi_j(\Emb(S^2,\bzptilde))$ and  $\ker\left[\pi_j(\Emb(S^2,\bzptilde))\to \pi_j(\Embtop(S^2,\bzptilde))\right]$
      \item $H_j(\Emb(S^2,\bzptilde))$ and  $\ker\left[H_j(\Emb(S^2,\bzptilde))\to H_j(\Embtop(S^2,\bzptilde))\right]$
   \end{enumerate}
  The $2$-spheres in question can have self-intersection $0$ or $\pm 1$, and a $\Z^\infty$ subgroup becomes trivial after a single external stabilization. The homology class of the sphere can be any primitive and non-characteristic element of $H_2(\bzptilde)$.
\end{theorem}
In his 2023 PhD thesis\cite{drouin:embeddings} Josh Drouin has extended Theorem~\ref{T:emb2} to $2$-spheres of arbitrary self-intersection.\\[1ex]

By taking the boundary of a tubular neighborhood of one of the spheres in Theorem~\ref{T:emb2}, we get a similar result for embeddings of certain $3$-manifolds in the same manifolds $\bztilde$. 
\begin{theorem}\label{T:emb3} 
For any $p \geq 0$ and for all $j \leq p$ with $j \equiv p \pmod{2}$ the following groups have $\Z^\infty$ summands:
\begin{enumerate}[topsep=2pt,parsep=2pt,partopsep=1pt]
    \item $\pi_j(\Emb(S^3,\bzptilde))$ and  $\ker\left[\pi_j(\Emb(S^3,\bzptilde))\to \pi_j(\Embtop(S^3,\bzptilde))\right]$
      \item $H_j(\Emb(S^3,\bzptilde))$ and $\ker\left[H_j(\Emb(S^3,\bzptilde))\to H_j(\Embtop(S^3,\bzptilde))\right]$.
    \item $\pi_j(\Emb(\SSS,\bzptilde))$ and  $\ker\left[\pi_j(\Emb(\SSS,\bzptilde))\to \pi_j(\Embtop(\SSS,\bzptilde))\right]$
      \item $H_j(\Emb(\SSS,\bzptilde))$ and $\ker\left[H_j(\Emb(\SSS,\bzptilde))\to H_j(\Embtop(\SSS,\bzptilde))\right]$.
  \end{enumerate}
  In each case, a $\Z^\infty$ subgroup becomes trivial after a single external stabilization.
\end{theorem}

The notion of concordance extends to families of embeddings. Two families of embeddings $J, J'\co S^k\to \Emb(\Sigma,X)$ are concordant if there is a smooth embedding $K\co I\times S^k \times \Sigma \to I\times S^k\times X$ acting trivially on the $S^k$ factor, and satisfying $K(0,\theta,v) = J(\theta)(v)$, and $K(1,\theta,v) = J'(\theta)(v)$, 

\begin{corollary}\label{C:embK}
 The generators of the subgroups from Theorem~\ref{T:emb2} (or Theorem~\ref{T:emb3} that become trivial after one external stabilization are all concordant.    
\end{corollary}

Our results on diffeomorphism groups can be applied, as in~\cite{ruberman:swpos,konno:family} to show that the space $\calrp(Z)$ of Riemannian
metrics with positive scalar curvature (PSC) can have non-trivial homotopy and homology groups in a range of dimensions. The idea is to take advantage of the action of the $\diff(Z)$ on the space of all Riemannian metrics, which preserves the subspace $\calrp(Z)$. We show, using the Weitzenb\"ock formula, that the non-trivial families of diffeomorphisms detected by \SW invariant give rise to infinitely generated homotopy and homology groups of $\calrp(Z)$. This method, by construction, does not shed any light on the higher homotopy and homology groups of the \emph{moduli space} of PSC metrics, $\calrp(Z)/\diff(Z)$ (or the homotopy quotient $\calrp(Z)//\diff(Z)$)  an interesting object of study in its own right. Botvinnik and Watanabe~\cite{botvinnik-watanabe:families} have detected non-trivial elements in the rational homotopy groups of the PSC moduli space using rather different methods.
\begin{theorem}\label{T:psc} 
For any $p \geq 0$
%\smargin{Should also have spin examples using sums of $\sss$}
and for all $j \leq p$ with $j \equiv p \pmod{2}$  there are non-spin manifolds $\zzn^p$ and spin manifolds  $\zzs^p$ for which the following groups have $\Z^\infty$ summands:
\begin{enumerate}[topsep=2pt,parsep=2pt,partopsep=1pt]
    \item $\pi_j(\calrp(\zzn^p))$ \text{\ and \ }  $\pi_j(\calrp(\zzs^p))$
      \item $H_j(\calrp(\zzn^p))$ \text{\ and \ }  $H_j(\calrp(\zzs^p))$.
  \end{enumerate}
\end{theorem}
The manifolds $\bZ^p_{NS}$ are all diffeomorphic to connected sums of copies of $\sss$, and the manifolds $\bZ^p{_S}$ are diffeomorphic to sums of copies of $\cptwo$ and $\cptwobar$. It is known that all of these admit metrics of positive \emph{Ricci} curvature~\cite{sha-yang:ricci-spheres,sha-yang:ricci-4,perelman:positive-ricci}.  Since a metric with positive Ricci curvature automatically has positive scalar curvature, we immediately see that Theorem~\ref{T:psc} applies to $\mathcal{R}^{\Ric >0}$, the space of Riemannian metrics of positive Ricci curvature.
\begin{corollary}\label{C:prc} 
For any $p \geq 0$
%\smargin{Should also have spin examples using sums of $\sss$}
and for all $j \leq p$ with $j \equiv p \pmod{2}$  there are non-spin manifolds $\zzn^p$ and spin manifolds  $\zzs^p$ for which the following groups have $\Z^\infty$ summands:
\begin{enumerate}[topsep=2pt,parsep=2pt,partopsep=1pt]
    \item $\pi_j(\mathcal{R}^{\Ric >0}(\zzn^p))$ \text{\ and \ }  $\pi_j(\mathcal{R}^{\Ric >0}(\zzs^p))$
      \item $H_j(\mathcal{R}^{\Ric >0}(\zzn^p))$ \text{\ and \ }  $H_j(\mathcal{R}^{\Ric >0}(\zzs^p))$.
  \end{enumerate}
\end{corollary}
\begin{remark}
    As for our results on diffeomorphisms, Theorem~\ref{T:psc} shows a large difference with the situation in higher dimensions. The $\Z^\infty$ summands in the homology and homotopy groups of $\calrp(\bZ^p)$ are (by construction) in the image of the homomorphism induced by the natural evaluation map $\diff(\bZ^p) \to \calrp(\bZ^p)$. This would apply to connected sums of sufficiently many copies of $\sss$ or $\cptwo\cs\cptwobar$, both of which have trivial first Pontryagin class. On the other hand, work of Ebert and Randal-Williams~\cite[Theorem F]{ebert-randal-williams:psc-category} says that for $2$-connected manifolds of dimension $d \geq 6$ with vanishing Pontryagin classes, the image of this map on homotopy groups is \emph{finite}. The case of connected sums of $S^n \times S^n$ can (per~\cite{ebert-randal-williams:psc-category}) be deduced from~\cite{botvinnik-ebert-randal-williams:psc-loop}.
\end{remark}
Most of our constructions will be made with simply connected closed manifolds. We record in the following theorem that all of our results hold for more general classes of $4$-manifolds.
\begin{theorem}\label{T:stab-all}
For any finitely-presented group, $\Gamma$, there are manifolds $\bZ^p_\Gamma$ with fundamental group $\Gamma$ satisfying Theorem~\ref{T:infgen},
Corollary~\ref{C:Hinf},
Corollary~\ref{C:HQinf}, Theorem~\ref{T:emb2}, and Theorem~\ref{T:emb3}. The manifolds $\bZ^p_\Gamma$ may be taken to be closed, or to have non-empty boundary, in which case the families of diffeomorphisms will be trivial on the boundary. Moreover,  Theorems~\ref{T:emb2} and Theorem~\ref{T:emb3} extend to give exotic embedded surfaces and 3-manifolds with boundary in $\bZ^p_\Gamma$.
Given any closed, simply-connected $4$-manifold, $X$, one may stabilize it sufficiently many times to arrive at a manifold $\bZ^p_X$ satisfying the same theorems and corollaries.
\end{theorem}
The extensions of Theorem~\ref{T:infgen} and 
Corollary~\ref{C:Hinf} are described in item \eqref{i:variations} of Remark~\ref{R:def-comments}. 
\\[1ex]
{\bf Conventions:} All $4$-manifolds considered in this paper will be simply connected (unless explicitly stated otherwise) and oriented. In Section~\ref{glue-proof}, we will consider manifolds with cylindrical ends~\cite{taubes:l2}; such manifolds (and objects that live on them) will be decorated with a `\, $\hat{}$\ '. 
The Poincar\'e dual of a homology class $A$ will be denote by $\check{A}$. We will denote by $b^+(Q)$ the dimension of a maximal positive definite subspace of a symmetric bilinear form. We will sometimes use brackets to enumerate a set of maps or spaces  and superscripts to indicate a dimension. In particular,  $\balpha^j[q]$ would be the $q^{th}$ member of a collection of maps defined on the $j$-sphere, and $F^j[q]$, $G^j[q]$, and $K^j[q]$ would be $1$-parameter families of such maps. Finally, we find it convenient to discuss \SW invariants either as a function of $\Spinc$ structures denoted by $\sw(X,\spincs)$, or as a function on characteristic cohomology classes defined as 
\[
\sw(X,K) = \sum_{\substack{\spincs \in \Spinc(X), \\
c_1(\spincs) = K}} \sw(X,\spincs).
\]
These notions are equivalent if there is no $2$-torsion in $H^2(X)$ as is mostly the case in this paper.

\textbf{Acknowledgments:}
Conversations with Hokuto Konno were very helpful in the development of this work. We thank Jianfeng Lin for numerous comments, including the statement and references for Lemma~\ref{L:jianfeng}, and Nikolai Saveliev for supplying a useful reference. Conversations and email exchanges with Alexander Kupers and Oscar Randal-Williams helped us to understand the relation of our work with recent work in high dimensions. We also thank Richard Bamler for an interesting exchange that pointed us to Corollary~\ref{C:prc}.

Many of the results in this paper were developed when the authors were in residence at the Fall 2022 SLMath program, \emph{Analytic and Geometric Aspects of Gauge Theory}. We thank SLMath and the program organizers for providing a stimulating environment.

\tableofcontents

\section{Invariants of families}\label{S:inv}
The non-trivial families discussed in this paper will be detected using parameterized Seiberg-Witten 
theory, as suggested in~\cite{donaldson:durham,donaldson:swsurvey} and carried out in~\cite{ruberman:swpos,konno:classes}; compare also~\cite{szymik:cohomotopy-families,li-liu:family,xu:thesis}. There are several different variations corresponding to the three conclusions in Theorem~\ref{T:infgen}, although the underlying idea is similar in each case.  We start the discussion with a brief review of  \SW invariants in the family setting.

\subsection{Family moduli spaces}\label{S:family}
Recall that the Seiberg-Witten equations associated to a $\Spinc$ structure $\spincs$ on $X$ depend on a Riemannian metric $g$ on $X$ and a closed  $2$-form $\eta$.  The \SW moduli space $\calm_{X,\spincs}(g,\eta)$ is defined as the set of gauge equivalence classes of solutions  $[A,\psi]$ to the Seiberg-Witten equations 
\begin{align}
% {\bf \cc}(F_A^+ + i\eta^+) - \frac12(\psi^*\psi)_0&=0 \label{E:SWcurv}\\
\sqrt{2}(F^+_{A}  - \frac{1}{2}{\bf c}^{-1}(q(\psi) ))&= -i\sqrt{2}\eta^+\label{E:SWcurv}\\
D_A^+\psi
 & = 0. \label{E:SWdirac}
\end{align}
where $\cc$ denotes the isomorphism $i \Lambda^2_+\to \mathfrak{su}(S^+)$ induced by Clifford multiplication. Here the spin bundles are denoted $S^\pm$ and the `$+$' superscript on $2$-forms indicates the projection to the self-dual forms. For convenience, we are following the conventions in \cite{nicolaescu:swbook}.
% The left-hand side of these equations defines the \SW map
% $ \widehat{SW}_\data(\theta,A,\psi) $

The equations depend on $g$ through the definition of the Dirac operator $D_A^+$ and also the projection to self-dual forms; the dependence on $\eta$ is clear. For fixed $(g,\eta)$, we can view the left-hand side of these equations as defining a map 
\begin{equation}\label{E:swmap}
   \sw\colon \calb_{X,\spincs} \to \Omega^2_+(X;i \R) \oplus \Gamma(S^-)
\end{equation}
from the gauge equivalence classes of connections and spinors to imaginary self-dual $2$-forms and spinors. 

In addition to families of closed manifolds, we need to consider families of manifolds with a cylindrical end metrically modeled on $[0,\infty)\times S^3$. We will use $\hat X$ to denote the manifold obtained from $X$ by deleting a point together with a metrically cylindrical end. A comprehensive description of the cylindrical end setting for Seiberg-Witten theory may be found in \cite{nicolaescu:swbook}. There are only a few small differences between the $S^3$ cylindrical end case and the closed case. 

\begin{remark}\label{R:sobolev}
    An important point in the underlying analysis is that the \SW equations are defined on spaces of Sobolev connections and spinors, and one uses Sobolev gauge transformations. As this is standard, we will not dwell on this point, but record here our convention that connections and positive spinors should be of class $L^2_k$, so that curvature, negative spinors, and perturbations would then be of class $L^2_{k-1}$, and gauge transformations of class $L^2_{k+1}$. Here $k \geq 2$ is an integer. In the $S^3$-cylindrical end case one must add weights and work with the Sobolev spaces $L^2_{k,\delta}$ for a sufficiently small fixed $\delta$. The fact that a suitable $\delta$ may be chosen follows as in the non-parameterized case, provided the parameter space is compact. The cylindrical case will also require various extended Sobolev spaces. Details are given in Section~\ref{an-frame}, where we deal with some analytical aspects of our work. Until then, we will not indicate the Sobolev completions in our notation.
\end{remark}
Assume that $b_2^+(X) > 0$. For a generic pair $(g,\eta)$ the moduli space $\calm_{X,\spincs}(g,\eta)$ will have no reducible solutions (those with $\psi=0$) and be a smooth manifold of dimension equal to the virtual dimension 
\begin{equation}\label{E:dim}
\vdim(\calm_{X,\spincs}(g,\eta)) = \frac{c_1(\spincs)^2-\sigma(X)}{4} - (1 -b_1(X) + b_2^+(X)).   
\end{equation}
We refer to such a pair $(g,\eta)$ as \emph{good}.
% \danny{we write $\vdim$ elsewhere. use that instead? }
% \danny{We define good at least 3 times.}
The moduli space is oriented by a choice of orientation $\calo$ for the vector space $H^2_+(X)$. We follow the conventions in~\cite[\S2.2.4]{nicolaescu:swbook}.

The space  of data 
is simply the contractible space
\[
\Pi(X) =  \met(X)\times \ker\left(d:\Omega^2(X)\to \Omega^3(X)\right).
\]
 We   will often abbreviate $(g,\eta)$ to $\data$ and refer to $\data$ as \emph{data}. In the cylindrical end case, one replaces $\Omega^2(X)$ by forms with compact support, denoted $\Omega^2_0(\hat X)$. In the parameterized case one needs to consider functions from the parameter space into the space of data.
For  a map $\data \colon \Xi \to \Pi$, we write $\data(\theta) = \data_\theta = (g,\eta)_\theta$ and define the parameterized moduli space 
\begin{equation}\label{E:msw}
    \calm_{X,\spincs}(\{\data_\theta\}_{\theta\in \Xi}) = \{(\theta,[A,\psi]) \in \Xi\times \bigcup_{\theta\in \Xi}\calm_{X,\spincs}(\data_\theta)\,|\, [A,\psi]\in  \calm_{X,\spincs}(\data_\theta)\}.
\end{equation}
It is more compact to write this as $\calm_{X,\spincs}(\{\data\})$ but the other notation can be helpful in some arguments; we'll pass between the two notations without comment. We use braces around the data to denote the parameterized moduli space and $\calm_{X,\spincs}(\data)$ or $\calm_{X,\spincs}(\data_\theta)$ to denote the unparameterized moduli space associated with a specific data point. Note that projection onto the first factor in \eqref{E:msw} gives a map $p_\Xi\colon \calm_{X,\spincs}\left(\{\data_\theta\}_{\theta\in \Xi}\right) \to \Xi$. There is no reason to expect this to be a fibration; consider the setting when the formal dimension $\vdim(\calm_{X,\spincs}(\data_\theta))$ is negative but $\vdim(\calm_{X,\spincs}(\{\data_\theta\})) \geq 0$. In this situation, the preimage of a generic point $\theta \in \Xi$ will be empty, but the preimage over some points may be non-empty. 
\begin{definition}\label{D:exceptional}
    A point $\theta \in \Xi$ is \emph{exceptional} if $p_\Xi^{-1}(\theta)$ is non-empty. 
\end{definition}
Although logically this makes sense for any virtual dimension, we will only use this terminology when $\vdim(\calm_{X,\spincs}(\data_\theta))$ is negative.

%\[\calm_{X,\spincs}(\{(g,\eta)_\theta\}_{\theta\in \Xi}) = \{(\theta,[A,\psi]) \in \Xi\times \bigcup_{\theta\in \Xi}\calm_{X,\spincs}((g,\eta)_\theta)\,|\, [A,\psi]\in  \calm_{X,\spincs}((g,\eta)_\theta)\}.\]

\subsection{Subgroups of \texorpdfstring{$\diff(X)$}{Diff(X)}}\label{s:groups}
%The usual setup for the family Seiberg-Witten equations %for a bundle $p: E \to B$ with fiber a $4$-manifold $X$ involves a restriction on the structure group. This is discussed carefully in Konno's paper~\cite{konno:classes} (see also~\cite{szymik:cohomotopy-families}) and we follow his treatment and notation. 
%\dannyin{Not sure if we need the full generality from Konno. He works with $\diff(X,\spincs)$ and $\Aut(X,\spincs)$ (its extension by the gauge group) as well as $\diff(X,\spincs,\calo)$ and its $\Aut(X,\spincs,\calo)$ }
Some of the invariants we use require a restriction on which diffeomorphisms are being discussed. The most important of these has to do with the orientation of the moduli space, which influences whether our invariants are defined as integers or only mod $2$.  We will 
assume that $b^1(X) =0$ so that the moduli space $\calm$ is oriented by a choice of orientation $\calo$ for the vector space 
$H^2_+(X;\R)$. We denote by $\diff(X,\spincs)$ the group of diffeomorphisms $f\colon X \to X$ with $f^*\spincs \cong \spincs$, and 
$\diff(X,\spincs,\calo)$ those for which in addition $f^*\calo = \calo$.  Note that both of these groups  contain $\tdiff(X)$, the \emph{Torelli group of $X$}, consisting of diffeomorphisms inducing the identity map on homology, which is in turn contained in the identity component $\diffo(X)$ of $\diff(X)$.

\begin{definition}\label{D:Shat}
    $\Shat^k_X$ is the set of $\Spinc$ structures $\spincs$ on $X$  for which the formal dimension of the \SW moduli space $\calm_X(\spincs)$ is $-(k+1)$.  
\end{definition}
For any $\spincs \in \Shat^k_X$, we will define homomorphisms denoted, respectively, by $\sw^{\pi_k}_{X}(\text{--},\spincs)$ or $\sw^{H_k}_{X}(\text{--},\spincs)$ from the homotopy or homology groups of $G = \diff(X,\spincs)$, $\diff(X,\spincs,\calo)$, or $\diffo(X)$ to $\Z$ or $\Z_2$. When we want to emphasize that the invariant is only defined mod $2$, we will add a further embellishment and write  $\sw^{\pi_k,\Z_2}_{X}(\text{--},\spincs)$.
We will also make use of work of Konno to define similar homomorphisms $\fsw^{\Z}$ and $\fsw^{\Z_2}$ on the homology of the classifying space $BG$. 

All of these homomorphisms are connected in various ways, as we will explain in due course. For instance, the homomorphisms defined on the homology and homotopy groups are, essentially by definition, related by the Hurewicz map. The most interesting relation, described in Proposition~\ref{P:eval}, is between the invariants of $\pi_k(\diffo)$ and $H_{k+1}(\bdiffo)$.  All of the results about infinite generation are gotten by summing over $\Spinc$ structures, to get homomorphisms
\begin{equation}\label{E:swtot}
    \swu^{\pi_k}: pi_k(Diff^0) \to \bigoplus_{\hat{\mathcal{S}}_X^k)} \Z = \Z^\infty
\end{equation}
for $k>0$ with similar (possibly $\Z_2^\infty$-valued) constructions and notations for $k=0$ and for the homological invariants as well.

There is a more general construction~\cite{li-liu:family,konno:classes} of a family moduli space associated to a bundle $X \to E \to \Xi$ where the structure group is a suitable subgroup of $\diff(X)$.  We will mainly be concerned with families supported on trivial bundles, but the basic arguments will apply in this more general setting.

\subsection{Regularity of family moduli spaces}\label{S:regularity}
The standard transversality results~\cite{kronheimer-mrowka:monopole,morgan:swbook,nicolaescu:swbook} for the \SW equations state that if $b_2^+(X) >0$, then for generic data $\data \in \Pi$, the moduli space $\calm_{X,\spincs}(\data)$ contains no reducible solutions and is smooth of dimension equal to the virtual dimension $\vdim(\calm_{X,\spincs}(\data))$ defined in Equation~\ref{E:dim}.
The analogous result for family moduli spaces (even with $S^3$-cylindrical ends also holds~\cite{li-liu:family,konno:classes,baraglia-konno:gluing} by the same proof. In the $S^3$-cylindrical end case one may first consider perturbations in $L^2_{k-1,\delta}$ and then uses the fact that for compact parameter spaces the space of generic maps into the space of data is open to conclude that one may pick the perturbations to have compact support. We state this as a proposition for convenience; one version of the argument is outlined as part of the proof of Proposition~\ref{SPNMID}. 
\begin{proposition}\label{P:generic}
    Suppose that the parameter space $\Xi$ is a manifold and that $b_2^+(X) > \dim(\Xi)$. Then for a generic map $\data \colon \Xi \to \Pi$ the family moduli space
    $\calm_{X,\spincs}(\data)$ contains no reducible solutions and is a smooth manifold of dimension
    \begin{equation}\label{E:Fdim}
\dim(\Xi) + \frac{c_1(\spincs)^2-\sigma(X)}{4}- (1 -b_1(X) + b_2^+(X)).   
\end{equation}
\end{proposition}
Maps satisfying the conclusion of the Proposition will be called \emph{good}. A slight extension of Proposition~\ref{P:generic} says that if $\Xi$ is a manifold with boundary and $\data \colon \partial\, \Xi \to \Pi$ is good, then a generic extension of $\data$ to $\Xi$
will be good. 

\begin{remark}\label{R:family-orient}
    When the structure group is contained in $\diff(X,\spincs,\calo)$ and $\data$ is a good map defined on an oriented manifold $\Xi$, then $\calm_{X,\spincs}(\data)$ is oriented by the fiber orientation. We adopt the convention that in a smooth fiber bundle $E \to B$, the orientation of the total space $E$ is given by the orientation of $B$ followed by an orientation of the vertical tangent bundle. This convention applies as well to the index bundle. In particular, a family \SW moduli space is oriented by the orientation of the base space $\Xi$ followed by the orientation of the index bundle for the deformation complex of the fiber. 
\end{remark}

\subsection{Invariants of homotopy groups and families}\label{S:homotopy}
The second author's papers~\cite{ruberman:isotopy,ruberman:swpos} used $1$-parameter \SW moduli spaces to define an isotopy invariant for diffeomorphisms of $4$-manifolds, essentially the invariant $\sw^{\pi_0}$ discussed below. The authors jointly observed shortly afterwards that a similar construction would give invariants for higher homotopy groups, although it has taken many years to find examples where these are non-trivial. We give the definition, and then explain how these can be interpreted as an evaluation of characteristic classes on $\bdiff(X)$ defined by Konno~\cite{konno:classes}.  Invariants for higher \emph{homology} groups of $\diff(X)$ require some further constructions, which will be given in Section~\ref{s:hom} below.

%For computational purposes, it is convenient to use an approach closer to the original definition of~\cite{ruberman:isotopy} of an invariant of a single diffeomorphism (thought of as an element of $\pi_0(\diff(X))$). 
\subsubsection{Definition of the homotopy group invariant \texorpdfstring{$\sw^{\pi_k}$}{SWk}}\label{s:swpik} In discussing the homotopy groups $\pi_k(\diff(X))$ for $k>0$, we will implicitly work in the identity component $\diffo(X)$ and use the identity element as a basepoint.  Since the action of $\pi_1(\diff(X))$ on $\pi_k(\diff(X))$ is trivial, the choice of basepoint is not important when $k>0$. For $k=0$, we still use the identity element as basepoint, with the following convention. A diffeomorphism $\alpha\colon X \to X$ naturally determines an element of $\pi_0(\diff(X))$ whose value at $1\in S^0$ is the identity and at $-1$ is $\alpha$.  We will make no notational distinction between the diffeomorphism and its associated $\pi_0$ element.

Given $\data = (g,\eta) \in \Pi$, a map $\alpha\colon \Xi \to \diff(X)$
induces a map $\alpha^*$ from $\Xi$ to $\Pi$ by $\alpha^*(\theta) = (\alpha(\theta)^*g,\alpha(\theta)^*\eta) = \alpha(\theta)^*\data$. 
% Note that if $\eta$ is $g$-self-dual, then $\alpha(\theta)^*\eta$ is $\alpha(\theta)^*g$-self-dual, so this makes sense. 
It is convenient to use this construction in a more general setting where $\data$ varies as well.
\begin{definition}\label{D:pullback}
    For maps $\data\colon \Xi \to \Pi$ and  $\alpha:\Xi \to \diff(X)$, define the pullback $\alpha^*\data \colon \Xi \to \Pi$ by
    \[
    (\alpha^*\data)(\theta) = \alpha(\theta)^*\data(\theta).
    \]
\end{definition}
\noindent
In particular, we can define the parameterized moduli space $\calm_{X,\spincs}(\alpha^*\data)$ for a single $\data\in \Pi$ or more generally  $\calm_{X,\spincs}(\{\alpha^*\data)\}$ for a map $\data$. A key observation is that the Seiberg-Witten moduli spaces satisfy a functorial property, as long as all of the data is pulled back.
\begin{lemma}\label{L:pullback}
    Let $\data\colon \Xi \to \Pi$ define a family moduli space $\calm_{X,\spincs,\calo}(\{\data\})$, and let $\alpha\colon \Xi \to \diff(X)$. Then $\alpha$ induces a canonical isomorphism 
    \[
    \alpha^*\colon \calm_{X,\spincs,\calo}(\{\data\}) \to \calm_{X,\alpha^*\spincs,\alpha^*\calo}(\{\alpha^*\data\}).
    \]
    If $\calm_{X,\spincs,\calo}(\{\data\})$ is smooth, then $\alpha^*$ is orientation preserving or reversing according to whether $\alpha^*\calo$ agrees with $\calo$. 
\end{lemma}
The isomorphism is given by pulling back $\Spinc$ connections and section of the spinor bundle; it is easy to see that this is well-defined. Isomorphism here means that not only is the pullback map a homeomorphism, but it induces an isomorphism on the Zariski tangent spaces. In particular, pullback preserves the property of being good. The behavior of $\alpha^*$ with respect to orientations is determined by the effect of $\alpha(\theta)^*$ for any $\theta \in \Xi$ and so is the same as for the  fiberwise \SW invariant. The transformation rule for the usual \SW invariant summarized in~\cite[Theorem 2.3.5]{nicolaescu:swbook}  and yields the statement above.

Now suppose that $\data_0\colon S^k \to \Pi$ is good,  $b^2_+(X) > k+2$, and that the virtual dimension of the Seiberg-Witten moduli space $\calm_{X,\spincs}$ is $-(k+1)$. The assumption on the dimension means that $\calm_{X,\spincs}(\{\data_0\})$ is empty.  Let $\alpha\colon S^k \to \diff(X)$. Since $\Pi$ is contractible, there is a map 
\[\data \colon I \times S^k \to \Pi\  \text{with}\ \data\vert_{\{0\} \times S^k} = \data_0\  \text{and} \ \data\vert_{\{1\} \times S^k} =\alpha^*\data^0.
\]
We can assume that $\data$ is good, which implies that  
\[
\calm_{X, \spincs}(\hat\alpha) = \calm_{X, \spincs}(\{\data_{\hat\alpha(t,\theta)}\}_{(t,\theta)\in I \times S^k}) 
\]
is a compact $0$-dimensional manifold with no elements for which $t= 0$ or $1$.

\begin{definition}\label{D:SW invariant}
$\sw_{X}^{\pi_k}(\alpha, \spincs)$ is the (signed) count of points in $\calm_{X, \spincs}(\hat\alpha)$. For $n=0$, if $\alpha \in \pi_0(\diff(X,\spincs))$ we write $\sw_{X, \spincs}^{\pi_0,\Z_2}(\alpha) \in\Z_2$ or $\sw_{X}^{\pi_0}(\alpha, \spincs,\calo)\in \Z$ depending on whether $\alpha$ preserves the homology orientation $\calo$. 
\end{definition}

We need to check that $\sw_{X}^{\pi_k}(\alpha, \spincs)$ and $\sw_{X, \spincs}^{\pi_0,\Z_2}(\alpha)$ are independent of all choices, namely the initial good data $\data^0$, the choice of homotopy $\data$, and the homotopy class of $\alpha$. The proof of each of these is essentially the same, and follows the proof in~\cite{ruberman:isotopy} with more parameters. We go through the first one in detail to set the stage for later arguments.  
% An important observation is that pullback induces a canonical orientation-preserving isomorphism $\calm_{\spincs,\calo}(X,\data)\to \calm_{f^*\spincs,f^*\calo}(X,f^*\data)$. 
\begin{lemma}\label{L:spheredef}
Let $\alpha\colon S^k\to \diffo(X)$ be smooth, and let $\{\data^0_\theta\}_{\theta\in S^k}$ and $\{\data^1_\theta\}_{\theta\in S^k}$ be good. Let
$\{\data^j_{t,\theta}\}_{(t,\theta)\in I\times S^k}$ for $j=0,1$ be good families in $\Pi$ so that  $\data^j_{1,\theta}=\alpha_\theta^*\data^j_{0,\theta}$. Then
\[
 \#\M(\{\data^0_{t,\theta}\}_{(t,\theta)\in I\times S^k}) = \#\M(\{\data^1_{t,\theta}\}_{(t,\theta)\in I\times S^k}).
\]
The resulting counts $\sw_X^{\pi_k}$ is a homomorphism from 
\begin{itemize}
    \item $\pi_k(\diffo(X))$ to $\Z$ when $k>0$,
    \item $\pi_0(\diff(X,\spincs,\calo))$ to $\Z$,
\end{itemize} 
and $\sw_{X, \spincs}^{\pi_0,\Z_2}$ is a homomorphism from $\pi_0(\diff(X,\spincs))$ to $\Z_2$. When the $\Z$-valued $\sw_X^{\pi_0}$ is defined, its mod $2$ reduction is $\sw_{X, \spincs}^{\pi_0,\Z_2}$.
\end{lemma}
\begin{proof}
Contractibility of  $\Pi$ implies that there is a family of data $\{\data^s_{0,\theta}\}_{(s,0,\theta)\in I\times \{0\}\times S^k}$
extending $\{\data^1_\theta\}_{\theta\in S^k}$  and  $\{\data^0_\theta\}_{\theta\in S^k}$; the generic perturbations Proposition~\ref{P:generic} says that we can assume that this family is good. Write $\{\data^1_{0,\theta}\} = \alpha(\theta)^*\{\data^s_{0,\theta}\}$.  These families combine to give a family
$\{\data^s_{t,\theta}\}_{(s,t,\theta)\in \partial(I^2\times S^k)}$ with $\data^s_{0,\theta} = \data^s_{\theta}$ and $\data^s_{1,\theta} = \alpha_\theta^*\data^s_{\theta}$ that is good by Lemma~\ref{L:pullback}.

Using the fact that the space $\Pi$ is contractible together with Proposition~\ref{P:generic} once again, we construct a good family of data $\{\data^s_{t,\theta}\}_{(s,t,\theta)\in I^2\times S^k}$ with $\data^s_{0,\theta} = \data^s_{\theta}$ and $\data^s_{1,\theta} = \alpha_\theta^*\data^s_{\theta}$.

In the first and last cases, $\alpha_\theta$ preserves orientation, homology orientation and $\spincs$, so Lemma~\ref{L:pullback} gives an orientation-preserving isomorphism between  $\M(\{\data^s_{0,\theta}\}_{(s,\theta)\in I\times S^k})$ and $\M(\{\data^s_{1,\theta}\}_{(s,\theta)\in I\times S^k})$. These have opposite orientations induced from the family $\{\data^s_{t,\theta}\}_{(s,t,\theta)\in I^2\times S^k}$ giving the following trivial signed count.
\[
\#\M(\{\data^s_{t,\theta}\}_{(s,t,\theta)\in I\times \{0,1\}\times S^k}) = 0.
\]
The difference between $\sw$ computed with the families $\data^0$ and $\data^1$ is then
\[
\begin{aligned}
\#\M(\{\data^1_{t,\theta}\}_{(s,t,\theta)\in \{1\}\times I \times S^k}) &- \#\M(\{\data^0_{t,\theta}\}_{(s,t,\theta)\in \{0\}\times I \times S^k}) \\
&= \#\M(\{\data^1_{t,\theta}\}_{(s,t,\theta)\in \{1\}\times I \times S^k}) - \#\M(\{\data^0_{t,\theta}\}_{(s,t,\theta)\in \{0\}\times I \times S^k}) \\
&+  \#\M(\{\data^s_{0,\theta}\}_{(s,t,\theta)\in I\times \{0\} \times S^k}) - \#\M(\{\data^s_{1,\theta}\}_{(s,t,\theta)\in I\times \{1\} \times S^k}) \\
&= \#\M(\{\data^s_{t,\theta}\}_{(s,t,\theta)\in \partial(I\times I \times S^k)}) = 0.
\end{aligned}
\]
The proof described by this calculation is schematically illustrated in Figure~\ref{F:independence}. In the case that $k=0$ but the orientation is not preserved, everything goes through identically except that the isomorphism between the moduli spaces $\M(\{\data^s_{0,\theta}\}_{(s,\theta)\in I\times S^k})$ and $\M(\{\data^s_{1,\theta}\}_{(s,\theta)\in I\times S^k})$ may not be orientation preserving, so we only get equality mod $2$.

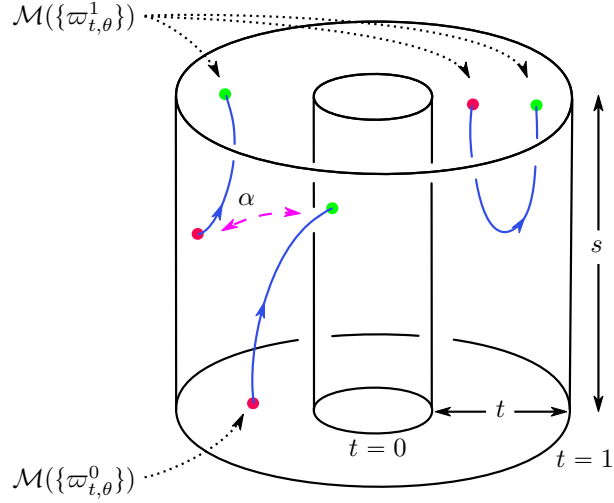
\begin{figure}[h]
\labellist
\small\hair 2pt
\pinlabel {$\alpha$} [ ] at 35 99
\pinlabel {$t$} [ ] at 122 27
\pinlabel {$t=0$} [ ] at 80 14
\pinlabel {$t=1$} [ ] at 151 10
\pinlabel {$s$} [ ] at 155 83
\pinlabel {$\M(\{\data^1_{t,\theta}\})$} [ ] at -24 161
\pinlabel {$\M(\{\data^0_{t,\theta}\})$} [ ] at -24 2
%\pinlabel {Legend} [ ] at 210 28
%\pinlabel { \framebox(150,115){}} [ ] at 210 72
%\pinlabel {$U_\epsilon$} [ ] at 280 210
\endlabellist
\centering
    \includegraphics[scale=1.1]{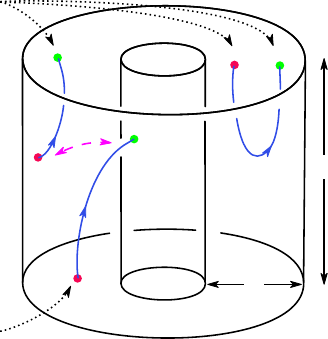}
  \caption{Independence from metric and perturbation}\label{F:independence}
\end{figure}
The proof that $\sw^{\pi_k}$ is a homomorphism is straightforward,  given the independence from all choices.  It can also be deduced from Proposition~\ref{P:eval} below. The last sentence in the lemma is also straightforward.
\end{proof}
\begin{remark}\label{R:ball}
    A very useful consequence of the independence of $\sw_{X}^{\pi_k}(\alpha, \spincs)$ from all choices is that it can often be calculated using \emph{constant} initial (good) data $\data^0_{\theta}$, which we abbreviate to $\data_*$. In that case, the moduli space $\M(\{\data^0_{\theta}\})$ is empty, and one can view the family moduli space $\M(\{\data^0_{t,\theta}\})$ as being defined over a ball.   This is used in the proof of Proposition~\ref{P:eval} below and in the proof of the gluing formula, Theorem~\ref{anti-hol}, in Section~\ref{S:calculation}.
\end{remark}

%We will show that this is a well-defined integer (i.e. is independent of all choices) when $n>0$. For $n=0$ it will be integer-valued when $\alpha \in \pi_0(\diff(X,\spincs,\calo))$ and $\Z_2$-valued when $\alpha \in \pi_0(\diff(X,\spincs))$.  As needed, we will write $\sw_X^{\pi_k}(\alpha,\spincs)$ to express the dependence on the choice of $\Spinc$ structure $\spincs$.

% \Red{There is an important variation \danny{Not clear that we need this; delete?} for $n=0$, in the case that $\alpha^*\spincs = \bar\spincs$, where $\bar\spincs$ is the conjugate $\Spinc$ structure.  Recall~\cite{morgan:swbook} that there is a homeomorphism between $\calm_X((g,\eta,\spincs)$ and $\calm_X((g,\eta,\bar\spincs)$, but the orientations are not necessarily the same. They
% compare by the product of the following signs
% , in which case we get a $\Z_2$-valued invariant.} 

\subsubsection{Relation to Konno's characteristic classes \texorpdfstring{$\swc$}{SW}}\label{s:konno}
For any $\spincs \in \Shat^k_X$, Konno's paper~\cite{konno:classes} defines Seiberg-Witten cohomology classes
\begin{align*}\label{E:konno}
\swc^{\Z_2}_X(\spincs) &\in H^{k+1}(\bdiff(X,\spincs);\Z_2),\\
\swc_X(\spincs,\calo) &\in H^{k+1}(\bdiff(X,\spincs,\calo);\Z)
\end{align*}
that give invariants for families of $4$-manifolds when $b^2_+(X) > k+2$. In particular, a family parameterized by a space $B$ with structure group preserving a $\spincs$ is classified by a map $B \to \bdiff(X,\spincs)$, yielding classes in $H^{k+1}(B;\Z_2)$. If the structure group also preserves the homology orientation $\calo$, then we get a classifying map $B \to \bdiff(X,\spincs,\calo)$) and hence classes in $H^{k+1}(B;\Z)$. There are several closely related ways to package these invariants. If
$\Xi\in H_{k+1}(\bdiff(X,\spincs);\Z_2)$ (respectively,  $H_{k+1}(\bdiff(X,\spincs,\calo);\Z)$) we write $\fsw^{\Z_2}(\Xi,\spincs) = \swc^{\Z_2}_X(\spincs)\cap\Xi$ (respectively, $\fsw^{\Z}(\Xi,\spincs,\calo) = \swc^{\Z}_X(\spincs,\calo)\cap\Xi$) for the \emph{family Seiberg-Witten invariants.}

Starting with a map $\alpha\colon S^k \to \diff(X)$
one builds a spherical family $X \to S^{k+1}\times_\alpha X \to S^{k+1}$, via the clutching construction 

% \Red{Here it is reasonable to view the family invariant as an invariant of a family of diffeomorphisms $\sw^{\pi_k,\Z_2}_X(\alpha,\spincs) = \fsw^{\Z_2}(\Xi,\spincs)$ ($\sw^{\pi_k,\Z}_X(\alpha,\spincs,\calo) = \fsw^{\Z_2}(\Xi,\spincs,\calo)$) where $\Xi_\alpha$ is the class represented by the classifying map of $S^{k+1}\times_\alpha X$.} \danny{are we doing signs? DA We should if we can} 

% For $\alpha: S^k \to \diff(X)$, write $E = S^{k+1} \times_\alpha X$ for the clutching construction 
\begin{equation}\label{E:clutch}
\left(D^{k+1}_+ \times X \amalg D^{k+1}_- \times X\right)/(z,x) \sim (z,\alpha(z)x) \ \text{for}\ z \in S^k.
\end{equation}

\begin{proposition}\label{P:eval}
Let $E = S^{k+1} \times_\alpha X$ have classifying map $f\colon S^{k+1} \to \bdiff(X,\spincs,\calo)$. Then 
$\sw_X^{\pi_k}(\alpha) = \langle f^*\swc_X(\spincs),S^{k+1}\rangle$.
\end{proposition}
\begin{proof}
By definition, $\langle f^*\swc_X(\spincs),S^{k+1}\rangle$ is computed cell-by cell, so in this case is the sum of the counts of solutions over the two hemispheres $D^{k+1}_\pm$.
%\danny{Proof doesn't really make sense as written; needs fleshing out.}
Because the initial data $\data^0$ is good, the moduli space corresponding to points on the upper hemisphere is empty. Hence the family moduli space associated to $\spincs$ that defines $\swc^{k+1}_{X}(f_*[S^{k+1}],\spincs)$ is the same as the moduli space associated to the lower hemisphere that computes $\sw_{X}^{\pi_k}(\alpha,\spincs)$.
\end{proof}
There is one additional difference in the way that $\sw_X^{\pi_k}$ and $\swc_X(\spincs)$ are defined. The invariants of the diffeomorphisms defined in this paper are based on suitable generic data, while Konno's invariants are based on the virtual neighborhood technique. For suitably generic data the virtual fundamental class is just the class generated by the moduli space. Some additional details about the equality of our invariant of diffeomorphisms and Konno's invariant are given in Drouin's thesis \cite{drouin:embeddings}.
\begin{remark}\label{R:BHurewicz}
    Another way to phrase Proposition \ref{P:eval} is via the isomorphism
    \[\partial\colon \pi_{k+1}(\bdiff(X)) \to \pi_k(\diff(X))\] 
    from the long exact sequence of the universal bundle. We have a commutative diagram
\begin{center}
\begin{tikzcd}
    \pi_k(\diff(X,\spincs,\calo)) \arrow{r}{\sw_{X}^{\pi_k}(\text{--}{,}\spincs)}  &\Z \\
    \pi_{k+1}(\bdiff(X,\spincs,\calo)) \arrow{u}[anchor=center,rotate=90,yshift=1ex]{\cong}[swap]{\partial}
    \arrow[r,"\calh" ]& H_{k+1}(\bdiff(X,\spincs,\calo))\arrow[u, "\swc^{k+1}_{X}(\text{--}{,}\spincs)"']
\end{tikzcd}
\end{center}
where $\calh$ is the Hurewicz map. The composition $\calh \circ (\partial)^{-1}(\alpha)$ is the homology class of the family representing the clutching construction \eqref{E:clutch} on $\alpha$.
\end{remark}

\subsection{Homology invariants of \texorpdfstring{$\diff(X)$}{Diff(X)}}\label{s:hom}
The constructions of the previous section give us invariants that will detect homotopy groups of $\diff(X)$ or equivalently, of $\bdiff(X)$, while Konno's invariants can detect homology groups of $\bdiff(X)$. Detecting elements of the homology groups of $\diff(X)$ requires some new constructions, which we provide in this section.  
\subsubsection{Moduli spaces parameterized by chains}\label{s:chains}
The invariants $\sw^{H_k}(\text{--}{,}\spincs)$ for $\spincs \in \Shat_{X}^k$ will defined in terms of a $(k+1)$-parameter moduli space, where the parameterizing space $\Xi$ is a $(k+1)$-chain in the space $\Pi(X)$ of data on $X$.  We discussed smoothness for moduli spaces parameterized by a manifold in Section~\ref{S:regularity}.
In this section, we explain how to extend those results to the setting of chains.  A reader who is interested only in invariants of $\pi_k(\diffo(X))$ can take $\Xi$ to be a $(k+1)$-dimensional ball as in  Definition~\ref{D:SW invariant} and Remark~\ref{R:ball}. The same setup would apply to bordism invariants defined on $\widetilde{\Omega}_k(\diffo(X))$, where now $\Xi$ would be a manifold with boundary.

Our approach to smoothness and transversality for chains is to work on one (singular) simplex at a time.  Recall that a singular simplex on a space $A$ is a continuous map $\sigma:\Delta^k \to A$, and a singular $k$-chain is a formal sum 
\[
\alpha = \sum a_\sigma \sigma\ \text{with}\ a_\sigma \in \Z
\]
and $a_\sigma \neq 0$ for finitely many $\sigma$. If $A$ is a smooth manifold, then singular homology based on such chains can be computed in terms of smooth maps $\sigma$. We recall that $\sigma$ is smooth if it has a smooth extension to a neighborhood of $\Delta^k$ (thought of as a subspace of $\R^k$).  The restriction of a smooth $k$-simplex to any face $\Delta^k_{(j)}$ is smooth, so the smooth simplices form a subcomplex giving the same homology as the full singular complex~\cite[Chapter 18]{lee:manifolds}. The proof is based on the Whitney approximation theorem and standard arguments in homology theory. As such, it applies directly when $X$ is instead a Banach manifold. 

Similarly, one can consider homology based on (now smooth) singular simplices that satisfy a transversality condition.  
\begin{definition}\label{D:transverse}
Let $p\colon \calp \to \calq$ be a Fredholm map between Banach manifolds, and let $\sigma\colon \Delta^k \to \calq$ be smooth. We say that $\sigma $ is transverse to $p$ if the restriction of $\sigma$ to the interior of each face of $\Delta^k$ is transverse to $p$ in the usual sense, i.e., the image of $d(\sigma)_x\colon T_x \Delta^k_{(j)}\to T_{\sigma(x)}\calq$ and the image of $dp_y$ span $T_{\sigma(x)}\calq$ whenever $\sigma(x) = p(y)$.
\end{definition}

From the definition, it is immediate that the span of the simplices transverse to $p$ forms a subcomplex $C_*^{\pitchfork p}(\calq)$ of the (smooth) singular chains of $\calq$.  
The transversality theorem for Fredholm maps~\cite[Proposition 4.3.10]{donaldson-kronheimer} (extending~\cite{smale:sard} from submanifold to map transversality) combines with the proof that smooth singular homology is the same as ordinary homology to prove the following result.
\begin{proposition}\label{P:transverse-hom}
For a Fredholm map $p\colon \calp \to \calq$, the inclusion $C_*^{\pitchfork p}(\calq) \subset C_*(\calq)$ induces an isomorphism on homology.
\end{proposition}

We will apply this proposition to the following situation; see Section~\ref{S:regularity} for more details.
Fix a $\Spinc$ structure $\spincs$ on a $4$-manifold $X$.  The space $\calq$ will be the space of data, denoted $\Pi$ above.  Moreover, $\calp$ will be the `universal' parameterized \SW moduli space consisting of pairs $([A,\psi], \data)$ where $[A,\psi]$ is a gauge-equivalence class of irreducible configurations satisfying the \SW equations with data $\data$.  (The structure of these spaces as Banach manifolds comes from the choice of suitable Sobolev norms as fixed in Remark~\ref{R:sobolev}.) The essence of the genericity theorem~\cite{kronheimer-mrowka:monopole,morgan:swbook,nicolaescu:swbook} is that $\calp$ is a Banach manifold, and $p([A,\psi], \data) = \data$ is a Fredholm map. The (usual) map transversality theorem then gives smoothness for moduli spaces parameterized by a smooth manifold. We will make a similar argument when the parameter space is a smooth chain. A related approach to homological family invariants was given in Section 3.4.1 of~\cite{munoz-echaniz:monopole} of Mu\~{n}oz-Ech\'aniz.

The genericity theorem does not work quite as well for the full space of connections, i.e. when we include the reducible points. However, we can avoid reducibles in any family that comes up in the definition and calculation of our invariants.  

%The signs in \eqref{E:hdef} are determined by 
%
If $\Xi$ is a manifold with boundary, then a {\em good family} of data $\data_{\theta\in \Xi}$ is one for which there is no reducible solution point in any moduli space $\M(\data_\theta)$, and   the map $\theta \to \data_\theta$ as well as its restriction to $\partial\hspace{.07em}\Xi$ are transverse to $p$. 
% (We will give a more general version in Definition~\ref{good}.)  
The parameterized moduli space $\M(\{\data_\theta\}_{\theta\in \Xi})$ is then an oriented manifold with boundary, using the convention in Remark~\ref{R:family-orient}.

A {\em good simplex} $\sigma\colon \Delta^k \to \Pi(X)$ is one for which the restriction of $\sigma$ to each face is good.  For a good simplex, the part of the parameterized moduli space over the interior of $\Delta$ and the interior of the faces of $\Delta$ is then an oriented manifold by the discussion above.  A chain $\sum a_\sigma \sigma$ is good if each of its simplices is good.
\subsubsection{Definition of the invariants}\label{S:invariants}
Making use of Proposition~\ref{P:transverse-hom} we can give the definition of $\sw^{H_k}(\text{--}{,}\spincs)$. 
\begin{definition}\label{Hinvdef}
Let $Z$ be a smooth oriented manifold with homology orientation and $b^2_+(X) > k+2$.   Let $\spincs\in\Shat_Z^k$, and 
let $\alpha = \sum a_\sigma \sigma$ be a smooth cycle, representing a homology class in $H_k(\diff(X))$. Let $\data_0$ be  
good data on $Z$; by the definition of $\Shat_Z^k$ and the condition on $b^2_+(X)$, the moduli space $\M(\data_0)$ is in fact empty. 

For each $\sigma$ in $\alpha$ with $a_\sigma \neq 0$, we get a smooth simplex in $\Pi(X)$ given by
\[
\begin{aligned}
\sigma^*\data_0 \colon& \Delta^k \to \Pi(X)\\
& x \to \sigma(x)^*\data_0
\end{aligned}
\]
This gives rise to a chain $\alpha^*\data_0 =  \sum a_\sigma \sigma^*\data_0$, which is readily seen to be a cycle. On the other hand, since $\Htilde_k(\Pi(X)) = 0$, this cycle is a boundary, i.e.,
\begin{equation*}%\label{E:boundary}
\alpha^*\data_0 = \partial \beta,\ \text{where}\ \beta = \sum b_\tau \tau.
\end{equation*}
By Proposition~\ref{P:transverse-hom}, we may assume that each $(k+1)$-simplex $\tau$ appearing in $\beta$ is good, so that we can count the points in the compact oriented $0$-manifold $\M(\tau)$. Thus we define
\begin{equation}\label{E:hdef}
\sw^{H_k}(\alpha,\spincs) =  \sum b_\tau\,  \#\M(\tau).
\end{equation}
%
%
%For each 
%$\{g_{t,\theta}\}_{(t,\theta)\in I\times S^k}$ be a family of metrics so that $g_{0,\theta}=g_\theta $ and $g_{1,\theta}=\alpha_\theta^*g_\theta$. Then
%\begin{equation}\label{E:def}
%\sw_\spincs^{\pi_k}(\alpha)= \#\M(\{g_{t,\theta}\}_{(t,\theta)\in I\times S^k}).
%\end{equation}
\end{definition}

\begin{remark}
Under the assumption that $b^2_+(X) > k+2$, we can avoid reducibles in both the moduli spaces associated to each simplex $\tau$ in the definition of the invariant and in homologies between cycles used to prove that the invariants are well defined. From here forward when we speak of these invariants for a manifold $Z$ we will assume that $b^2_+(X) > k+2$.
In analogy with the usual Donaldson invariants one expects to have versions of these invariants when $b^2_+(X) = k+1$ that are well-defined up to a choice of chamber.
\end{remark}

\begin{definition}\label{D:Mchain}
For $\beta =   \sum b_\tau \tau$ a good chain in $\Pi(X)$, write  $\M(\beta)$ for the disjoint union of the moduli spaces $\M(\tau)$, with appropriate orientations and multiplicities given by the coefficients $b_\tau$. If $\M(\beta)$ is $0$-dimensional (i.e. $\dim(\M(\tau)) = 0$ for all $\tau$) and there is no bubbling, then this is a finite set of points that we may count with sign to get $\# \M(\beta)$; this is the quantity on the right-hand side of \eqref{E:hdef}.  If  $\M(\beta)$ is $1$-dimensional, then it is an oriented $1$-manifold with boundary $\M(\partial \beta)$.
\end{definition}
By a standard argument, if $\dim(\M(\beta)) = 1$, then $\# \M(\partial \beta)= 0$.

\begin{remark} We comment briefly on an alternate approach to defining the invariant $\sw^{H_k}(\text{--},\spincs)$ using Kreck's concept of {\em stratifolds}~\cite{kreck:stratifolds}. A stratifold $\bT$ is a stratified space with strata $\bT^j$ for $0 \leq j \leq n = \dim(\bT)$, where $\bT^j$ has the structure of a smooth $j$-manifold. Part of the data is a subsheaf of the sheaf of continuous functions on $\bT$, whose sections restrict to smooth functions on the strata. There is a notion of a smooth map from a stratifold to a manifold (satisfying a version of Sard's theorem), and Kreck shows that a natural bordism group defined in terms of stratifolds with boundary is isomorphic to ordinary singular homology. Given the bordism-style definition of $\sw^{H_k}(\text{--},\spincs)$, it seems reasonable to expect that one can  define it in terms of maps of stratifolds. A key point would be to establish a suitable Sard-Smale type theorem (and a version of Proposition~\ref{P:transverse-hom}) for maps of a stratifold to a Hilbert manifold such as the space of metrics; compare~\cite{kreck-tene:hilbert}.
%This is a reasonable idea, but there are some technical issues to overcome. The notion of smooth map to an infinite dimensional manifold is a bit tricky (compare~\cite{kreck-tene:hilbert}) but the following approach seems promising. The notion of transversality to a Fredholm map such as $p$ depends ultimately on expressing $p$ locally in terms of a map between finite-dimensional manifolds, so it seems that one can define a notion of a map of stratifolds transverse to $p$ and prove a result analogous to Proposition~\ref{P:transverse-hom}.  With this in hand, the definition of $\sw_\spincs^{H_k}$ should proceed as above. Since good chains provide an acceptable approach, we will not pursue the stratifold version any further in the current work.
\end{remark}

We now discuss the basic properties of the invariants $\sw^{H_k}(\text{--},\spincs)$: they are well-defined and give rise to $\Z$-valued homomorphisms from the homology groups of $\diffo(X)$.  By composition with the Hurewicz map, we then get homomorphisms $\sw^{\pi_k}(\text{--},\spincs)\colon \pi_k(\diffo(X)) \to \Z$ and in fact our main calculation (Theorem~\ref{anti-hol}) evaluates $\sw^{\pi_k}(\text{--},\spincs)$ for the spherical families of diffeomorphisms constructed in Section~\ref{fam}. These results generalize the properties of the invariants from the $k=0$ case defined in~\cite{ruberman:isotopy}. When $k = 0$, there are some modest complications in the formulation of the results,  because an arbitrary diffeomorphism might not preserve a given $\Spinc$ structure or homology orientation. Hence we restrict ourselves to $k > 0$.

\begin{proposition}\label{basic}
%Let $[S^k,\diff(X)]_+$ denote the group of free homotopy classes of maps that preserve the orientation and homology orientation. 
Let $k$ be greater than $0$. For any $\spincs\in\Shat_Z^k$ the map $\sw^{H_k}(\text{--},\spincs)$ on good chains descends to a well-defined group homomorphism
\[
\sw^{H_k}(\text{--},\spincs)\colon H_k(\diffo(X))\to\Z.
\]
In particular, it is independent from the choices of good cycle $\alpha$ representing a given homology class, the choice of bounding chain $\beta$, as well as the initial good data $\data_0$.  It is natural in the following sense. Let $\varphi\colon Y\to Z$ be a diffeomorphism that preserves the orientation and homology orientation and let $\spincs\in\Shat_Z^k$. Then
\[
\sw^{H_k}(\alpha^\varphi,\varphi^*\spincs) = \sw^{H_k}(\alpha,).
\]
\end{proposition}
By considering all of the possible $\Spinc$ structures with the appropriate dimensional moduli space, we may also regard $\sw^{H_k}$ as a homomorphism from $H_k(\diff(X))$ to $\Z^{\Shat_Z}$.

\begin{proof} First we show the independence from the choice of initial data, with a fixed cycle $\sum a_\sigma \sigma$ representing the class $[\alpha]$. Let $\data_0$ and $\data_1$ be good data, and let  $\data_s$ be a good path between them. The assumption that $k>0$ implies that $\M(\data_s)$ is empty for all $s \in [0,1]$. For each simplex $\sigma$ in $\alpha$, we get a map $\tilde{\sigma}^*\colon \Delta \times I  \to \Pi(X)$ that sends $(x,s)$ to $\sigma(x)^*(\data_s)$. Since the moduli space $\M(\data_s)$ is empty, the same is true for $\M(\sigma(x)^*(\data_s))$.  Apply
 the standard prism operator~\cite{hatcher} to get a chain $\tilde{\alpha}^*\{\data_s\}$ in $\Pi(X)$ with boundary $\alpha^*\data_1 - \alpha^*\data_0$.  It is possible that $\tilde{\alpha}^*\{\data_s\}$ as initially constructed is not good, but after a small perturbation of the faces in the interior of the prism followed by a perturbation of the interior it will be good.

Following the definition of $\sw^{H_k}$, choose good chains $\beta_i$ in $\Pi(X)$ with $\partial \beta_i = \alpha_i^* \data_i$. Then 
\begin{equation}\label{E:cobound}
\tilde{\beta} = \beta_1 - \beta_0 - \tilde{\alpha}^*\{\data_s\}
\end{equation}
is a good cycle in $\Pi(X)$. Again using the contractibility of $\Pi(X)$ and Proposition~\ref{P:transverse-hom}, there is a good chain $\gamma\in C_{k+2}^{\pitchfork p}(\calq)(\Pi(Z))$ with $\partial \gamma = \tilde{\beta}$.  As remarked after Definition~\ref{D:Mchain}, the boundary of the moduli space $\M(\gamma)$ is the oriented $0$-manifold
\[
\M( \beta_1) - \M(\beta_0) - \M(\tilde{\alpha}^*\{\data_s\}).
 \]
 Since this last moduli space is empty, we conclude that $\# \M(\beta_1) = \#\M(\beta_0)$. 

The proof that $\sw^{H_k}$ is independent of the choice of representative for $[\alpha] \in H_k(\diffo(Z))$ is similar.
Pick  initial good data $\data_0$. Given two smooth representatives $\alpha_0$ and $\alpha_1$, there is a smooth $(k+1)$-chain $\beta$ with $\partial \beta = \alpha_1 - \alpha_0$. Choosing good chains $\beta_i$ in $\Pi(Z)$ with $\partial \beta_i = \alpha_i^*\data_0$, one gets a good $(k+1)$-cycle $\tilde\beta = \beta_1 -\beta_0 -\beta^*\data_0$. As above,
\[
\#\M(\tilde\beta) = \#\M(\beta_1) -\#\M(\beta_0) -\#\M(\beta^*\data_0) =  \#\M(\beta_1) -\#\M(\beta_0)
\]
since $\#\M(\beta^*\data_0)$ is empty. On the other hand, $\tilde\beta$ bounds a good $(k+2)$-chain, so $\#\M(\tilde\beta) = 0$. As a special case of this argument, we could take $\alpha_1 = \alpha_0$, and conclude that the value of $\sw^{H_k}$ is independent of the choice of bounding chain $\beta$.

Given the independence from all choices, the fact that $\sw^{H_k}$ is a homomorphism is a straightforward consequence of the definition.
\end{proof}
If $\Xi$ is a manifold (an oriented manifold), and $\data$ is generic, then $\swc^{\Z_2}_X(\spincs)\cap\Xi = \# \calm_{X,\spincs}(\data)$ is the parity of the number of elements in the parameterized moduli space. More generally, Konno argues that it suffices to understand the evaluation of the Seiberg-Witten characteristic classes on cycles represented by CW complexes which via restriction to open cells essentially reduces to the case of cycles represented by smooth manifolds.

The smooth manifold  $\calm_{X,\spincs}(\data)$ inherits an orientation from an orientation on $\Xi$ together with an orientation $\calo$ for the vector space $H^2_+(X)$.  We use the convention that at a regular point $\data$ of the projection $\calm_{X,\spincs}(\data) \to \Xi$, the orientation is determined by the orientation of $\Xi$ followed by the orientation of the moduli space $\calm_{X,\spincs}(\data)$. In this case, Konno defines a $\Z$-valued invariant  $\swc^{\Z}_X(\spincs)\cap\Xi$ as the signed count of points.

\subsection{Basic properties of the invariants}\label{s:kis0}
In this short section, we note some basic properties of invariants that will be useful later on. One of the most important features of the \SW equations is the following \emph{a priori} bound for solutions \cite{nicolaescu:swbook}.
\begin{proposition}
    If $(A,\psi)$ is a solution to the Seiberg-Witten equations, then
    \[
\|\psi\|^2_\infty \le 2 \,\text{\rm max}\left(0,4\|\eta\|_\infty - \text{\rm min}(s(x))\right),
    \] where $s$ is the scalar curvature on $X$.
\end{proposition}
Of the many consequences of this proposition is the fact that any manifold has a finite collection of basic classes. In the case of families, one has the following.
\begin{lemma}\label{L:finite}
For any $\alpha \in \pi_k(\diff(X))$, the set of $\Spinc$ structures $\spincs$ for which $\sw_X^{\pi_k}(\alpha,\spincs) \neq 0$ is finite.
\end{lemma}
\begin{remark}\label{R:finite}
It is worth noting that this lemma does not extend to say that for fixed $X$ and $k$, the set of $\Spinc$ structures $\spincs$ for which the homomorphism $\sw_X^{\pi_k}(\cdot,\spincs)$ is non-zero is finite. Indeed, the proof of Theorem~\ref{T:infgen} shows that this set can be infinite.
\end{remark}

\section{Non-isotopic diffeomorphisms}\label{S:pi0}
We start with the proof of Theorem~\ref{T:infgen} for $p=0$, which exhibits the basic idea and will serve as the basis for an inductive proof of the theorem for $p>0$. 
The first step is the construction of non-isotopic diffeomorphisms that will eventually be detected by a $1$-parameter Seiberg-Witten invariant.  This is analogous to the construction in~\cite{ruberman:isotopy}, and is a different way to describe diffeomorphisms studied in  \cite{baraglia-konno:gluing}. The basic idea is to transport a basic model diffeomorphism on $\cptwo\cs 2\cptwobar$ (in \cite{ruberman:isotopy}) or  $\sss$ (in the present paper and in~\cite{baraglia-konno:gluing}) to a manifold that has two distinct decompositions of the form $X\cs\cptwo\cs 2\cptwobar$ or $X\cs (\sss)$.

%but with a twist that makes it more compatible with composition of diffeomorphisms. 
To follow this idea and allow room to modify the construction to obtain higher-dimensional families, notice that the elliptic surface $E(2)$ contains three different Gompf nuclei, $N(2)_a$, $a = 1, 2, 3$. Each of these nuclei contains a fiber denoted $T_a$ and a section denoted $\sigma_a$. Construct the collection $\{\bX_q=E(2;2q+1)\}$ of log-transformed elliptic surfaces via log-transform in the $N(2)_1$ nucleus.  These are all homeomorphic, and we fix homeomorphisms $\psi_q \colon \bX_q \to \bX_0 = E(2)$ via homeomorphisms that restrict to the identity on $N(2)_2$. The manifolds $\bX_q$ are distinguished by their Seiberg-Witten invariants, as follows.
% w\dannyin{We use $\{\bX_q=E(2;2q+1)\}$ throughout the paper. Should we use ${\bf X}_q$ to indicate that these are `special' manifolds?}
\begin{equation}\label{E:logsw}
\sw_{\bX_q}(2\ell \Tc_1)=
\begin{cases}
    1 \ \text{if} \ |\ell | \leq q\\
    0\ \text{otherwise}
\end{cases}
\end{equation}
where $\Tc_1$ denotes the Poincar\'e dual of $T_1$.

It is known~\cite{gompf:elliptic} that $\bX_q$ and $\bX_0$ become diffeomorphic   after a single stabilization, and we choose a diffeomorphism 
\[
\vphi_q\colon \bX_q \cs (\sss) \to \bX_0 \cs (\sss)
\]
taking the copy of $N(2)_2$ in $\bX_0$ to the copy $N(2)_2$ in $\bX_q$.  We can assume, by composing with a fiber-preserving self-diffeomorphism, that $\vphi_q \simeq \psi_q \cs 1_{\sss}$.

Now we make use of a pair of simple diffeomorphisms of $\sss$. Let $S$ be a sphere of self-intersection $\pm 2$ in a $4$-manifold. Wall introduced a diffeomorphism that is the identity outside of a neighborhood of this sphere and acts as
$x\mapsto x\mp\langle x, [S]\rangle [S]$ on the homology \cite{wall:diffeomorphisms}. Denote this diffeomorphism by $R_S$. Let $A$ denote the sphere $S^2\times\{\text{pt}\}$, $B$ denote the sphere $\{\text{pt}\}\times S^2$, and $A\pm B$ denote the result of smoothing the intersection point in the union of the spheres using the indicated orientations. The diffeomorphisms on $\sss$ are just $R_{A\pm B}$. The composition $R_{A+B}R_{A-B}$ is isotopic to the diffeomorphism  $(z,w) \mapsto (\bar{z},\bar{w})$ (with $S^2$ viewed as $\mathbb{C}\cup\{\infty\}$) used in \cite{baraglia-konno:gluing}. This replaces the reflections in $2$-spheres of self-intersection $\pm 1$ used in~\cite{ruberman:isotopy}. 

\begin{remark}\label{explicitRR}
    Evaluating the parameterized deformation complex will be simplified by using the following explicit models of these diffeomorphisms. Write $\sss = \{(v,w)\in\R^2\times\R^3\,|\,|v| = |w| = 1\}$, so that $A\pm B = \{(v,w)\in \sss\,|\, v =\pm w\}$. Set $R'_{A\pm B}(v,w) = \mp(w,v)$. Thus $R'_{A+B}R'_{A-B}(v,w) = -(v,w)$. Composing this with an order two rotation yields the isotopic map
    $RR'((v_1,v_2,v_3),(w_1,w_2,w_3)) = ((v_1,v_2,-v_3),(w_1,w_2,-w_3))$. The maps $R_{A\pm B}$ and $RR$ are obtained from these via isotopy to be fixed on a disk. 
\end{remark}

One advantage of writing the diffeomorphisms as a composition of reflections is that this allows one to understand the behavior of diffeomorphisms under stabilization by understanding how spheres behave under stabilization. This idea of creating isotopies of diffeomorphisms by finding isotopies of spheres was the initial motivation for the stable isotopy results of \cite{auckly-kim-melvin-ruberman:isotopy} and was also implemented for higher-parameter families in \cite{auckly-ruberman:diffym}; the reflections $R_{A\pm B}$ were used in this way in \cite{auckly:internal}.

The manifold $\bZ^0$ that appears in Theorem~\ref{T:infgen} for $p=0$ is by definition  $\bX_0 \cs (\sss)$. Let $\spincs_0$ be the trivial $\Spinc$ structure on $\sss$. Conjugating $R_{A+B}R_{A-B}$ by $\vphi_q$, we get a collection of self-diffeomorphisms of $\bZ^0$.
%\danny{the inverse isn't strictly necessary but it helps me keep things straight.}
\begin{multline}
\balpha[q] = \vphi_q(1_{\bX_q} \cs R_{A+B}R_{A-B}) \vphi_q^{-1}(1_{\bX_0} \cs R_{A+B}R_{A-B})^{-1}\\ =  ~^{\vphi_q}(1_{\bX_q} \cs R_{A+B}R_{A-B}) (1_{\bX_0} \cs R_{A+B}R_{A-B})^{-1}.
\end{multline}

As we would usually do in a group, we write $~^hg = hgh^{-1}$. It is useful to summarize this definition as $\balpha[q] = [\vphi_q,1_{\bX_q} \cs R_{A+B}R_{A-B}]$.
By construction, the diffeomorphisms $\balpha[q]$ are in the Torelli group $\tdiff(\bZ^0)$, and in particular lie in $\diff(\bZ^0,\ell\check{T},\calo)$ for all $\ell \in \Z$. Note that these latter groups are the same for all non-zero $\ell$.
\begin{proposition}\label{P:pi0}
The collection of homomorphisms $\sw^{\pi_0}_{\bZ^0}(-,2\ell \Tc)$ for $\ell \in \Z_{\geq 0}$ define a homomorphism  $\vphi\colon \pi_0(\diff(\bZ^0,\check{T},\calo)) \to  \Z^\infty$ with infinitely generated image. In particular $\pi_0(\diff(\bZ^0,\check{T},\calo))$ is infinitely generated, and contains an infinitely generated subgroup lying in the kernel of the natural map $\pi_0(\diff(\bZ^0,\check{T},\calo)) \to\pi_0(\home(\bZ^0))$.
\end{proposition}
As will be evident from the proof, the same holds for $\pi_0(\tdiff(\bZ^0))$.
\begin{proof}
%To show that the image is infinitely generated, we show that the mod $2$ reduction of $\vphi$, with image in $(\Z_2)^\infty$ has infinitely generated image. From this we deduce that $\vphi$ has infinitely generated image.  

We note first that the self-diffeomorphisms
\[
(1_{\bX_0} \cs R_{A+B}R_{A-B})^{-1}\ \ \text{and}\ \ 
~^{\vphi_q}(1_{\bX_q} \cs R_{A+B}R_{A-B})
\]
of $\bZ^0$ each preserve the $\Spinc$ structure $2\ell \Tc \cs \spincs_0$. They reverse the homology orientation, and so live in the group $\diff(\bZ^0,\Tc)$, which implies that the invariants $\sw_{\bZ^0}^{\pi_0,\Z_2}((1_{\bX_0} \cs R_{A+B}R_{A-B})^{-1},2\ell \Tc \cs \spincs_0)$ and $\sw_{\bZ^0}^{\pi_0,\Z_2}(~^{\vphi_q}(1_{\bX_q} \cs R_{A+B}R_{A-B},2\ell \Tc \cs \spincs_0))$ are well-defined. 
We have $\sw_{\bZ^0}^{\pi_0}(\balpha[q],2\ell \Tc \cs \spincs_0) \equiv \sw_{\bZ^0}^{\pi_0,\Z_2}(\balpha[q],2\ell \Tc \cs \spincs_0) \ (\text{mod} \ 2)$ and
\begin{align*}
\sw_{\bZ^0}^{\pi_0,\Z_2}(&\balpha[q],2\ell \Tc \cs \spincs_0) \\ &=  \sw_{\bZ^0}^{\pi_0,\Z_2}(~^{\vphi_q}(1_{\bX_q} \cs R_{A+B}R_{A-B}),2\ell \Tc \cs \spincs_0)) - \sw_{\bZ^0}^{\pi_0,\Z_2}((1_{\bX_0,2\ell \Tc \cs \spincs_0} \cs R_{A+B}R_{A-B}))\\
& = \sw_{\bZ^0}^{\pi_0,\Z_2}(1_{\bX_q} \cs R_{A+B}R_{A-B},2\ell \Tc \cs \spincs_0) - \sw_{\bZ^0}^{\pi_0,\Z_2}(1_{\bX_0} \cs R_{A+B}R_{A-B},2\ell \Tc \cs \spincs_0).
\end{align*} 
 
Now Baraglia and Konno~\cite{baraglia-konno:gluing} have computed the latter invariants to be (respectively) $\sw_{\bX_q}(2\ell \Tc)$ and $\sw_{\bX_0}(2\ell \Tc)$. The $\sw_{\bZ^0}^{\pi_0}$ invariants may be combined into a single invariant $\Phi$ taking values in the free abelian group generated by $\Spinc$ structures (a $\Z^\infty$-valued invariant).  Combined with the computation~\eqref{E:logsw} of $\sw_{\bX_q}$ this shows that (again, mod $2$)
\[
\Phi(\balpha[q]) \equiv (0,\overbrace{1,\ldots,1}^{q-1\rm\ times},0,\ldots) \ (\text{mod} \ 2).
\]
Furthermore, by Lemma~\ref{L:finite} only a finite number of components of $\Phi(\balpha[q])$ may be non-zero.
It follows readily that %the image of $\vphi$ in $(\Z_2)^\infty$ is infinitely generated and hence the 
image of $\Phi$ in $\Z^\infty$ is infinitely generated. The subgroup of $\pi_0(\diff(\bZ^0,\check{T},\calo))$ generated by the $\balpha[q]$ is infinitely generated, and since each $\balpha[q]$ acts trivially on homology, are trivial in $\pi_0(\home(\bZ^0))$.
\end{proof}
\begin{remark}\label{R:nonabelian}
    Because $\zz^0$ is the stabilization of an indefinite manifold, it follows from Wall's theorem~\cite{wall:diffeomorphisms} that the natural map from $\pi_0(\diff(\zz^0))$ to $\Aut(Q_{\zz^0})$ (where $Q_X$ denotes the intersection form of $X$) is surjective. It follows that $\pi_0(\diff(\zz^0))$ is not abelian. The proof of Proposition~\ref{P:pi0} shows a different failure of commutativity that is not detected in this fashion. Indeed, one checks that up to isotopy $(~^{\vphi_q}(1_{\bX_q} \cs R_{A+B}R_{A-B}))^2$ and $(1_{\bX_0} \cs R_{A+B}R_{A-B})^2$ are trivial, but $\balpha[q]$ has infinite order in the mapping class group.  Thus it seems unlikely that  there is a splitting of $\Phi$ over its (free abelian) image; compare~\cite[Theorem 1.4]{baraglia:mcg}.

\end{remark}

Proposition~\ref{P:pi0} is one part of the base case of our recursive construction. The other elements necessary for the base case are contained in the following. The idea of using isotopies of spheres to create isotopies of diffeomorphisms is the motivation for using the diffeomorphism $R_{A+B}R_{A-B}$. Note that the diffeomorphism $(z,w) \mapsto (\bar{z},\bar{w})$ used  is isotopic to $R_{A+B}R_{A-B}$ but  we found it easier to visualize the stable isotopy  described below using the $R_{A+B}R_{A-B}$ model.

\begin{proposition}\label{P:Rstab}
    The diffeomorphisms $\balpha[q]$ can be assumed to be the identity on a ball $B^4$, and become isotopic to the identity after one external stabilization, i.e., $\balpha[q]\cs 1_{\sss}$. They are also  smoothly pseudoisotopic and topologically isotopic to the identity. The stable isotopy, pseudoisotopy, and topological isotopy can all be assumed to be the identity on a ball. 
\end{proposition}
\begin{proof}
    We have 
    \begin{align*}
        \balpha[q] &= ~^{\vphi_q}(1_{\bX_q} \cs R_{A+B}R_{A-B}) (1_{\bX_0} \cs R_{A+B}R_{A-B})^{-1} \\
        &= (1_{\bX_0} \cs R_{\vphi_q(A+B)})(1_{\bX_0} \cs R_{\vphi_q(A-B)}) (1_{\bX_0} \cs R_{A-B})^{-1}(1_{\bX_0} \cs R_{A+B})^{-1}.  
    \end{align*}
Now $A\pm B$ and $\vphi_q(A\pm B)$ each have a geometric dual, given by $B$ and $\vphi_q(B)$, respectively. This implies that these are primitive, ordinary with simply connected complement. Thus the main theorem of \cite{auckly-kim-melvin-ruberman-schwartz:1-stable} implies that they become isotopic after one external stabilization so that $\balpha[q]$ becomes isotopic to the identity after one stabilization. Since all of $\balpha[q]$ induce the identity homomorphism on $H_2(\bZ^0)$, they are smoothly pseudoisotopic~\cite{kreck:isotopy,quinn:isotopy} and hence topologically
%\danny{Corrections to Quinn} isotopic~\cite{perron:isotopy2,quinn:isotopy,gabai-etal:pseudoisotopies} 
to the identity. 

By construction, the diffeomorphisms $\balpha[q]$ are supported in neighborhoods of spheres representing $A \pm B$, so $\balpha[q]$ are supported away from a fixed ball $B^4$ disjoint from the union of those spheres. The stable isotopy of $\balpha[q]$ to the identity is constructed by the isotopy extension theorem, and hence can be assumed to be the identity on that same ball. It is not hard to check that the pseudoisotopy and topological isotopy theorems work on manifolds with boundary $S^3$, and so they can be assumed to be the identity on that same ball. 
\end{proof}
This proposition also appears as Lemma~2.2 of the first author's paper~\cite{auckly:internal}. The statements about pseudoisotopy and topological isotopy can also be deduced from more general theorems in~\cite{orson-powell:mcg4}. 
\begin{remark}\label{R:FGK-1}
    The stable isotopy, topological isotopy, and pseudoisotopy are important ingredients in the recursive construction of interesting spherical families of diffeomorphisms. We use the letters $F[q]$, $G[q]$, and $K[q]$ respectively to denote particular choices for these (with some specifications below). When we create spherical families, we will add superscripts to denote the dimension of the associated sphere; in particular $F$ would be denoted by $F^0$.
  
%     The stable isotopy $F[q] = F^0[q]$ is fixed as the one coming from~\cite{auckly-kim-melvin-ruberman:isotopy}, but $G^0$ and $K^0$ are arbitrary. As we will see, those choices for G^0$ and $K^0$  do not matter, but the choice of $F$  
     For example, we have
    \begin{align*}
   F^0[q] =    F[q]\colon I\times \bZ^0 \cs (\sss) &\overset{\cong}{\rightarrow} I\times \bZ^0\cs(\sss), \\
       G^0[q] = G[q]\colon I\times \bZ^0  &\overset{\approx}{\rightarrow} I\times \bZ^0, \\
       K^0[q] = K[q]\colon I\times \bZ^0  &\overset{\cong}{\rightarrow} I\times \bZ^0, \\
    \end{align*}
with each restricting to the identity when $t=0$ and to $\balpha[q]\cs 1_{\sss}$ or $\balpha[q]$ when $t=1$.
%
%
%$F^0 (t,z) = (t,F_t[q](z))$, $F_1 = 1_{\bZ^0\cs(\sss)}$, $F_1[q] = \balpha[q]\cs 1_{\sss}$, $G^1[q](t,z) = (t,G_t[q](Z))$, $G_0[q] = 1_{\bZ^0}$, $G_1[q] = \balpha[q]$, $K = 1_{\bZ^0}$, $K_1[q] = \balpha[q]$, where $F$ and $K$ are diffeomorphisms. 
%

We may further take each of these with support away from $I\times N(2)_2$ so that each restricts to the identity on this subspace. For $F$, this follows from the explicit construction in~\cite{auckly-kim-melvin-ruberman:isotopy}. For $G$ and $K$ we can use the extension of~\cite{kreck:isotopy,quinn:isotopy, gabai-etal:pseudoisotopies,perron:isotopy2} to the case of manifolds with boundary, which is relatively straightforward in this case because the boundary of $N(2)_2$ is a homology sphere. Since their supports are away from $I\times N(2)_2$, each of these also extend to maps defined on related manifolds constructed via connected sum or submanifold sum away from the support of the map. We will generally denote such an extension with the same letter. For example, because $\bZ^0\cs_{T_2}\bZ^0 \cong (\bZ^0\cs(\sss))\cs_{T_2}\bX_0$, we have an isotopy
\[
\bZ^0\cs_{T_2}\bZ^0 \to I\times \bZ^0\cs_{T_2} \bZ^0
\]
that we would also denote by $F$.   In principle, there may be more than one $\sss$ summand so we will need to specify which one is being used to create an isotopy of $2$-spheres and hence of our diffeomorphism.  
\end{remark}

\section{Families of diffeomorphisms and collections of manifolds}\label{fam}

In this section, we give our basic construction of families of diffeomorphisms parameterized by spheres. The construction is inductive, with the starting point being a variant of the diffeomorphisms defined in the previous section. The induction makes use of the \emph{commutator construction} described in subsection \ref{S:comm}. The inductive step takes a $k$-dimensional family of diffeomorphisms on a given manifold $\bZ^k$ that becomes trivial after a single stabilization and produces a $(k+1)$-dimensional family on a new manifold $\bZ^{k+1}$. Morally, $\bZ^{k+1}$ is just $\bZ^{k}$ stabilized once, but in practice will in fact be somewhat bigger. Subsection~\ref{S:blocks} gives a suitable family of manifolds, based on constructions of Jongil Park~\cite{park:spin} that will be used to prove Theorem~\ref{T:infgen}. The Betti number of these manifolds gets large rapidly, and we have not tried to optimize the construction to make the manifolds smaller.

\subsection{Commutators and spheres in \texorpdfstring{$\diff(Z)$}{Diff(Z)}}\label{S:comm}
Even when one is just interested in maps $X\to Y$, it is useful to study maps $H\colon \Xi\times X\to  \Xi\times Y$ for a parameter space $\Xi$, generally taken to be a sphere or disk or a product of such. For example, a pseudoisotopy is a diffeomorphism $K\colon I\times Z \to I\times Z$ taking $\{i\} \times Z$ to $\{i\} \times Z$ for $i=0,1$, and we say that 
$f_0 = K\circ i_0$ and $f_1 = K\circ i_1$ are pseudoisotopic. The same definition for homeomorphisms gives the equivalence relation of topological pseudoisotopy. A \emph{family of maps} is a map $H\colon \Xi\times X\to  \Xi\times Y$ that preserves the parameter direction, or in other words is level-preserving. This means there are maps
$H_t\colon  X \to Y$, so that $H(t,x) = (t,H_t(x))$. A topological isotopy is the special case (denoted with $G$) where $\Xi = I$, $X= Y = Z$ and each $G_t$ is a homeomorphism. An isotopy (denoted with $F$) is the special case where in addition $Z$ is a smooth manifold, $F$ is smooth and $F_t$ is a diffeomorphism for all $t$. A \emph{pseudocontraction}, \emph{topological contraction}, or \emph{contraction} will be the analogous types of maps replacing $\Xi = I$ by $\Xi = I\times S^k$.

The \emph{support} of a family of self-maps ($X = Y = Z$) is the closure of the set 
\[
\{z \,|\, \ \text{there is a} \ t \ \text{so that} \ H_t(z) \neq z \}.
\]
When we speak of an isotopy or contraction relative to $A$ we will mean one with support disjoint from $A$. The notion of pseudoisotopy $K$ relative to $A$ also makes sense; one requires that $K$ be the identity on $I \times A$. We will mostly use this notion when $K$ is a ball, a neighborhood of a torus, or a nucleus.

Given two families $g_t,\ h_t$ of self-maps parameterized by the same space $\Xi$ with with each $g_t,\ h_t$ a homeomorphism or diffeomorphism, we form the commutator as the composition $[g,h]_t = g_th_t(g_t)^{-1}(h_t)^{-1}$ for $t \in \Xi$. We will regard a diffeomorphism or homeomorphism $h\colon Z\to Z$ as a constant family $h\colon \Xi\times Z\to  Z$ with $h_t=h$ for all $t$.
%\unred{Let $F\colon I\times S^k \to \diff(Z)$ and $h\in \diff(Z)$ so that the support of $F$
%is disjoint from $({0}\times S^k)\cup(I\times pt)$ and the supports of $F_1$ and $h$ are distinct. It follows that the support of $[F,h]$ is disjoint from $({0,1}\times S^k)\cup(I\times pt)$ so it descends to a family on $S^{k+1}$. In practice the family of maps $F$ will arise as
%the smooth contraction after stabilizing a family defined over $S^k$ and $h$ will be a standard
%diffeomorphism in a local model that is part of the larger manifold.}
This construction gives rise to maps of spheres into $\diff(Z)$, as follows. 
Let $\beta$ be  a diffeomorphism of $Z$ and $F\colon  I \times S^k \to \diff(Z)$ be a family of maps with the support of $F_0$ and the support of $F_1$ both disjoint from the support of $\beta$. In this case, view $S^{k+1}$ as the unreduced suspension of $S^k$ and consider the commutator $[F,\beta]$. On points of the form $(0,\theta)$ or $(1,\theta)$, $F$ and $\beta$ commute so the commutator is trivial and therefore gives a well-defined family of maps parameterized by the unreduced suspension.  If in addition, each of the maps $F_t$ in the family takes the basepoint to the identity, then the commutator will descend to a well defined map on the reduced  
 suspension $I\times S^k/(I\times\text{pt}\cup\{0,1\}\times S^k)$, with the basepoint in each dimension denoted pt. 

This commutator construction is especially useful when a family of maps is stably trivial under submanifold sum. Let $\Sigma$ be a submanifold of $Z'$ and of $M$,  so that the normal bundles $N^{Z'}(\Sigma)$ and $N^M(\Sigma)$ are fiber-orientation-reversing isomorphic. One then defines the submanifold sum \begin{equation}\label{E:submanifoldsum}
    Z = Z'\cs_\Sigma M = (Z' - \interior N^{Z'}(\Sigma)) \cup ( M - \interior N^M (\Sigma)).
\end{equation}

If
% $\alpha\in\pi_k(\diff(Z',N^{Z'}(\Sigma)))$
% is $1$-stably trivial, meaning 
$\alpha\cs_\Sigma 1_M$ is trivial in $\pi_k(\diff(Z))$, then the null-homotopy is a based family of maps $F$ so that the support of $F_0$ and $F_1$ is disjoint from $M\setminus\Sigma$. Given any $\beta$,  a diffeomorphism on $M$ with support in $M\setminus\Sigma$, the commutator $[F,\beta]$ then gives an element of $\pi_{k+1}(Z)$.

By using commutators to raise the dimension $k$ of the parameterizing sphere in this manner, we prove Theorem~\ref{T:infgen} inductively. The starting point is the collection of diffeomorphisms $\balpha[q]$ introduced in section~\ref{S:pi0}. To define spherical families, we let $\balpha^0[q]\colon S^0\times \bZ^0 \to S^0\times \bZ^0$, be given by 
\[
\balpha^0[q](\epsilon,z) = \begin{cases} (\epsilon,\balpha[q](z)), \ \epsilon = -1, \\
(\epsilon,z), \ \epsilon = 1.
\end{cases}
\]
Extend the stable isotopy, topological isotopy, and pseudoisotopy from Remark~\ref{R:FGK-1} to maps
   \begin{gather}\label{E:FGK0}
        F^0\colon I\times S^0\times \bZ^0 \cs (\sss) \to I\times S^0\times \bZ^0\cs(\sss), \\
        G^0\colon I\times S^0\times \bZ^0  \to I\times S^0\times \bZ^0, \\
        K^0\colon I\times S^0\times \bZ^0 \to I\times S^0\times \bZ^0, \\
    \end{gather}
that take $(t,1,z)$ to $(t,1,z)$ in order to be consistent with the basepoint $1\in S^0$. If we wish to keep track of the map at the end of the isotopy we will append the index to the notation in square brackets as in $F^0[q]$. We will eventually proceed from the base case to families $\balpha^p$, stable isotopies $F^p$, topological isotopies $G^p$ and concordances $K^p$. But first, we take a detour to produce the manifolds on which these families live. 

\subsection{Building blocks}\label{S:blocks}
We rely on the following observation about the relationship between torus sums and decompositions to understand the behavior of our spaces and maps under stabilization. The basic fact, sometimes called the Mandelbaum-Moishezon trick~\cite{mandelbaum:irrational,mandelbaum-moishezon:algebraic,gompf:elliptic}, is summarized in the following lemma. 
\begin{lemma}\label{L:MM}
If $X$ and $Y$ are simply connected manifolds each each containing a nucleus $N(2)$ with its distinguished torus $T$, then
\[
    (X \cs_T Y) \cs \sss \cong (X \cs Y) \cs 2(\sss).
\]
\end{lemma}
A similar argument~\cite{gompf:elliptic} implies that for the log-transformed nucleus, 
\[
N(2;2q+1)\cs(\sss) \cong N(2)\cs(\sss).
\]

A collection of manifolds and spherical families of diffeomorphisms is now constructed based on two ingredients. 
\begin{itemize}
    \item A collection of distinct, 1-stably diffeomorphic smooth structures $V_q$ on an initial manifold $V$, detected by mod $2$ Seiberg-Witten invariants, such that each $V_q$ contains a copy of $N(2)$;
    \item a symplectic manifold  $U_0$ containing two symplectic nuclei $N(2)_a$, $a = 1, 2$. Set $U_1$ to be the result of the $3$-log transform in $N(2)_1$. By Freedman, \cite{freedman} there is a homeomorphism, $\psi_1:U_1 \to U_0$ intertwining the embeddings of $N(2) = N(2)_2$ and by Gompf, \cite{gompf:elliptic} a diffeomorphism $\varphi_1:U_1\cs(\sss) \to U_0\cs(\sss)$ intertwining the embeddings of $N(2)$ so that $\varphi_1$ is homotopic to $\psi_1\cs 1_{\sss}$.
\end{itemize}
 Set $\bu = U_0\cs(\sss)$ and 
    \[
    \bbeta = \vphi_1(1_{U_1} \cs R_{A+B}R_{A-B}) \vphi_1^{-1}(1_{U_0} \cs R_{A+B}R_{A-B})^{-1}.
    \]
    Notice that Proposition~\ref{P:Rstab} establishes that $\beta$ is topologically isotopic to the identity and becomes smoothly isotopic to the identity after one stabilization. Recall that $T$ denotes the torus in $N(2)$, and let $T_1$ be the torus in $N(2)_1$.

We will specify some particular choices for $V$ and $U_0$ in the paragraphs after Theorem~\ref{T:PA}
below, but the construction does not depend on more than the above properties.  
In Section~\ref{S:pi0}, we used $V = V_0 = E(2;1)\cong E(2)$, with further log transforms $V_q$ giving such a collection of smooth structures on $V$. As described there, this leads to an infinite collection of distinct isotopy classes of diffeomorphisms $\balpha^0[q]$ on $V\cs(\sss)$. 
 
 Let $F$, $G$, and $K$ denote particular choices of stable isotopy, topological isotopy, and pseudoisotopy respectively. 

\begin{definition}\label{D:families}
    Set $\bZ^0_{V,\bu} = V\cs(\sss)$ and $\bZ^{p+1}_{V,\bu} = \bZ^p_{V,\bu}\cs_T\bu$ where the sum is taken along the tori in the distinguished $N(2)$ nuclei in the summands.  Define maps
    
       \begin{align*}
        \balpha^{p+1}[q]&\colon  S^{p+1} \times\bZ^{p+1}_{V,\bu}  \to I\times S^{p+1}\times \bZ^{p+1}_{V,\bu}, \\
        F^{p+1}[q]&\colon I\times S^{p+1}\times \bZ^{p+1}_{V,\bu} \cs (\sss) \to I\times S^{p+1}\times \bZ^{p+1}_{V,\bu}\cs(\sss), \\
        G^{p+1}[q]&\colon I\times S^{p+1}\times \bZ^{p+1}_{V,\bu}  \to I\times S^{p+1}\times \bZ^{p+1}_{V,\bu},\ \text{and} \\
        K^{p+1}[q]&\colon I\times S^{p+1}\times \bZ^{p+1}_{V,\bu} \to I\times S^{p+1}\times \bZ^{p+1}_{V,\bu}, \\
    \end{align*}
by the formulas $\balpha^{p+1}[q] = [F^p[q],\bbeta]$, $F^{p+1}[q] = [F^p[q],F]$, $G^{p+1}[q] = [F^p[q],G]$, and $K^{p+1}[q] = [K^p[q],F]$. 
\end{definition}
\begin{remark}\label{R:def-comments}
\leavevmode
\begin{enumerate}
\item  \label{i:uv} There are many other choices for $V$, $\bu$, and $\bbeta$.  The main requirements are for $V$ and $\bu$ are to have suitable surfaces for submanifold sum, and $V$ to have non-vanishing Seiberg-Witten invariant. The requirements for $\bbeta$ are its $1$-stable triviality, and non-triviality of $\sw_{\bZ^0}^{\pi_0}$ as in the proof of Proposition~\ref{P:pi0}.
\item  \label{i:decomp} Lemma~\ref{L:MM} implies that 
    \[
    \bZ^p_{V,\bu} \cong V\cs p\bu\cs(p+1)(\sss).
    \] 
    Note that this implies that $\bZ^p_{V,\bu}$ has (at least) $2p+1$ summands diffeomorphic to $\sss$.
\item If $V$ and $\bu$ are spin, then every manifold in the $\bZ^p_{V,\bu}$ collection will be spin. If $V$ is non-spin, then every manifold in the collection will be non-spin.
\item \label{multi} The statement of Theorem~\ref{T:infgen} says that there are manifolds $\bZ^p$ admitting not only an interesting $p$-dimensional family of diffeomorphisms but also interesting $(p-2i)$-dimensional families. The starting point for this is a symplectic manifold $\bW$ containing a copy of $N(2)$ so that $\bW\cs(\sss)$ splits as $ \bu\cs_T\bu\cs(\sss)$ for an appropriate choice of $\bu$. An example of such a manifold is obtained in Proposition~\ref{P:Wconstruct} below, by modifying a construction of Park~\cite{park:spin}. Then item~\ref{i:decomp} above implies that
$\bZ^{p-2i}_{V\cs_T\bW^{\cs^i_T},\bu} \cong \bZ^p_{V,\bu}$. 
\item 
Varying the initial collection of $4$-manifolds was used in \cite{auckly:internal} to show that after enough stabilizations, any simply connected $4$-manifold would contain exotic surfaces separated by many internal stabilizations in any non-trivial homology class. The same argument will show that after enough stabilizations any simply connected $4$-manifold will support spherical families of diffeomorphisms generating an infinite rank summand in the kernel of the map between the homology group of homeomorphisms to the homology group of diffeomorphisms in a range of degrees just as the manifolds $\bZ^p$ do. A similar construction was used by Konno and Lin in their investigation of homology stability~\cite{konno-lin:instability}.
\item \label{i:variations}
Even though we have stated our main results (Theorems~\ref{T:infgen}, \ref{T:emb2}, and~\ref{T:emb3}, and Corollaries~\ref{C:Hinf} and~\ref{C:HQinf}) for closed-simply connected manifolds, the same construction allows immediate generalization to some manifolds with arbitrary fundamental group and/or non-empty boundary. Indeed, given any finitely presented group $\Gamma$, let $W$ be a closed symplectic manifold with $b^2_+>1$, fundamental group $\Gamma$, and containing a copy of $N(2)$; see~\cite{gompf:symplectic}. The construction  of families of diffeomorphisms on $\bZ^p_{V\cs_TW,\bu}$ proceeds as above, yielding results analogous to those theorems. This is because the supports of $F$, $G$, and $K$ are disjoint from the nucleus where one would sum with $W$. Furthermore, the definition of the invariants that we use does not require the manifolds to be simply connected, and will not change with the addition of a symplectic manifold along a torus contained in a standard nucleus. If $C$ is a compact, codimension zero submanifold of $W$ disjoint from the nucleus, then the same results will hold (rel boundary) for $\bZ^p_{V\cs_T(W \setminus\text{int}C),\bu}$. Since the families of diffeomorphisms are all trivial on the boundary, one can just take the union with $C$ and extend across $C$ as the identity. 
\end{enumerate}
    
\end{remark}

According to Taubes~\cite{taubes:sw-symplectic}, the Seiberg-Witten invariant of the canonical class on a symplectic manifold $V$ with $b^2_+>1$ satisfies $|\text{SW}_V(K)| = 1$. Let us assume that $V$ contains a nucleus $N(2)$,
which in turn contains a torus $T$ of self-intersection zero. Note that the adjunction inequality applied to $T$ and $T+\sigma$ implies that the canonical class satisfies $K\cdot T = K\cdot\sigma = 0$, so that the log transform formulas from~\cite[Theorem 8.5]{Fintushel-Stern:Rat-Blow} and~\cite[Cor 1.4]{morgan-mrowka-szabo:torus} apply. Let $V_{q}$ denote the result of $(2q+1)$-log-transform applied to the torus in the nucleus. Then  
$SW^{\Z_2}_{V_{q} }(K+2\ell\Tc_1) = 1$ precisely when $|\ell| \le q$, where $\Tc_1$ represents
the cohomology class Poincar\'e dual to $[T_1] = [F_{2q+1}]$ and $F_{2q+1}$ is the multiple fiber in $V_1$. The fact that the torus is in $N(2)$ allows one to easily conclude that the log transforms do not change the topological type of the manifold.

When $V$ contains a nucleus such that $T$ is symplectically embedded, then we can do a submanifold sum of $V$ with any other symplectic manifold containing a symplectic torus. 
This means we are interested in having a rich collection of symplectic manifolds to use as a summand in this fashion.

Now we give some specific choices for the manifolds $V$, $\bu$, and $\bW$ satisfying the conditions above. Our building blocks  are simply-connected symplectic manifolds constructed by Park~\cite{park:spin}. His manifolds are spin and fill out a portion of the integer points in the plane parameterized by $(\chi = \frac{\sigma + e}{4},{\bf c} = 3\sigma + 2e)$ where $\sigma$ and $e$ are the signature and Euler characteristic, respectively. We find it convenient to make some small modifications and do a small change of coordinates and to state the result in terms of the rank and signature, as in the first author's paper~\cite{auckly:internal}.
\begin{theorem}[Theorem 4.2 of \cite{auckly:internal}]\label{T:PA} 
Let $Q$ be a unimodular form over the integers with rank $r(Q)$, signature $\sigma(Q)$.  If $Q$ is even, then assume that $\sigma(Q) \equiv 0 \pmod{16})$. 
There is a constant $r_*$ so that if $r(Q) > r_*$,
% \begin{itemize}
%     \item ;
%     \item 
% \end{itemize} and
$|\sigma(Q)| \le \frac{3}{13}r(Q)$, and $b^+(Q)$ is odd, then $Q$ is represented by the intersection form of a simply-connected, closed, symplectic, 
$4$-manifold $X_Q$ containing a self-intersection zero symplectic torus in an $N(2)$ nucleus. Moreover, if $Q$ is odd, then $X_Q \cs \sss$ is diffeomorphic to $a\cptwo\cs b\cptwobar$ for appropriate $a,b$. If $Q$ is even, then $X_Q \cs \sss$ is diffeomorphic to $aE(2)\cs b(\sss)$ for appropriate $a,b$ where negative values of $a$ refer to $|a|$ summands of $E(2)$ with the anti-holomorphic orientation. 
\end{theorem}
This is basically a reformulation and modest extension of the main result, Theorem 1, in~\cite{park:spin}. The additional information from~\cite{auckly:internal} is the odd case, the fact that the manifolds constructed contain an $N(2)$ nucleus with a symplectically embedded torus, and that the manifolds decompose as stated after summing with $\sss$. This will be important to us in the inductive construction of families of diffeomorphisms. It is also notable that the manifolds constructed in this way include manifolds homotopy equivalent to $\cs^{2n-1} (\sss)$ for all sufficiently large $n$. This will allow us to construct the first examples of non-isotopic PSC metrics (and exotic families of such metrics) on simply connected spin 4-manifolds. (For certain non-simply connected $4$-manifolds, this was done in~\cite{mrowka-ruberman-saveliev:periodic-index}.) 
The constant $r_*$ appearing in Theorem~\ref{T:PA} depends on the number of $\sss$ summands needed to make certain manifolds from~\cite{park:spin} completely decompose.

\begin{proposition}\label{P:Wconstruct}
    There are manifolds $V$ and $\bu$ satisfying conditions \eqref{i:uv} and \eqref{i:decomp} of Remark~\ref{R:def-comments}, and a spin symplectic manifold $\bW$ containing a copy of $N(2)$ so that $\bW\cs(\sss) \cong \bu\cs_T\bu\cs(\sss)$.
\end{proposition}
\begin{proof}
    By taking the torus sum of $E(2)$ with one of the Park manifolds with signature $16$, one will arrive at a manifold of signature zero that contains a pair of disjoint nuclei. Denote the manifolds of this form that are homeomorphic to $(2n+1)\sss$ by $P_{2n+1}$. These exist provided $n$ is sufficiently large. Set $V = E(2)\cs_TP_{2n+1}$, $\bu = V\cs(\sss)$, and $\bW = E(4)\cs_TP_{4n+5}$. Notice~\cite{gompf-mrowka} that $E(2)$ has three disjoint $N(2)$ nuclei. Use one to perform the submanifold sum with $P_{2n+1}$, use a different one to perform the $(2q+1)$-log transforms to obtain the exotic smooth structures $V_q$, and reserve the last one as the distinguished nucleus. 
 \end{proof}
 These choices lead to exotic diffeomorphisms $\balpha^0[q]$ defined on $\bu$ as in section~\ref{S:pi0}. Finally,  Lemma~\ref{L:MM} implies that
\[
\bu\cs_T\bu \cong V\cs_TV\cs 2(\sss) \cong E(4)\cs_T2P_{2n+1}\cs 2(\sss) \cong \bW\cs (\sss)
\]
so that item \eqref{multi} in Remark~\ref{R:def-comments} can be applied.  

\section{Calculation of the invariants}\label{S:calculation}

In this section we present the fundamental calculation
 in this paper, which allows us to pass from $SW_Z^{\pi_k,\Z_2}$ applied to an $S^k$ family in the diffeomorphism group of a manifold $Z$ to $SW_Z^{\pi_{k+1},\Z_2}$ applied to an $S^{k+1}$ family of diffeomorphisms on $\diff(Z \cs (\sss))$ created by the commutator construction as in Definition~\ref{D:families}. In principle, this would work for the invariants  $SW_Z^{\pi_*,\Z}$ but this would require some additional care with orientations.

Note that $\sss$ has a distinguished $\Spinc$ structure $\spincs_0$ corresponding to its unique spin structure. For a $\Spinc$ structure $\spincs$ on $Z$, we write $\spincs \cs \spincs_0$ for the unique $\Spinc$ structure  on $Z\cs(\sss)$ that restricts to the spin structure $\spincs_0$ on the $\sss$ summand and to $\spincs$ on $Z$. When $H_1 = 0$, so that the first Chern class gives a bijection between $\Spinc$ structures and characteristic cohomology classes in $H^2$, this corresponds to the natural splitting $H^2(Z\cs(\sss)) \cong H^2(Z)\oplus \Z^2$.

Theorem~\ref{pIRglue} below leads to the new computational formula that allows for the computation of the Seiberg-Witten invariants for certain families of diffeomorphisms. The families in Definition~\ref{D:families} are compositions of families that live on a connected sum $N_1 \cs N_2$, and are built out of families on $N_1$ and $N_2$.  We need some preliminaries in order to state the theorem. First, we make the convention that a hat accent is used to denote a cylindrical end manifold obtained by removing a point from a closed manifold and adjusting the metric appropriately in the (punctured) neighborhood of that point. The \SW equations make sense in the cylindrical end setting, and will be discussed in Section~\ref{an-frame}. In particular, we will denote the resulting family moduli space on the manifold $\hat{N}$ by ${\mathcal{M}}_{\hat N,\spincs}(\{\data\})$.

The proof of the gluing theorem is a bit easier if the metric portion of the data is constant near a data point where gluing occurs.  Recall from Definition~\ref{D:exceptional} that a point  $\theta \in \Xi$ is exceptional if the moduli space $\M_\spincs(\data_\theta)$ is non-empty. 
\begin{definition}\label{D:local-g}
Data $\{\data_\theta\}_{\theta\in\Xi}$ is \emph{locally metric independent} if every exceptional $\theta\in\Xi$ has a neighborhood $U$ such that so that $g_\vartheta$ is independent of $\vartheta\in U$. 
\end{definition}
In Section~\ref{s:constant}, we will show how, in the setting we need, to take given data and deform it to be locally metric independent.

It is useful to recall the classification of reducible solutions and the definition of walls. A reducible solution is one for which $\psi = 0$, which readily implies that $F^+_A + i\eta^+ = 0$. Since the first Chern class of the $\Spinc$-structure is represented by $c = \frac{i}{2\pi}F^+_A$, we see that reducibles only exist when $2\pi\eta^+ = c^+$. Notice that this condition depends on both the metric and perturbation. The corresponding wall is given by
\[
\mathcal{W}_c = \{\data\,|\, 2\pi\eta^+ = c^+\}.
\]

When the data is locally metric independent, we may restrict the parameter neighborhoods and assume that the metrics on $\hat{N}_i$ are constant. Recall that a slice for the gauge group can be defined by the gauge-fixing condition $d^*A = 0$, and that the $*$-operator depends on the metric.  Combined with the observation that the action of the gauge group does not change the parameters, we see that in a metric independent neighborhood, this gauge-fixing condition is independent of the parameter.
% \danny{Dave--how does this seem?}
Furthermore, once we assume the metrics are constant, we can take the domain and codomain of the gauge-fixed Seiberg-Witten map to be fixed linear spaces.

\begin{definition}\label{irr-red-good}
    We say that a pair of family data $\data^i:\Xi_i\to \widehat{\Pi}(\hat N_i)$, $i=1,2$ is \emph{irreducible-reducible good} if the following conditions are met.
    \begin{enumerate}
        \item $\data^1$ is locally metric independent data. In particular,
        all $[\theta,A,\psi]\in {\mathcal{M}}_{\hat N_1,\spincs_1}(\{\data^1\})$ are irreducible, isolated, and parameterized regular,
        \item $\data^2$ is locally metric independent data, and $\data^2 \pitchfork \mathcal{W}_{c_1(\spincs_2)}$,    
        \item $b^2_+(\hat N_1) > \text{dim}(\Xi_1)$,
        \item $\vdim(\mathcal{M}_{\hat N_2,\spincs_2}(\{\data^2\})) = -1$,
        \item ${\mathcal{M}}_{\hat N_2,\spincs_2}(\{\data^2\})$ finite,
        \item  \label{i:non-degenerate} $[A,\psi,\theta]\in {\mathcal{M}}_{\hat N_2,\spincs_2}(\{\data^2\})$ are reducible with $H^1_{\theta,A,\psi} = 0$.
    \end{enumerate}
\end{definition}

\smallskip 
\begin{remark}
    ~\smallskip
    \begin{enumerate}
       \item The virtual dimension of the parameterized moduli space $\vdim(\mathcal{M}_{\hat N,\spincs}(\{\data\}))$
        is by definition the index of the deformation complex. It is related to the virtual dimension of the moduli space over a fixed data point by
        \[
        \vdim(\mathcal{M}_{\hat N,\spincs}(\{\data\})) = \vdim(\mathcal{M}_{\hat N,\spincs}(\data_0)) + \dim \Xi.
        \]
        \item We will see that the assumption on $\vdim(\mathcal{M}_{\hat N_2,\spincs_2}(\{\data^2\}))$
        % \danny{Added  `2' sub and superscripts to $\vdim(\mathcal{M}_{\hat N,\spincs}(\{\data\}))$.}
        implies that the gluing problem is unobstructed. If the assumption was weakened to 
        $\vdim(\mathcal{M}_{\hat N_2,\spincs_2}(\{\data^2\})) \le -1$, the standard argument would, as outlined in \cite[Section 6]{baraglia-konno:gluing}, would imply that the moduli space was isomorphic to the zeros of the Kuranishi map. 
        \item The existence of an irreducible-reducible good pair of family data in case $\Xi_1$ is a point \cite{baraglia-konno:gluing} is established in \cite[Proposition~7.2]{baraglia-konno:gluing}. This same result establishes the conditions we need for $\hat{N}_2$.
    \end{enumerate}
\end{remark}

\begin{theorem}[Parameterized irreducible-reducible gluing]\label{pIRglue}
    Let $\data^i:\Xi_i\to \widehat{\Pi}(\hat N_i)$, $i=1,2$ be irreducible-reducible good, then
    there is an $r_0$ so that for every $r>r_0$ 
    \[{\mathcal{M}}_{\hat N_1\cs_r\hat N_2,\spincs_1\cs_r\spincs_2}(\{\data^1\cs_r\data^2\})
    \cong \widehat{\mathcal{M}}_{\hat N_1,\spincs_1}(\{\data^1\})\times_{S^1}\widehat{\mathcal{M}}_{\hat N_2,\spincs_2}(\{\data^2\}),
    \]
    consists of parameterized regular, irreducible configurations.
\end{theorem}

\begin{remark} 
~
\begin{enumerate}
    \item     One expects an even stronger irreducible-reducible gluing result generalizing Nicolaescu's degenerate gluing result \cite[Theorem 4.5.19]{nicolaescu:swbook}, and the Baraglia-Konno gluing theorem \cite{baraglia-konno:gluing}. It should not be necessary to assume that the solutions on the irreducible side are isolated, or that the first cohomology  of the deformation space on the reducible side vanishes. In this case, one should conclude that there is an $r_0$ so that for every $r>r_0$ ${\mathcal{M}}^o_{\hat N_1\cs_r\hat N_2,\spincs_1\cs_r\spincs_2}(\{\data^1\cs_r\data^2\})$
    consists of parameterized regular, irreducible configurations. Furthermore, it is orientation-preserving diffeomorphic to the set of zeros of a transverse section of the obstruction bundle 
    \[
    \Psi_r: \widehat{\mathcal{M}}^o_{\hat N_1,\spincs_1}(\{\data^1\})\times_{S^1}\widehat{\mathcal{M}}^o_{\hat N_2,\spincs_2}(\{\data^2\}) \to \mathcal{V}_r.
    \]
    \item Since all of the invariants we construct arise from zero-dimensional parameterized moduli spaces, we do not establish this more general gluing theorem. While we do need to know that there are $\Z$-valued invariants that do not vanish while others do, we do not need to consider orientations in our gluing theorem since the non-vanishing of the $\Z_2$ invariants implies the non-vanishing of the $\Z$-invariants. We expect that there is a version of the gluing theorem that takes orientations into account and computes the $\Z$-valued invariants.
\end{enumerate}
\end{remark}

This theorem includes the the following as special cases: the blow-up formula for manifolds of simple type, the formulas from~\cite{ruberman:isotopy,ruberman:swpos} relating $1$-parameter invariants of some special diffeomorphisms to the Seiberg-Witten invariants of the seed manifolds, and the following corollary.

\begin{corollary}\label{glue}
Let $\alpha\colon S^k\to\diff(\widehat{Z},D^4)$ and let $R_{A+B}R_{A-B}\colon \sss\to\sss$ be the standard reflection diffeomorphism described in Section~\ref{S:pi0}.
Let 
$\{\data^Z_{s,\theta}\}_{(s,\theta)\in I\times S^k}$ 
be a good family of locally metric independent data on $\widehat{Z}$, and let 
$\{\data^{\sss}_{t}\}_{t\in I}$ 
be the standard data of Lemma~\ref{L:standardsssdata}  with $(R_{A+B}R_{A-B})^*\data^{\sss}_{0} = \data^{\sss}_{1}$. 
\[
\#\calm_{Z\cs(\sss)}(\{\data^{Z\cs(\sss)}_{s,t,\theta}\},\spincs\cs\spincs_0) \equiv \#\calm_{Z}(\{\data^Z_{t,\theta}\},\spincs) \pmod{2}. 
\]
\end{corollary}

This corollary will be proved in Section~\ref{1st-apps} at the end of the paper; here we used it to to prove the following theorem, which we call the \emph{Suspension Theorem}, because it relates $\pi_k(\diff)$ for one manifold to $\pi_{k+1}(\diff)$ for a stabilized manifold. In the course of the proof, we will make use of some technical results from Section~\ref{glue-proof}.
\begin{theorem}\label{anti-hol} Let $\alpha\colon S^k\to\diff(Z,D)$ be a family such that $\alpha\cs 1_{(\sss)}$ smoothly contracts via a contraction $F$.
We then have
\[
SW_{Z\cs(\sss)}^{\pi_{k+1},\Z_2}(F(1_Z\cs (R_{A+B}R_{A-B}))(F)^{-1},\spincs\cs\spincs_0) =  SW_Z^{\pi_{k},\Z_2}(\alpha,\spincs).
\] 
\end{theorem}

\begin{proof}
In the proof, all counts of points in moduli spaces should be taken mod $2$. Let $\data^Z_*$ be good on $Z$ and let $\{\data^Z_{s,\theta}\}_{I\times S^k}$ be good with $\data^Z_{0,\theta} =\data^Z_*$ and  $\data^Z_{1,\theta}=\alpha(\theta)^*\data^Z_*$.

Since $\alpha$ fixes $D$, there is an induced family, which we denote by the same letter, on $\hat{Z}$. Let $\data^{\widehat{Z}}_*$ be good on $\hat{Z}$, and
let $\data^{\widehat{Z}}_*$ be good on $\widehat{Z}$ with $\data^{\widehat{Z}}_{0,\theta} =\data^{\widehat{Z}}_*$ and  $\data^{\widehat{Z}}_{1,\theta}=\alpha(\theta)^*\data^{\widehat{Z}}_*$. In  Proposition~\ref{SPNMID}, we will show that there is locally metric independent data satisfying the same assumptions. In Corollary~\ref{closed-deleted}, we will relate the family moduli spaces on $Z$ and $\widehat{Z}$ and show that
\[
SW_Z^{\pi_{k},\Z_2}(\alpha) = \#\calm_{Z}(\{\data^Z_{s,\theta}\}_{I\times S^k},\spincs) = \#\calm_{\widehat{Z}}(\{\data^{\widehat{Z}}_{s,\theta}\}_{I\times S^k},\spincs).
\]
Let
$\{\data^{\widehat{\sss}}_{t}\}_{t\in I}$ be the standard data (constructed in Lemma~\ref{L:standardsssdata}) with $(R_{A+B}R_{A-B})^*\data^{\widehat{\sss}}_{0}= \data^{\widehat{\sss}}_{1}$.
Let $F$ be the smooth contraction and consider $\{(F_{s,\theta}^{-1})^*\data^{Z\cs(\sss)}_{s,t,\theta}\}_{I \times I \times S^k}$. We show that this represents a family of data 
on a spherical shell $I\times S^{k+1}$ of the form required to compute $SW_{Z\cs(\sss)}^{\pi_{k+1},\Z_2}(F(1_Z\cs (R_{A+B}R_{A-B}))(F)^{-1}, \spincs\cs\spincs_0)$. 

Since $F_{0,\theta} = 1_{Z\cs (\sss)}$ and $\data^Z_{0,\theta} = \data^Z_* = (g^Z_*,\eta^Z_*)$ we have
\[
(F_{0,\theta}^{-1})^*\data^{Z\cs(\sss)}_{0,t,\theta} = \data^Z_*\cs \data^{\sss}_t.
\]
Now, $F_{1,\theta} = \alpha(\theta)\cs 1_{\sss}$ and $\data^Z_{1,\theta} = \alpha(\theta)^*\data^Z_*$. Thus,
\begin{align*}
(F_{1,\theta}^{-1})^*\left(\data^Z_{1,\theta}\cs \data^{\sss}_t\right) &= 
\left((\alpha(\theta)\cs 1_{\sss})^{-1}\right)^*(\alpha(\theta)^*\data^Z_{0,\theta}\cs \data^{\sss}_t) \\
&= \data^Z_*\cs \data^{\sss}_t.
\end{align*}
Thus $(F_{s,\theta}^{-1})^*\data^{Z\cs(\sss)}_{0,t,\theta}$ is really defined on the spherical shell obtained by collapsing $\{0\} \times I \times S^k$ to $\{0\} \times I\times \text{point}$ and $\{1\} \times I\times S^k$ to $\{1\} \times I\times \text{point}$. Finally, we compute

%\begin{small}
\begin{align*}
(F_{s,\theta}^{-1})^*\data^Z_{s,\theta}\cs \data^{\sss}_1 
&= (F_{s,\theta}^{-1})^*(1_Z\cs (R_{A+B}R_{A-B}))^*\data^Z_{s,\theta}\cs \data^{\sss}_0 \\
&= (F_{s,\theta}^{-1})^*(1_Z\cs (R_{A+B}R_{A-B}))^*(F_{s,\theta})^*(F_{s,\theta}^{-1})^*\data^Z_{s,\theta}\cs \data^{\sss}_0 \\
&= (F_{s,\theta}(1_Z\cs (R_{A+B}R_{A-B}))F_{s,\theta}^{-1})^*(F_{s,\theta}^{-1})^*\data^Z_{s,\theta}\cs \data^{\sss}_0. 
\end{align*}
%\end{small}

Thus naturality, Corollary~\ref{glue}, and the definition give
\[
\begin{aligned}
SW_{Z\cs(\sss)}^{\pi_{k+1},\Z_2}(F(1_Z\cs & (R_{A+B}R_{A-B}))(F)^{-1},\spincs\cs\spincs_0) \\ &= \#\left(\calm_{Z\cs(\sss),\spincs}(\{(F_{s,\theta}^{-1})^*\data^Z_{s,\theta}\# \data^{\sss}_t\}\right) \\
&= \#\left(\calm_{Z\cs(\sss),\spincs}(\{\data^Z_{s,\theta}\# \data^{\sss}_t\}\right) \\
&= \#\left(\calm_{\widehat{Z},\spincs}(\{\data^{\widehat{Z}}_{s,\theta}\})\right) \\
&=  SW_{Z}^{\pi_{k},\Z_2}(\alpha,\spincs).
\end{aligned}
\] 
\par\nopagebreak\vspace{-1.5\baselineskip}\mbox{}
\end{proof}

We now use the Suspension Theorem to compute the invariants for families of diffeomorphisms on $\bZ^{p+1}_{V,\bu}$. These are constructed recursively, but by a more elaborate scheme than just one stabilization at a time. Using the notation established in Section~\ref{S:blocks}, this means that we need to compute the invariants for families 
of diffeomorphisms on $U_0\cs_T\bZ^p_{V,\bu} = \bZ^p_{V\cs_TU_0,\bu}$ and $U_1\cs_T\bZ^p_{V,\bu} = \bZ^p_{V\cs_TU_1,\bu}$. With this in mind let 
\begin{equation}\label{E:Zrs}
    \bZ^p_{r,s} = \bZ^p_{V\cs_TrU_0\cs_TsU_1,\bu}.
\end{equation}
To start the recursion one needs the Seiberg-Witten invariants of $V\cs_TU_0^{r\cs_T}\cs_TU_1^{s\cs_T}$. Let $K$ denote the canonical class for the symplectic structure
of $V\cs_TU_0^{(r+s)\cs_T}$ constructed from the fiber sum, so that $|\text{SW}_{V\cs_TU_0^{(r+s)\cs_T}}(K)| = 1$. To generate  $V\cs_TU_0^{r\cs_T}\cs_TU_1^{s\cs_T}$ we perform $3$-log-transforms on tori in nuclei in $s$ of the summands of $V\cs_TU_0^{(r+s)\cs_T}$. Label the {symplectic} torus in an $N(2)$ nucleus in each summand of $U_0^{(r+s)\cs_T}$ by $T^a$, $a = 1, \cdots, r+s$, and let $\Tc^a$ denote both the corresponding cohomology class and $\Spinc$ structure.  The adjunction inequality implies that 
$\text{SW}_{V\cs_TU_0^{(r+s)\cs_T}}(K+2\Tc^a) = 0$, for $a = 1, \cdots, r+s$. Let $\Sigma\Tc = \sum_{a=r+1}^{r+s}\Tc^a$.
By the log-transform formula~\cite{Fintushel-Stern:Rat-Blow,morgan-mrowka-szabo:torus},  $SW^{\Z_2}_{V\cs_TU_0^{r\cs_T}\cs_TU_1^{s\cs_T} }(K+2\Sigma\Tc) = 1$, and it would be zero by the adjunction inequality if one performed fewer than the $s$ log-transforms leading to $V\cs_TU_0^{r\cs_T}\cs_TU_1^{s\cs_T}$.

\begin{theorem}\label{T:Zinfty}
The family Seiberg-Witten invariants of the spherical families $\balpha^p$ of diffeomorphisms satisfy 
\[
SW_{\bZ^{p+1}_{r,s}}^{\pi_{p},\Z_2}(\balpha^{p}[q],K+2\Sigma\Tc+2\ell \Tc_1) = 1,
\] 
if $0 < |\ell | \leq q$, and for each $q$ and $p$ there is a $Q_{q,p}$ so that $|\ell| > Q_{q,p}$ implies that 
\[
SW_{\bZ^{p+1}_{r,s}}^{\pi_{p},\Z_2}(\balpha^{p}[q],K+2\Sigma\Tc+2\ell \Tc_1) = 0.
\] 
It follows that 
\[
\text{ker}\left(H_p(\diffo(\zz^p)) \to H_p(\homeo(\zz^p)) \right)
\]
admits an infinite rank summand for $p>0$. 
\end{theorem}

\begin{proof}
To start with, note that the topological contraction $G^{p}$ described in Definition~\ref{D:families} show that $\balpha^p$ is trivial in $\pi_p(\homeo(\zz^p))$ and hence in $H_p(\homeo(\zz^p))$ as well. 

The first statement follows by induction via the Suspension Theorem and the following calculation. Most of the steps are formal, but the equalities which require some explanation are marked $=^{(*)}$ and discussed below. 

\begin{small}
\begin{align*}
SW_{\bZ^{p+1}_{r,s}}^{\pi_{p+1},\Z_2}&(\balpha^{p+1}[q],K+\Sigma_{r,s}\Tc+2\ell \Tc_1) = SW_{\bZ^{p+1}_{r,s}}^{\pi_{p+1},\Z_2}([F^p[q],\bbeta],K+2\Sigma_{r,s}\Tc+2\ell \Tc_1) \\
&= SW_{\bZ^{p+1}_{r,s}}^{\pi_{p+1},\Z_2}([F^p[q],~^{\vphi_1}(1_{U_1} \cs R_{A+B}R_{A-B}) (1_{U_0} \cs R_{A+B}R_{A-B})^{-1}],K+2\Sigma_{r,s}\Tc+2\ell \Tc_1) \\
&=^{(1)} SW_{\bZ^{p+1}_{r,s}}^{\pi_{p+1},\Z_2}(F^p[q]~^{\vphi_1}(1_{U_1} \cs R_{A+B}R_{A-B}) (1_{U_0} \cs R_{A+B}R_{A-B})^{-1}(F^p[q])^{-1},K+2\Sigma_{r,s}\Tc+2\ell \Tc_1) \\
&=^{(2)} SW_{\bZ^{p+1}_{r,s}}^{\pi_{p+1},\Z_2}(F^p[q]~^{\vphi_1}(1_{U_1} \cs R_{A+B}R_{A-B})(F^p[q])^{-1},K+2\Sigma_{r,s}\Tc+2\ell \Tc_1) \\
&\quad - SW_{\bZ^{p+1}_{r,s}}^{\pi_{p+1},\Z_2}(F^p[q](1_{U_0} \cs R_{A+B}R_{A-B})(F^p[q])^{-1},K+2\Sigma_{r,s}\Tc+2\ell \Tc_1) \\
&=^{(3)} SW_{\bZ^{p}_{r,s+1}}^{\pi_{p},\Z_2}(\balpha^{p}[q],K+2\Sigma_{r,s}\Tc+2\ell \Tc_1)  - SW_{\bZ^{p}_{r+1,s}}^{\pi_{p},\Z_2}(\balpha^{p}[q],K+2\Sigma_{r,s}\Tc+2\ell \Tc_1) \\
&=^{(4)} SW_{\bZ^{p}_{r,s+1}}^{\pi_{p},\Z_2}(\balpha^{p}[q],K+2\Sigma_{r,s}\Tc+2\ell \Tc_1) . 
\end{align*}
\end{small}
\noindent
The last four steps are explained as follows. 
\begin{enumerate}
    \item  is simply the composition law for the SW invariants combined with the observation that $[a,b] = aba^{-1}\cdot b^{-1}$, 
    and $b^{-1} = (1_{U_0} \cs R_{A+B}R_{A-B})~^{\vphi_1}(1_{U_1} \cs R_{A+B}R_{A-B})^{-1}$ is the constant family, so has trivial invariants.
    \item  is also the composition law after writing  
    \begin{align*}
          &F^p[q]~^{\vphi_1}(1_{U_1} \cs R_{A+B}R_{A-B}) (1_{U_0} \cs R_{A+B}R_{A-B})^{-1}(F^p[q])^{-1} \quad \text{as}\\ 
  &F^p[q]~^{\vphi_1}(1_{U_1} \cs R_{A+B}R_{A-B})(F^p[q])^{-1} F^p[q] (1_{U_0} \cs R_{A+B}R_{A-B})^{-1}(F^p[q])^{-1} .
    \end{align*}
    It is crucial here that we are working with $\Z_2$ coefficients, because the SW invariants of the individual factors may not be defined over the integers. 
    \item comes from the Suspension Theorem~\ref{anti-hol} along with the following observations. 
\begin{align*}
    &\vphi_1^{-1}F^p[q]~^{\vphi_1}(1_{U_1} \cs R_{A+B}R_{A-B})(F^p[q])^{-1}\vphi_1  \\ &= \vphi_1^{-1}F^p[q]{\vphi_1}(1_{U_1} \cs R_{A+B}R_{A-B})\vphi_1^{-1}(F^p[q])^{-1}\vphi_1.
   \end{align*}
Furthermore, for any diffeomorphism $f:X\to Y$ and family of diffeomorphisms \hfill\newline $\beta:S^p\to\diff(Y)$ one has
$SW_X(f^{-1}\beta f,f^*\spincs) = SW_Y(\beta,\spincs)$ by naturality. This observation is applied in the second-to-last line with 
$f =\vphi_1\colon \bZ^{p}_{r,s+1}\cs(\sss) \to \bZ^{p}_{r+1,s}\cs(\sss)$.
\item   In the last line the manifold $\bZ^p_{r+1,s}$ has fewer $U_0$ summands than $\Sigma\Tc$ has torus summands, so the second term on line (3) vanishes by the adjunction inequality. 

\end{enumerate}
  % In the last line the manifold $\bZ^p_{r+1,s}$ has fewer $\bX_0$ summands than $\Sigma\Tc$ has torus summands, so the invariant vanishes by the adjunction inequality.

The families of diffeomorphisms that we consider are all in $\tdiff$ so the integer-valued invariants are defined and have the same parity as the $\Z_2$ invariants. It follows that we have many families with non-vanishing integral invariants. While the $\Z_2$ computation will not demonstrate that the integer invariant vanishes, we know that it will vanish in some range by Lemma~\ref{L:finite}. This gives the constants $Q_{q,p}$ from the theorem statement.

To conclude that there is an infinite rank summand of \[
\ker\left(H_p(\diffo(\zz^p)) \to H_p(\homeo(\zz^p)) \right),
\]
notice that the homological invariant agrees with the homotopy invariant when both are defined:
\[
SW_{\bZ^{p}}^{H_{p},\Z}(\balpha^{p}[q],K+2\Sigma\Tc+2\ell \Tc_1) = 
SW_{\bZ^{p}}^{\pi_{p},\Z}(\balpha^{p}[q],K+2\Sigma\Tc+2\ell \Tc_1).
\] 
In addition, by viewing the homological invariant as taking values in integer-valued maps, 
\[
\underline{SW}^{H_p,\Z}\colon \tdiff\to\text{Maps}(\Spinc(\bZ^p),\Z), 
\]
we see that the image of the invariant contains an infinite independent set. Define a sequence of $q$-values, starting with $q_1^p = 1$ and then $q_{i+1}^p = 1+ \text{min}\{q_1^p, Q_{q_i^p,p}\}$. It follows that $\underline{SW}^{H_p,\Z}(\balpha^p_{q_i^p})$ is an independent set. For if $\sum_{i=1}^N a_i\underline{SW}^{H_p,\Z}(\balpha^p_{q_i^p}) = 0$, then applying this to the basic class $K+2\Sigma\Tc+2q_N^p \Tc_1$ would imply that $a_N=0$. Thus the image of $\underline{SW}^{H_p,\Z}$ is an infinite rank subgroup of $\Z^\infty$. Since it is free, there is a splitting mapping from the image back to $\text{ker}\left(H_p(\diffo(\zz^p)) \to H_p(\homeo(\zz^p)) \right)$, proving that it has an infinite rank summand.
\end{proof}

With this calculation in hand, we can now establish the main result of the paper. It still depends on the gluing result Corollary~\ref{glue}, whose proof will be presented in Section~\ref{glue-proof}.
\begin{proof}[Proof of Theorem~\ref{T:infgen}]
The theorem for $p=0$ is proved in Section~\ref{S:pi0}, so we now assume that $p>0$. The proof of Theorem~\ref{T:Zinfty} gives a splitting from a free subgroup of $\text{Maps}(\Spinc(\bZ^p),\Z) \cong \Z^\infty$ back to $\ker\left(H_p(\diffo(\zz^p)) \to H_p(\homeo(\zz^p)) \right)$, and hence a summand in that group. Since the image of that splitting consists of spherical elements, we also get a free summand of $\ker\left(\pi_p(\diffo(\zz^p)) \to \pi_p(\home(\zz^p)) \right)$. Using the isomorphisms $\partial\colon \pi_{p+1}(BG) \overset{\cong}{\rightarrow}\pi_p(G)$ and $G =\diffo(\zz^p)$ or $\homeo(\zz^p)$ from the homotopy long exact sequence for $EG \to BG$, we get the same statement for 
\[
\ker\left[\pi_{p+1}(\bdiffo(\zz^p)) \to \pi_{p+1}(\bhomeo(\zz^p)) \right].
\]
The fact that these summands are detected by Konno's characteristic classes $\swc$ means that these classes also give a summand in 
\[\ker\left[H_{j+1}(\btdiff(\zz^p))\to H_{j+1}(\bhome(\zz^p))\right]. \]

This is the statement of Theorem~\ref{T:infgen} for homotopy or homology groups in the single dimension $j=p$. To obtain the full statement of the theorem with summands for all $j\leq p$ with $j \equiv p \pmod{2}$, write $j = p-2i$ and recall the manifold $\bW$ from item (\ref{multi}) of Remark~\ref{R:def-comments}. From that remark $\bZ^p_{V,\bu} \cong\bZ^{p-2i}_{V\cs_T\bW^{\cs^i_T},\bu}$, so by the argument above there is a $\Z^\infty$ summand in $\ker\left(H_j(\diffo(\zz^p)) \to H_j(\homeo(\zz^p)) \right)$. As above, this yields summands in the other groups as well.\end{proof}

\section{Consequences of the main theorem}\label{S:corollaries}
% \begingroup
% \renewcommand{\thesection}{1}
% \stepcounter{corollary}
% \stepcounter{corollary}
% \stepcounter{corollary}
In this section, we recall the statements of Corollaries~\ref{C:stabilize}-\ref{C:bundle}  and Corollary~\ref{C:MMM} stated in the introduction and give their proofs.
% \begin{corollary}%\label{C:stabilize}
%     The spherical families in Theorem~\ref{T:infgen} are in the image of $\diff(\bZ^p,B^4)$ and a $\Z^\infty$ subgroup of the claimed summand consists of $1$-stably trivial elements. 
% \end{corollary}
\begin{proof}[Proof of Corollary~\ref{C:stabilize}]
We need to show that the spherical families in Theorem~\ref{T:infgen} are in the image of $\diff(\bZ^p,B^4)$ and a $\Z^\infty$ subgroup of the claimed summand consists of $1$-stably trivial elements.

    Both parts follow from the formulas given in Definition~\ref{D:families}. As shown in Proposition~\ref{P:Rstab}, the initial choice of diffeomorphism $\balpha^0$ can be assumed to be the identity on a ball $B^4$, and the stable isotopy $F^0$ can be assumed to be the identity on that same $B^4$. Note that a commutator as described in Section~\ref{S:comm} of maps that are the identity on $B^4$ will also be the identity on $B^4$.  By induction, it follows that for any choices of $V$ and $\bu$, the families $\balpha^{p}[q]$ and stable contractions $F^{p}[q]$ are the identity on $B^4$. In particular, the $F^{p}$s show  any combination of $\balpha^p[q]$ will be stably trivial. 
\end{proof}

Recall that any $\alpha \in \pi_p(\diff(Z))$ determines $\widehat{\alpha}\colon S^p \times Z \overset{\cong}{\rightarrow} S^p \times Z$, which is an element of the block diffeomorphism group $\pi_p(\diffb(Z))$, and that Corollary~\ref{C:pseudo} asserts that $\widehat{\alpha}$ is trivial.
\begin{proof}[Proof of Corollary~\ref{C:pseudo}]
    An extension of $\widehat{\alpha}$ is provided by the diffeomorphism $K^p$ from Definition~\ref{D:families}, where the case $p=0$ is the pseudoisotopy in Equation~\ref{E:FGK0}.
\end{proof}
Corollary~\ref{C:bundle} states that for each non-trivial $\alpha$ constructed in proving Theorem~\ref{T:infgen}, the associated bundle $S^{p+1} \times_\alpha Z$ satisfies
     \begin{enumerate}
        \item the bundle is smoothly non-trivial;
        \item the bundle is topologically trivial;
        \item the total space of the bundle is diffeomorphic to a product.  
     \end{enumerate}
% \begin{corollary}%\label{C:bundle}
%      For each non-trivial $\alpha$ constructed in proving Theorem~\ref{T:infgen}, the associated bundle $S^{p+1} \times_\alpha Z$ satisfies the following:
%      \begin{enumerate}
%         \item the bundle is smoothly non-trivial;
 %         \item the bundle is topologically trivial;
%         \item the total space of the bundle is diffeomorphic to a product.  
%      \end{enumerate}
% \end{corollary}
\begin{proof}[Proof of Corollary~\ref{C:bundle}]
    The first item is the standard correspondence between isomorphism classes of bundles over a sphere and homotopy classes of their clutching functions. The fact that $\alpha$ is trivial in $\homeo(Z)$ implies the second item. 
    
    For the third, recall that $S^{p+1} \times_\alpha Z$ is obtained by 
    gluing $(\theta,z) \in \partial D_+^{p+1} \times Z$ to $(\theta, \alpha (\theta)(z))\in \partial D_-^{p+1} \times Z$.  Choose a collar neighborhood $S^p \times [0,1]$ of $\partial D_+^{p+1}$ with $S^p \times {1}$ corresponding to the boundary. Define a map 
    $S^{p+1} \times_\alpha Z \to  S^{p+1} \times Z$ as the identity on $D_-^{p+1} \times Z$, the identity on the complement of the collar in $D_+^{p+1} \times Z$, and $K^p$ on the collar. It is readily checked that this respects the gluing maps and is a diffeomorphism.
\end{proof}
Finally we show that for the manifolds $\bZ^p$ in Theorem~\ref{T:infgen}, Konno's classes $\swc_{\bZ^p}$ are not in the subring of $H^*(\btdiff(\bZ^p);\Q)$ generated by the Miller-Morita-Mumford classes in $H^*(\bdiff(\bZ^p);\Q)$, pulled back to $H^*(\btdiff(\bZ^p);\Q)$.
\begin{proof}[Proof of Corollary~\ref{C:MMM}]
     For any smooth manifold $Z$, there is a natural map $j: \bdiff(Z) \to \bhome(Z)$ corresponding to the passage from a smooth bundle with fiber $Z$. Ebert and Randal-Williams~\cite[Theorem 2]{ebert-randal-williams:MMM-TOP} show that for any characteristic class $c\in H^*(\BSO(n);\Q)$, there is a class $\kappa_c^{\Top} \in H^*(\bhome(Z);\Q)$ such that $j^*\kappa_c^{\Top}$ is the usual MMM class $\kappa_c$ corresponding to $c$. It follows that any class $\kappa$ in the ring generated by the MMM classes $\kappa_c$ is likewise equal to $j^*\kappa^{\Top}$ where $\kappa^{\Top}\in H^*(\bhome;\Q)$.

     Now suppose that $\balpha \in \pi_p(\tdiff(\bZ^p))$ is such that $\sw^{\pi_p}(\balpha)\neq 0$, and let $f\colon S^{p+1} \to \btdiff(\bZ^p)$ be the classifying map for the bundle $S^{p+1}\times_{\balpha} \bZ^p$.  By item \eqref{homotopy} of Theorem~\ref{T:infgen} the composition of the following maps is null-homotopic.
     \[
     \begin{tikzcd}
         S^{p+1} \arrow{r}{f} & \btdiff(\bZ^p) \arrow{r}{i} &\bdiff(\bZ^p) \arrow{r}{j} &\bhome(\bZ^p).
     \end{tikzcd}
     \]
    Hence for any MMM class $\kappa$, the evaluation $\langle f^*i^* \kappa,[S^{p+1}]\rangle =  \langle f^*i^*j^* \kappa^{\Top},[S^{p+1}] \rangle = 0$. Proposition~\ref{P:eval} says that $\langle f^* \swc_{\bZ^p}^p, S^{p+1}  \rangle = \sw^{\pi_p}(\balpha)\neq 0$, we see that $\swc_{\bZ^p}^p$ cannot be in the ring generated by MMM classes.
\end{proof}
This could also be deduced from Corollary~\ref{C:pseudo} using the construction~\cite[Theorem 1]{ebert-randal-williams:MMM-TOP} of MMM classes of block bundles that pull back to the usual MMM classes
%, and the fact that $\balpha$ is trivial in the homotopy group of the block diffeomorphism group.
% \endgroup

\subsection{Algebraic topology of \texorpdfstring{$\diff$}{Diff}}\label{S:algtop}
The previous sections use gauge theoretic invariants to detect the non-triviality of certain homotopy groups of the diffeomorphism group.  This short section uses tools from homotopy theory to prove Corollaries~\ref{C:Hinf} and~\ref{C:HQinf}. These  show that information from Theorem~\ref{T:infgen} about a single homotopy group propagates to give information about the higher homology of the diffeomorphism group. 
\begin{proof}[Proof of Corollary~\ref{C:Hinf}]
   Browder proved~\cite{browder:torsion} that an H-space with finitely generated homology has trivial $\pi_2$, generalizing a classical theorem about Lie groups.  But Theorem~\ref{T:infgen} says that when $p$ is even, the homotopy group $\pi_2(\diffo(\zz^p))$ is non-zero, so the homology $H_*(\diffo(\zz^p))$ is infinitely generated. 
\end{proof}
As mentioned in the introduction, for some fixed prime $q$, Browder's argument gives rise to $q$-torsion classes  in $H_*(\diffo(\zz^p))$ in infinitely many degrees. It would be interesting to understand a direct gauge-theoretic way of detecting these classes, and to know if they vanish under the natural map to the homology of $\homeo(\zz^p)$.

The proof of Corollary~\ref{C:HQinf} depends on techniques of rational homotopy theory discussed in Chapter 16 of the book~\cite{felix-halperin-thomas:RHT} by F\'elix-Halperin-Thomas. The use of rational homotopy theory in this context was pointed out to us by Jianfeng Lin.

\begin{lemma}[J. Lin, personal communication]\label{L:jianfeng} 
Let $f:A\to B$ be a map between simply connected spaces. Let $V$ be the kernel of the map  
\[
\Omega f_{*}\colon \pi_{2k}(\Omega A)\otimes \mathbb{Q} \to \pi_{2k}(\Omega B)\otimes \mathbb{Q}
\]
and let $S(V)$ be the symmetric algebra generated by $V$. Then kernel of the map \[
\Omega f_{*}\colon H_{*}(\Omega A;\mathbb{Q})\to H_{*}(\Omega B;\mathbb{Q})
\]
contains $S(V)$. In particular, if $V\neq 0$, then the kernel of $\Omega f_{*}$ is infinite dimensional.
\end{lemma}
\begin{proof}
Let $F$ be the homotopy fiber of $f$. Then we have a fibration $\Omega F\xrightarrow{h} \Omega A\xrightarrow {\Omega f} \Omega B$. Our assumption implies that the image of  $h_{*}:\pi_{2k}(\Omega F)\otimes \mathbb{Q}\to  \pi_{2k}(\Omega A)\otimes \mathbb{Q}$ equals $V$.  We pick a collection $\{\alpha_{i}\}\subset \pi_{2k}(\Omega F)=\pi_{2k+1}(F)$ whose image under $h_{*}$ gives a basis of $V$.  Represent $\alpha_{i}$ by a map $\rho_{i}\colon S^{2k+1}\to F$ and take the induced map $\Omega \rho_{i}\colon \Omega S^{2k+1}\to \Omega F.$  Let $Z$ be the weak product $\widetilde{\prod}_{i}\Omega S^{2k+1}$. Then the maps $\Omega \rho_{i}$ can be multiplied together to give the map $\tilde{\rho}:Z\to \Omega F$.  We consider the composition $h\circ \tilde{\rho}\colon {Z}\to \Omega A$. Since this map induces an injection map on rational homotopy groups, by the proof of~\cite[Proposition 16.7]{felix-halperin-thomas:RHT}, the map $(h\circ \tilde{\rho})_{*}\colon S(V)\cong H_{*}({Z};\mathbb{Q})\to H_{*}(\Omega A;\mathbb{Q})$  is also injective. Since this map factors through $H_{*}(\Omega F;\mathbb{Q})$, its image is contained in the kernel of $H_{*}(\Omega A;\mathbb{Q})\to H_{*}(\Omega B;\mathbb{Q})$. 
\end{proof}

Now we apply this lemma to show that when $p$ is even the (based) diffeomorphism groups of the manifolds  $\zz^p$ have infinitely generated rational homology in all even degrees. 
\begin{proof}[Proof of Corollary~\ref{C:HQinf}]
    Consider one of the manifolds $\zz^p$ with $p>0$ even, so that 
    \[
    V = \ker\left[\pi_2(\diffo(\zz^p))\otimes \bq \to \pi_2(\homeo(\zz^p))\otimes \bq \right]
    \]
    is infinitely generated.
    Let $A$ and $B$ be the classifying spaces of $\diffo(Z)$ and $\homeo(Z)$, respectively, so that $\Omega A \simeq \diffo(\zz^p)$ and $\Omega B \simeq \homeo(\zz^p)$. Then $A$ and $B$ are simply connected, so Lemma~\ref{L:jianfeng} shows that $S(V)$, which is infinitely generated in every degree, injects into 
    \[
    \ker\left[H_*(\diffo(\zz^p))\otimes \bq \to H_*(\homeo(\zz^p))\otimes \bq \right].\qedhere
    \]
\end{proof}

\section{Applications to embedding spaces}\label{S:consequences}
We now turn to the proof of Theorems~\ref{T:emb2} and~\ref{T:emb3}, which construct and detect smoothly non-trivial $k$-parameter (for $k \leq p$, $k \equiv p \pmod{2}$) families of embeddings of the $2$-sphere, the $3$-sphere, and $S^1 \times S^2$ in suitable stabilizations of the manifolds $\zz^p$ used to prove Theorem~\ref{T:infgen}. We discuss the case of embedded $2$-spheres first, as the $3$-manifolds in Theorem~\ref{T:emb3} arise as the boundary of tubular neighborhoods of those $2$-spheres. To to show that the families are non-trivial, we will use a family submanifold sum to create a family of manifolds to which we can apply Konno's cohomological invariant $\swc$. As in Proposition~\ref{P:eval}, this is computed by the invariant $\sw^{H_k}$ defined in Section~\ref{s:chains}, which in turn is computed via the gluing formula.
The families of surfaces are parameterized by spheres and so the non-trivial homology classes we find are in the image of the Hurewicz map. There are some interesting variations on our construction and detection results that we outline as well.
\begin{remark}\label{R:submanifold}
    For us, an embedding is a function rather than just the submanifold that is its image. One can study spaces of submanifolds, essentially by dividing by the action of the diffeomorphism group of the submanifold. Some results along these lines are presented in Section 7 of~\cite{auckly-ruberman:emb}. 
\end{remark}

For $2$-spheres of square $1$, the case $k=0$ is well-known; see for example~\cite{auckly-kim-melvin-ruberman:isotopy}. The key observation is that the diffeomorphism type of a manifold $X$ is encoded in the embedded $\cpone$ (henceforth $\Sigma$) in $X \cs \cptwo$; one recovers $X$ by an `anti-holomorphic' blowdown of the $\cpone$. Thus distinct smoothings of $X$ that become diffeomorphic after stabilization by $\cptwo$ will give rise to inequivalent $2$-spheres in  $X \cs \cptwo$. Similarly, one can recover $X$ from $X \cs (\sss)$ by surgery on one of the obvious spheres in $\sss$, giving rise to exotic square $0$ spheres.  We will show that this technique extends to higher-dimensional families, via family versions of blowing down and surgery. A key aspect of the argument is to make use of strong results about the diffeomorphism groups of $3$-manifolds. 

\begin{remark}\label{R:summands} We record here some notational conventions for this section.
    First, for a manifold $X$, we denote by $X^0$ its connected sum with $\sss$, and (following~\cite{auckly-kim-melvin-ruberman:isotopy}) by $X^\pm$ its connected sum with $\cptwo$ or $\cptwobar$. %We are indebted to Paul Melvin for this notation, introduced in~\cite{auckly-kim-melvin-ruberman:isotopy}.
    
    Second, an important aspect of our constructions here is that the manifolds we work with decompose as connected sums and/or submanifold sums. Typically, they have more than one such decomposition and we need to keep track of which one is being discussed at various stages of the argument. To indicate that we are working with an $\sss$ summand of a piece of a bigger manifold, we will decorate $\sss$ with a subscript indicating that particular piece. So for instance, the manifold $Z$ mentioned below has several $\sss$ summands, one which will be denoted $(\sss)_Z$. The exotic collections of embeddings that we construct become smoothly isotopic after one stabilization. The isotopy depends on the particular $\sss$ summand used, so a smooth contraction arising from a summand labeled $(\sss)_Y$ will be denoted using the same subscript, as in $F_Y$. 
\end{remark}  

\subsection{Constructing families of embeddings}\label{s:con-fam-emb} Since the case $k=0$ has already been treated in the literature, we concentrate on the argument for higher values of $k$. As in the case of diffeomorphisms, we first construct a $k$-dimensional family of embeddings in a specific manifold $Z^k$ where the manifold $Z^k$ gets larger (as measured by $b_2$) as $k$ gets larger. We then use the freedom to make the target manifold larger by submanifold sum. 
In Remark~\ref{R:emb-many} below, we will describe the modifications needed to have one manifold that is the target of families of different dimensions. This is similar to argument for the analogous fact about families of diffeomorphisms from Theorem~\ref{T:infgen} presented in Section~\ref{S:calculation}.

There are two ingredients to our construction. The first one is a family of diffeomorphisms $\alpha^p\co S^p\to \diff(Z,N)$ with $SW^{H_p}(\alpha^p) \neq 0$. The right condition is that $Z$ be constructed via submanifold sums of stabilized symplectic manifolds as described in Remark~\ref{R:def-comments} with the simplest instance being
\[ 
\bZ^p = \bZ^{p-1}\cs_T\bX_0^0 \cong \bZ^{p-1}\cs_T(\bX_0\cs(\sss))\cong (\bZ^{p-1}\cs_T\bX_0)\cs(\sss)
\]
where $\bX_0$ is defined in Section~\ref{S:pi0}.
This last $\sss$ summand is disjoint from the nucleus $N$, and is denoted $(\sss)_{\bZ}$ as in Remark~\ref{R:summands}. It is useful to write $\bZ' = \bZ^{p-1}\cs_T\bX_0$ so that $\bZ^p = \bZ'\cs(\sss)_\bZ$. 

We further require $\alpha^p\cs 1_{\sss}$ to smoothly contract in $Z\cs(\sss)$, which holds for $Z = \bZ^p$. In this case, we use $(\sss)_*$ to denote the final summand and denote the contraction by $F_*$. Notice that if $X$ can be written as $X'\cs(\sss)$ for some $X'$, then 
\[
\alpha^p\cs_T1_X = \alpha^p\cs_T1_{X'\cs (\sss)} =(\alpha^p\cs 1_{\sss})\cs_T1_{X'}
\]
 also smoothly contracts. We will denote such a contraction by $F_*$ as well. The families constructed in section~\ref{fam} all have these properties.

The second ingredient is an interesting collection of embeddings of a surface into a $4$-manifold.  We start by describing a collection of self-intersection zero spheres in a spin manifold and then modify it to give a collection of self-intersection one spheres in a non-spin manifold. Recall the diffeomorphism 
\[
\vphi_q\colon \bX_q \cs (\sss) = E(2;2q+1) \cs (\sss) \to \bX_0^0 = \bX_0 \cs (\sss)_X = E(2;1) \cs (\sss)
\]
from Section~\ref{S:pi0}. 

The embedding $J:S^2\to (\sss)\setminus\text{pt}$ given by $J(v) = (v,\text{pt})$ generates embeddings
$J_q\co S^2\to \bX_0^0$ via inclusion of the $\sss$ summand followed by $\vphi_q$. Similarly, set $\widetilde{J}:S^2\times S^1\to (\sss)\setminus\text{pt}$ to be the embedding given by $\widetilde{J}(v,\phi) = (v,\phi)$ where we view $S^1$ in the second factor as the equator in the second $S^2$.  This generates embeddings
$\widetilde{J}_q\co S^2\times S^1\to \bX_0^0$ via inclusion of the $\sss$ summand followed by $\vphi_q$. Thus, $\widetilde{J}_q$ is the embedding generated as the boundary of the tubular neighborhood of the family of embeddings of the self-intersection zero spheres. Since these become isotopic after one stabilization, the boundary of the tubular neighborhood does as well. Similarly, this embeddings in this family are all topologically isotopic and smoothly concordant to one another.

Since the images of these embeddings are disjoint from the nucleus, one may submanifold sum with any manifold along a submanifold in the nucleus, and obtain analogous embeddings and equivalences that we will still denote $J_q, F_{J,*}, K_J$, and $G_J$. Here we append a $J$ to the subscript to emphasize that these maps are equivalences of embeddings. Since $\bZ^p = \bZ'\cs(\sss)_\bZ$ contains a $\sss$-summand there is smooth stable isotopy $F_{J,\bZ}$ defined on 
\[ 
\bZ^p\cs_T\bX_0^0 \cong (\bZ'\cs_T\bX_0^0)\cs(\sss)_\bZ.
\]
This means $J_q$ is smoothly isotopic to $J_0$ in $\bZ^p\cs_T\bX_0^0$, via $F_{J,\bZ}$. Thus, the domain of $F_{J,\bZ}$ is $\bZ^p\cs_T\bX_0^0$ and the domain of $F_{J,*}$ is $(\bZ^p\cs_T\bX_0^0)\cs(\sss)$. Note that $\bZ^p\cs_T\bX_0^0$ is just $\bZ^{p+1}$. This will imply that the same manifolds will support exotic families of diffeomorphisms, embedded surfaces, and embedded $3$-manifolds. 

To obtain examples of smoothly knotted spheres of self-intersection plus one, we just stabilize via $\cptwo$. Results of Mandelbaum~\cite{mandelbaum:decomp} and Moishezon~\cite{moishezon:sums} imply the existence of a diffeomorphism
\[
\vphi_q^+\colon \bX_q \cs \cptwo = E(2;2q+1) \cs \cptwo \to \bX_0^+ = \bX_0 \cs \cptwo = E(2;1) \cs \cptwo,
\]
and work of Wall~\cite{wall:diffeomorphisms} says that $\vphi_q^+$ may be chosen to be homotopic to $\psi_q\cs 1_{\cptwo}$. The linear embedding $J^+\co S^2 = \cpone \to \cptwo\setminus\text{pt}$ then induces embeddings $J_q^+$ and equivalences $F_{J,*}^+$, $K^+_J$, $G^+_J$ as in the self-intersection zero case. After fiber summing with $Z^p$, we also have the isotopy $F_{J,Z}^+$.  

The boundary of the tubular neighborhood of $\cpone\subset\cptwo$ is a $3$-sphere, so the family of embeddings $J_q^+$ generates a corresponding family of embeddings $\widetilde{J}_q^+$ of $S^3$. These embeddings of $S^3$ are topologically isotopic, smoothly concordant, and stably isotopic after one external stabilization.

Combining these ingredients gives the desired families of embeddings.
\begin{definition}\label{D:fam}
    The families of embeddings of self-intersection zero spheres are given by
    \[
    J_q^{p+1}\co S^{p+1} = I\times S^p/(\{0,1\}\times S^p \cup I\times \text{pt})\to \Emb(S^2,Z\cs_T\bX_0^0),
    \]
\[
J_q^{p+1}(t,\theta)(v) = F_{*}(t,\theta)(J_q(v)).
\]
The associated equivalences 
\begin{gather*}
    F_J^{p+1}\co I\times S^{p+1} \to \Emb(S^2,Z\cs_T\bX_0^0\cs\sss),\\
      G_J^{p+1}\co I\times S^{p+1} \to \Emb(S^2,Z\cs_T\bX_0^0),\\
      K_J^{p+1}\co I\times S^{p+1} \to \Emb(S^2,Z\cs_T\bX_0^0),
\end{gather*}
    are given by
\begin{gather*}
    F_J^{p+1}(s,t,\theta)(v) = F_{*}(t,\theta)(F_{J,*}(s,v)), \\
    G_J^{p+1}(s,t,\theta)(v) = F_{*}(t,\theta)(G_J(s,v)), \\   
    K_J^{p+1}(s,t,\theta)(v) = F_{*}(t,\theta)(K_J(s,v)). 
\end{gather*}
\begin{remark}\label{R:emb-basepoint}
    Note that for fixed $p$, as $q$ varies, the embeddings $J^{p+1}_q$ are based at different base points. We use the isotopy $F_{J,Z}$ as a path connecting the base point for $J^{p+1}_q$ to the base point for $J^{p+1}_0$. Using the path, the families may be modified to have a common basepoint. 
\end{remark}
The corresponding families of spheres of self-intersection one and associated equivalences come from stabilizing by $\cptwo$ a the last stage and are just denoted by appending a $+$ to the notation as an extra superscript. The families of embeddings of  $S^2\times S^1$ and $S^3$ obtained as the boundary of the tubular neighborhoods are denoted by appending a $\widetilde{~}$ to the notation.   
\end{definition}

\begin{remark}\label{R:emb-many}
    There are non-spin manifolds that support for a range of values of $k$ of the same parity, $(k+1)$-dimensional families of exotic diffeomorphisms and $k$-dimensional families of embeddings of all of the types that we have discussed (self-intersection zero sphere, self-intersection one sphere, $S^2\times S^1$, 3-sphere). The point is that 
    \[
    (\bZ^p\cs\cptwobar)\cs_T(\bX_0\cs(\sss)) \cong (\bZ^p\cs 2\cptwobar)\cs_T(\bX_0\cs\cptwo).
    \]
    The manifold $(\bZ^p\cs\cptwobar)\cs_T(\bX\cs(\sss))$ supports families of diffeomorphisms as it is obtained by blowing up a family of diffeomorphisms. It also supports families of embeddings of self-intersection zero spheres and families of embeddings of $S^2\times S^1$. The manifold $(\bZ^p\cs 2\cptwobar)\cs_T(\bX_0\cs\cptwo)$ supports  families of embeddings of self-intersection one spheres and families of embeddings of $S^3$. We refer to Remark~\ref{R:def-comments} \eqref{multi} for details.
\end{remark}

This completes the construction of many spherical families of embeddings of any given dimension. It further demonstrates that the families of embeddings of a fixed space type (self-intersection zero sphere, self-intersection one sphere, 3-sphere, $S^2\times S^1$) over the same dimensional parameter space are topologically isotopic, smoothly concordant, and stably equivalent after one sum with $\sss$. We  now proceed to show that the corresponding families are independent in the homology of the appropriate space of embeddings.

\subsection{Family blowdown}\label{s:family-blowdown} 
To distinguish the families of embeddings we use family blow-down, family surgery, and a family submanifold sum to convert families of embeddings into families of manifolds. Konno's characteristic class invariant of families of manifolds is the invariant that we associate to families of embeddings.

There are two equivalent ways of describing the smooth structure of the blowdown of a sphere of self-intersection $\pm 1$, $\Sigma$; they each lead in somewhat different fashion to a family version. The first, whose family version is discussed in Construction~\ref{embeddings-to-manifolds}, is to remove a tubular neighborhood $\nu(\Sigma)$ from $Z$ and then glue in a $4$-ball along the resulting boundary.  The second, together with the family version is described in Construction~\ref{c:blowdown}. It is the connected sum of $Z$ with $\cptwobar$ along a $2$-sphere. There is an analogous pair of ways to describe surgery on a self-intersection zero sphere. The first removes the tubular neighborhood of the sphere and glue in a $S^1\times S^2$. The second is to take submanifold sum with families of $S^4$.

In this paper, we will use the cut-and-glue approach, but comment on the submanifold sum approach. 

\begin{construction}\label{embeddings-to-manifolds} 
Let $\Xi$ be a finite $k$-dimensional cell complex, and let $J:\Xi\to \Emb(S^2,Z)$ be continuous. We obtain a family of spaces from $\Xi\times Z$ by removing a tubular neighborhood of $J$. Even though $\Xi\times Z$ is a trivial family of manifolds, the exterior
\[
\mathcal{Z}_J' = (\Xi\times Z)\setminus \interior\,\nu(J),
\]
could be a non-trivial family of manifolds with boundary. The boundary is a bundle over $\Xi$ with fiber $\partial\nu(J_{x_0})$.  If every sphere in the family has self-intersection zero, $\partial\nu(J_{x_0}) \cong S^2\times S^1$. If every sphere in the family has self-intersection $\pm 1$, $\partial\nu(J_{x_0}) \cong S^3$. 

Let $\widetilde{J}:\Xi\to \Emb(Y,Z)$ be a family of embeddings of a $3$-manifold. In the same way we consider the exterior of the total embedding:
\[
\mathcal{Z}_{\widetilde{J}}' = (\Xi\times Z)\setminus \interior\,\nu(\widetilde{J}).
\]
When $Y$ and $Z$ are orientable, the embedding will be two-sided so the boundary will be a bundle over $\Xi$ with fiber $Y\perp\!\!\!\perp - Y$. In the special case $Y = S^2 \times S^1$ the fiber of the bundle will be $2S^2\times S^1$, and when $Y = S^3$ the fiber of the bundle will be $2S^3$.

A special point here is that bundles with fibers consisting of copies of $S^2\times S^1$ and $S^3$ extend uniquely 
to bundles with fibers consisting of copies of $D^3\times S^1$ and $D^4$. (For this discussion, we temporarily suspend the convention that $\diff =\diffp$.)  For $S^2\times S^1$, the central point is Hatcher's calculation~\cite{hatcher:S1xS2} of the homotopy type of $\diff(\SSS)$, which follows from the Smale conjecture via a `disjunction' technique.  As part of the argument, Hatcher shows that the inclusion of $\diff_s(\SSS)$ into $\diff(\SSS)$ is a homotopy equivalence, where $\diff_s(\SSS)$ consists of diffeomorphisms that take slices $\{x\} \times S^2$ to slices $\{y\} \times S^2$. It follows from this and the homotopy equivalences $\diff(S^1) \simeq O(2)$ and $\diff(S^2) \simeq O(3)$ 
that any family of diffeomorphisms of $\SSS$ extends to a family of diffeomorphisms of $S^1 \times B^3$.

For $S^3$ notice that $\diff(S^3)$ is homotopy equivalent to ${SO}(4)$ by Hatcher's resolution of the three dimensional Smale conjecture, \cite{hatcher:smale}.

The extension property allows us to define $\calz_J = \calz_J'\cup_JE$ and $\calz_{\widetilde{J}} = \calz_{\widetilde{J}}'\cup_{\widetilde{J}}E$ where $E$ is a bundle over $\Xi$ with fiber consisting of the appropriate number of copies of $D^3\times S^1$ and $D^4$.
\end{construction}

\begin{remark}
    While defining the Seiberg-Witten characteristic classes for possibly disconnected $4$-manifolds is not a problem, for the examples that we consider, $\calz_{\widetilde{J}}$  decomposes into two components, and we discard the component with fibers having small second Betti number.
\end{remark}

The invariant for families of embeddings is defined via these families of manifolds.
\begin{definition}\label{emb-inv-def}
    If $J\colon\Xi\to\Emb(S^2,Z)$ is a family of self-intersection zero or one embeddings, set
    \[
    \sw^{\Emb}(J) = \fsw^{\Z}(\calz_J\to\Xi,-) = \swc^{\Z}_Z(-)\cap[\Xi].
    \]
    If $\widetilde{J}\colon \Xi\to\Emb(S^2,Z)$ is a family of embeddings of $S^3$ or $S^2\times S^1$, set
    \[
    \sw^{\Emb}(\widetilde{J}) = \fsw^{\Z}(\calz_{\widetilde{J}}\to\Xi,-) = \swc^{\Z}_Z(-)\cap[\Xi].
    \]    
\end{definition}
Since the first invariant is defined on the union of all self-intersection zero (or one) embeddings, it will allow one to conclude that homotopy and homology groups of these families of embeddings contain infinite rank summands. It also allows one to conclude that the kernels  of the maps to topological embeddings also contain infinite-rank summands. 
There are technical issues related to generalizing these results to  higher dimensional families. Namely, one would need to establish a natural extension of the boundary of the exterior of the family to a well-understood family of spaces.

Taking the fiber sum with a standard family of embeddings avoids extension questions but adds a requirement on the family. This approach is employed by Drouin in his thesis, \cite{drouin:embeddings}. Once examples are constructed for any non-negative self-intersection, the remaining self-intersections are obtained by orientation reversal. In particular, the sum of the exceptional spheres in $n\cptwobar$ is a standard model of a self-intersection $-n$ sphere. This gives rise to a standard family in $S^p\times n\cptwobar$. Drouin defines a family of self-intersection $n$ spheres to be \emph{parameter-untwisted} if the normal bundle of the family is fiber orientation-reversing isomorphic to the normal bundle to the normal bundle of the family in $S^p\times n\cptwobar$. He further constructs an invariant to detect parameter-untwisted families. Topologically trivial families are all parameter-untwisted.

\begin{construction}\label{c:blowdown}
Let
$J: \Xi \to \Emb_A(S^2,Z)$ be continuous, with
\[
\Gamma_J = \left\{(x,J_x(y)) \mid (x,y) \in \Xi \times S^2\right\} \subset \Xi \times  Z
\]
denoting its graph. Note that $\Gamma_J$ is diffeomorphic to $\Xi \times S^2$. 

In the parameter-untwisted case, Drouin creates a family $\calz_J$ as the fiberwise sum 
\begin{equation}\label{E:familysum}
(\Xi \times Z) \cs_{\Gamma_J = \Xi \times E} (\Xi \times \cptwobar).
\end{equation}
\end{construction}

\subsection{Computing the invariant of families}\label{compute-emb}
For spherical families of embeddings, the families of spaces used to define the invariants of families of embeddings are bundles over spheres. These bundles are determined by the transition function which is a map of the equator into the diffeomorphism group of the fiber. In this case, the invariant of the family of spaces agrees with the invariant of the family of diffeomorphisms. Indeed, the invariant of the family of spaces depends on a family of moduli spaces defined 
over the base sphere of the family of spaces via data on the sphere. The invariant of the diffeomorphism is defined via a collection of data defined over a disk. Since the family of spaces is obtained by gluing two trivial families of spaces defined over the hemispheres using the family of diffeomorphisms. One may take constant data over one of the hemispheres. There will be no solution to the parameterized equations over the hemisphere with constant data, and this leads to the equality of the two invariants. 

This leaves the problem of identifying the family of diffeomorphisms associated to $\calz_J$ and $\calz_{\widetilde{J}}$. The stable isotopy effectively defines a family isomorphism from the twisted family to a trivial family. To see this, it is useful to use the following model of a spherical family of spaces associated to a spherical family of diffeomorphisms.
\begin{definition}
    Let $\gamma\co S^p \to \diff(Z)$ and define
    \[
    S^{p+1}\times_\gamma Z = \left((O\times\{\pm\})\times D^{p+1}\times Z\right)/\sim
    \]
    where $(-,\theta,z) \sim (0,\theta,z)$ and $(+,\theta,z) \sim (1,\theta,\gamma(\theta)(z))$ when $|\theta| = 1$. The projection to $S^{p+1}$ just forgets the $Z$-factor.
\end{definition}

The smooth stable isotopy gives rise to the following family isomorphism.
\[
\Phi\co S^{p+1}\times_{\alpha^p\cs_T\,1_{E(2,2q+1)^0}}(\bZ^p\cs_TE(2,2q+1)^0) \to S^{p+1}\times (\bZ^p\cs_TE(2,2q+1)^0,
\]
given by 
\[
\Phi(t,\theta,z) = \begin{cases}
    (\pm,\theta,z), \ \text{if} \ t = \pm, \\
    (t,\theta,F_*(t,\theta)(z).
\end{cases}
\]
 The diffeomorphism $\vphi_q\co E(2,2q+1)$ induces the family isomorphism 
 \[
 1_{S^{p+1}}\times \vphi_q\co S^{p+1}\times (\bZ^p\cs_TE(2,2q+1)^0 \to  S^{p+1}\times (\bZ^p\cs_T\bX_0^0).
 \]

The composition $(1_{S^{p+1}}\times \vphi_q)\circ\Phi$ is an isomorphism between $S^{p+1}\times (\bZ^p\cs_T\bX_0^0)$ and  $S^{p+1}\times{\alpha^p\cs_T1_{E(2,2q+1)^0}}(\bZ^p\cs_TE(2,2q+1)^0)$. Performing parameterized surgery on each side yields a family isomorphism
\[
S^{p+1}\times_{\alpha^p\cs_T1_{E(2,2q+1)}}(\bZ^p\cs_TE(2,2q+1))\to \calz_J.
\]
This implies that $\sw^{\Emb}(J^{p+1}_q) = \sw^{H_p}(\alpha^p\cs_T1_{E(2,2q+1)})$. 
\begin{proof}[Proofs of Theorem~\ref{T:emb2}, Theorem~\ref{T:emb3} and Corollary~\ref{C:embK}]
Evaluating $\sw^{H_p}$ on the $\Spinc$-structures
corresponding to $K+2\Sigma\Tc+2\ell \Tc_1$ as in the proof of Theorem~\ref{T:Zinfty} establishes that $J^{p+1}_q$ are linearly independent in $H_{p+1}(\Emb(S^2,\bZ^p))$, and thus in $\pi_{p+1}(\Emb(S^2,\bZ^p))$. Since it was already discovered that $J^{p+1}_q$ topologically contract via $G_J$, this establishes that there is a $\Z^\infty$ summand in the kernel as well. The stable contraction $F_J$ demonstrates that the families $J^{p+1}_q$ are all equivalent after one external stabilization. This completes the proof in the self-intersection zero case. The proofs for families with self-intersection one, $S^1\times S^2$, and $S^3$ are all the same, just modifying the notation by appending $+$ and $\sim$ where appropriate. The family concordances $K_J$ establish the corollary.   
\end{proof}

\begin{remark}\label{R:surf-bound}
    The exotic families of surfaces that we construct all live in an extra $4$-manifold summand. The summation takes place in a neighborhood that avoids the support of the family of diffeomorphisms and is contained in a nucleus that avoids the support of the families of diffeomorphisms used in the construction. This means a finger of this surface can be dragged into this region and into any manifold attached via submanifold in the region, in particular, one can pick a symplectic manifold and sum along a symplectic torus. One can then excise any codimension zero submanifold that missed the symplectic surface together with any par of the finger intersecting this submanifold. The result will be an exotic family of surfaces (possibly with boundary) in a manifold with boundary. The proof that the resulting families are distinct relative to the boundary is to just glue back the excised part. 
    This establishes Theorem~\ref{T:stab-all}
\end{remark}

\section{Metrics of positive scalar curvature}\label{S:psc}
In this section, we prove Theorem~\ref{T:psc} concerning higher homology and homotopy groups of $\calrp(\bZ)$, the space of Riemannian metrics with positive scalar curvature (PSC) on certain $4$-manifolds $\bZ$.  The method extends 
Witten's proof~\cite{witten:monopole}  that for $4$-manifolds with $b_2^+ >1$, the non-vanishing of the Seiberg-Witten invariant obstructs the existence of a PSC metric. This is a much stronger restriction than in higher dimensions, where in many cases, existence is determined by index theoretic considerations~\cite{rosenberg:psc-progress}. For spin $4$-manifolds, admitting a PSC metric implies the vanishing of the $\widehat{A}$-genus, and hence the signature. Similarly, $1$-parameter Seiberg-Witten invariants were used in used~\cite{ruberman:swpos} to obstruct isotopy between PSC metrics. In that paper, the manifolds were connected sums of $\cptwo$ and $\cptwobar$ and hence were not spin. So the part of Theorem~\ref{T:psc} relating to $\pi_0(\calrp(\bZ))$ when $\bZ$ is spin, is new, as are all of the statements about higher homotopy groups.

For the remainder of this section, we assume that $X$ is a manifold admitting a PSC metric, $\spincs \in \Shat_X^k$, and that $b_2^+(X) >k+2$. In this situation, we define invariants of the homology of $\calrp(X)$ analogous to the diffeomorphism invariants $\sw_{\spincs}^{H_k}$. Suppose that $\Xi$ is compact, and note that any  map $g\colon \Xi \to \calrp(X)$ yields data $\data\colon \Xi \to \Pi(X)$ of the form $\data(\theta) = (g(\theta),0)$.   Witten's argument in the family setting shows that there will be no irreducible points in the parameterized moduli space $\calm(\data)$. Moreover, if $\Xi$ is a manifold with $\dim(\Xi) \leq k+2$, then there is a small perturbation $\data' =(g',0)$ for which there are no reducible solutions, or in other words will be good. This uses the observation~\cite[p. 142]{friedrich:dirac} \cite{donaldson:swsurvey}  \cite[Proposition 4.3.14]{donaldson-kronheimer} that the set of metrics for which there are reducible solutions to the unperturbed \SW has codimension $b^2_+(X)$ in $\Pi(X)$.  Since $\calrp(X)$ is open in $\met(X)$, that set also has codimension $b^2_+(X)$ in $\calrp(X)$.  Working simplex by simplex, we see that any homology class in $H_k(\calrp(X))$ can be represented by a chain $\alpha$ in $\calrp(X)$ that is good when viewed as a chain in $\Pi(X)$. 
\begin{definition}\label{D:psc}
With assumptions on $X$ and $\spincs$ as before, let $\alpha$ be a good chain representing a homology class in $\calrp(X)$. Choose a good chain $\beta$ in $\Pi(X)$ with $\partial \beta = \alpha$, and set $\sw_{\spincs}^{k,\, \psc}(\alpha) = \# \calm(\beta)$ as in Definition~\ref{Hinvdef}.
\end{definition}
Arguments similar to those in Section~\ref{s:hom} show that these invariants are well-defined and hence define homomorphisms $\sw_{\spincs}^{k,\, \psc}\colon H_k(\calrp(X)) \to \Z$.  Summing over $\spincs \in \Shat_k(X)$, we assemble these homomorphisms into a homomorphism 
\[
\swu^{k,\, \psc}\colon H_k(\calrp(X)) \to \Z^\infty.
\]

The main observation is that for any manifold $X$, the group $\diff(X)$ acts by pull-back on the space of Riemannian metrics on $X$, preserving the subspace $\calrp(X)$. So an initial choice $g$ of PSC metric defines a map $E^g\colon \diff(X)\to \calrp(X)$.  Directly from the definitions, we get the following lemma.
 \begin{lemma}\label{L:psc-diff}
     For any $\alpha \in H_k(\diff(X))$ we have $\swu^{H_k}(\alpha)= \swu^{k,\, \psc}(E^g_*(\alpha))$.
 \end{lemma}

\begin{proof}[Proof of Theorem~\ref{T:psc}]
    By construction, the manifolds $\bZ^p$ constructed in the proof of Theorem~\ref{T:infgen} are completely decomposable. In the non-spin case, we write $\zzn^p$ for the corresponding connected sum of copies of $\cptwo$ and $\cptwobar$. In the spin case, we write $\zzs^p$, being sure to choose the signature equal to $0$ so that $\zzs^p$ is diffeomorphic to a connected sum of copies of $\sss$. Note that all the summands have standard PSC metrics, which can be transported to the connected sums via the gluing construction of Schoen-Yau~\cite{schoen-yau:psc} and Gromov-Lawson~\cite{gromov-lawson:psc} (as amended in~\cite{rosenberg-stolz:psc}). 
    
    Pick an initial good PSC metric $g$ on each $\bZ^p$. Theorem~\ref{T:infgen} constructs, for all $0< j \leq p$ with $j \equiv p \pmod{2}$, a $\Z^\infty$ summand of $H_j(\diff(\bZ^p))$ detected by $\swu^{H_k}$. Lemma~\ref{L:psc-diff} implies that the image of $E^g_*$ is a $\Z^\infty$ summand of $H_j(\calrp(\bZ^p))$.  Since the summand in $H_j(\diff(\bZ^p))$ consists of spherical classes, the same is true for the corresponding summand of $H_j(\calrp(\bZ^p))$, and so we get a $\Z^\infty$ summand of $\pi_j(\calrp(\bZ^p))$.
\end{proof}

\section{Proof of the parameterized irreducible-reducible gluing theorem}\label{glue-proof}
In this section, we prove Theorem~\ref{pIRglue}, which describes how to glue families of solutions to the parameterized \SW equations on a connected sum. The term \emph{parameterized irreducible-reducible} indicates that the theorem is concerned with parameterized moduli spaces (with different parameter spaces for the two summands), where one moduli space consists of only good (and hence irreducible) solutions, and the other consists of only reducible solutions. 

The concept and structure of the proof are very similar to the proof of the usual (non-parameterized) blow-up formula~\cite[Theorem 4.6.7]{nicolaescu:swbook}. One first stretches a neck on $Z\cong Z\cs S^4$ to conclude that the moduli space for 
$Z$ is the product of the moduli space on $\hat Z$, a punctured copy of $Z$ with a cylindrical end, and the moduli space on $\R^4=\hat{S^4}$ with suitable cylindrical metric. (We remind the reader of our general convention that a `\, $\hat{}$\ '  over a manifold means that it has a cylindrical end, and more generally that a `\, $\hat{}$\ ' over a geometric object means that it is appropriately adapted to the cylindrical end.)  Due to the positive scalar curvature metric on $\hat{S^4}$ this last moduli space is just a point, so the moduli spaces associated to $Z$ and $\hat Z$
are isomorphic. Now $\hat{S^4}$ may be replaced with anything that has a moduli space consisting of exactly one suitably non-degenerate reducible configuration. Indeed one just ``stretches the neck" on such a configuration to arrive at the product of the moduli space of $\hat Z$ with the moduli space of the special family. This is the case for $\cptwobar$ as given in the traditional blow-up formula. It is also the case for any simply-connected negative definite manifold. It is also the case for a family where there is one data point with a suitably non-degenerate reducible configuration. The main example we have in mind is 
$\sss$ with the one-parameter family of data $\data_t^{\sss}$ as specified in Lemma~\ref{L:standardsssdata} below. 

The neck-stretching argument can get a bit technical, and leads to complicated notation. A very detailed description of these gluing arguments in the non-parameterized case is given in \cite{nicolaescu:swbook}. There have been various generalizations of this work into the parameterized setting. The wall-crossing theorem of \cite{li-liu:family} may be viewed as one, as can the calculation in~\cite{ruberman:isotopy,ruberman:swpos} and the families blow-up formula of \cite{liu}. The Baraglia-Konno gluing formula \cite{baraglia-konno:gluing} comes very close to what we need for this paper, but we could not see how to use that result exactly as stated. The Baraglia-Konno gluing formula considers the sum of a pair of families each defined over the same parameter space. The nature of the Suspension Theorem~\ref{anti-hol} requires a result that works for sums of families defined over different parameter spaces.

The gluing procedure we use has roots in work of Taubes~\cite{taubes:definite}, and  has four steps. The first step is to construct approximate solutions. The second step is to study the linearizations of the equations around the approximate solutions, and the third step is to solve the full non-linear equations via an application of the contraction mapping theorem. This provides a map from suitable 
pairs of configurations on the two sides into the moduli space of the closed manifold. The fourth and final step is to show that the map defined in this way is an isomorphism.

By using CW theory to define invariants for families as in \cite{konno:classes}, or transverse singular chains as done here in section~\ref{s:chains} one reduces to the case of topologically trivial families of manifolds over cells. This is the situation in which we describe our gluing theorem.

Even when the family of manifolds is trivial, the Seiberg-Witten map for the parameterized case is correctly viewed as a section of a Hilbert bundle over the space of parameterized configurations mod gauge. Our approach is to organize the parameterized gluing theorem so it resembles the unparameterized gluing results as closely as possible. By restricting to the case of zero dimensional moduli spaces, we show that we can locally work with metrics that do not vary as the parameter varies. This reduces the bundles to trivial bundles.
We further argue that we may select perturbations with support away from the neck. Since most of the analysis in the gluing theorem takes place in the neck, these two simplifications essentially reduce the proof of our gluing theorem to the unparameterized case. These simplifications are addressed in Section~\ref{S:spec-data}.

\subsection{Analytical framework}\label{an-frame}

The codomain of the Seiberg-Witten map includes the space of self-dual $2$-forms as a summand. This means that in the parameterized setting, the codomain should be viewed as a bundle over the parameter space. In the case of families of a closed manifold $X$ over a contractible base, $\Xi$, the correct geometry to describe this proceeds as follows.  {We have, implicitly, chosen a trivialization of the family; this does not affect any of the conclusions in this section.} Let $\mathcal{A}$ denote the space of $L^2_2$-connections on the determinant line of the $\Spinc$ structure, and let $\mathcal{C} = \Xi\times \mathcal{A}\times L^2_2(X,S^+)$ be the configuration space. 
Define a Hilbert bundle over the configuration space by
\[
\overline{\mathcal{E}} = \{((\theta,A,\psi),(\gamma,\vphi)) \in \mathcal{C}\times L^2_1(X,i\, \Bigwedge^2\oplus S^-)\,|\, P_+(g_\theta)\gamma = 0 \}
\]
where $P_+(g) = \frac12(1+*)$ is the projection to the self-dual forms.
One defines the gauge group to be $\mathcal{G} = L^2_3(X,S^1)$, and the based gauge group $\mathcal{G}^0$ to be the subgroup mapping a basepoint to $1\in S^1$. 
One then takes the quotient by the standard action of the gauge group \cite[Section 2.1.1]{nicolaescu:swbook} to get
\[
\mathcal{E} = \mathcal{G}\backslash\overline{\mathcal{E}} \to \mathcal{B} = \mathcal{G}\backslash\mathcal{C}.
\]

The Seiberg Witten equations may be viewed as a section of this bundle map given by 
\[
\widehat{SW}_\data(\theta,A,\psi) = (\sqrt{2}(F^+_{A} +i\eta^+(\theta) - \frac12{\bf c}^{-1}(q(\psi) )), D^+_{A}\psi).
\]
This is a well-defined section of $\mathcal{E}$, and the parameterized moduli space is just the zeros of this section. 
% \danny{compare conventions with order in \eqref{E:SWdirac} and \eqref{E:SWcurv}.} 

These equations generalize to manifolds with cylindrical ends. A detailed description of cylindrical structures may be found in the book by Nicolaescu \cite{nicolaescu:swbook}. 
The situation considered in \cite{nicolaescu:swbook} is that of a closed $4$-manifold $\hat N(0)$ containing a separating copy of $[0,1]\times N$. In particular manifolds with a `\, $\hat{}$\ ' accent are $4$-manifolds and manifolds with no accent are $3$-manifolds. The manifold $N(r)$ is obtained by adding a copy of $[0,r]\times N$ into the neck. Taking the limit as $r\to\infty$ leads to $\hat N(\infty) = \hat N_1 \coprod \hat N_2$. In fact, this process is turned around and the construction starts with the manifolds $\hat N_1$ and $\hat N_2$ to obtain the manifolds $N(r)$. 
\[
N(r) = \left(\hat N_1\setminus (r+1,\infty)\right)\cup_{\hat\vphi_r} \left(\hat N_2\setminus (r+1,\infty)\right),
\]
where $\hat\vphi_r(t,x) = (t - 2r - 1, \overline{x})$.
It is useful to consider that the cylindrical part extends even further to a copy of $[-1,\infty)\times S^3$ in $\hat N$.

Analogous constructions are performed using metrics, bundles, and $\Spinc$ structures.  \cite[Sec 4.1]{nicolaescu:swbook}. The symbol $\hat N$ may be used to represent either $\hat N_1$ or $\hat N_2$. At first, Nicolaescu considers a completely general situation in which a product neck is stretched. He later restricts to two special situations, both with $N = S^3$. The first is when the configurations on 
both sides are irreducible and strongly irreducible \cite[Theorem 4.5.17]{nicolaescu:swbook}. The second, which is the one relevant to us, is when all configurations on $\hat N_1$ are irreducible and strongly regular, and there is a unique configuration on $\hat N_2$ and this configuration is reducible and suitably non-degenerate \cite[Theorem 4.5.19]{nicolaescu:swbook}. Baraglia and Konno~\cite{baraglia-konno:gluing} have a generalization of this theorem so that both sides may be parameterized by the same parameter space, but our version allows different parameter spaces on each side. Generally, objects on the irreducible side will be indicated with a subscript of $1$, objects on the reducible side will be indicated by a subscript of $2$, and objects on the closed manifold with neck-length $r$ will be indicated by adding $(r)$ to the notation.

To state our gluing result we extend the definitions of data to the cylindrical case. Let $\widehat{\met}(\hat N_i)$ denote the set of cylindrical metrics limiting to the standard round metric on $S^3$ and  $\widehat{\Omega}^2(\hat N_i)_0$ denote the set of $2$-forms with support in $\hat N_i\setminus (0,\infty)\times S^3$. The set of cylindrical data is then given by 
\[
\widehat{\Pi}(\hat N_i) = \widehat{\met}(\hat N_i) \times \ker\left(d:\widehat{\Omega}^2(\hat N_i)_0\to\widehat{\Omega}^3(\hat N_i)_0\right).
\]

% \danny{Insert remark that we use compact supported. Write $\Omega(\hat{N})_0$ and $\Pi_0$}

Notice that there is a natural notion of sums of data. In particular, given $\data^i:\Xi_i\to \widehat{\Pi}(\hat N_i)$, we get
\[
\data^1\cs_r\data^2:\Xi_1\times \Xi_2\to \widehat{\Pi}(N(r)).
\]
In the situations where we apply this, $\Xi_i$ will be a simplex. We use this notation because $\Delta$ is used to represent difference maps here as in \cite{nicolaescu:swbook}.

In general, one may consider several different variations of cylindrical end moduli spaces corresponding to different constraints on the value at infinity. This leads to many technical complications. Since the situation here always has an end corresponding to the standard round $S^3$, most of these complications may be avoided. In the cylindrical case, 
there are several different ways to generalize spaces of gauge transformations, connections and spinor fields. The first possibility is to take gauge transformations modeled on the weighted Sobolev space $L^2_{3,\delta}$, connections and the positive spinor fields modeled on $L^2_{2,\delta}$, and self-dual $2$-forms and negative spinor fields modeled on $L^2_{1,\delta}$. These are the fibers of the Hilbert bundle over the configuration space. Here $\delta$ is any positive parameter so that the only eigenvalue of the linearization of the gauge-fixed Seiberg-Witten equations over $S^3$ less or equal to $\delta$ is zero. In this model, the exponential decay forces the gauge transformations to the identity at infinity, so this is an analog of the based gauge group with the base point at infinity. The resulting moduli space corresponds to the framed moduli space over the closed manifold.

To define a moduli space on the cylindrical end manifold corresponding to the moduli space over a closed manifold, one must extend the function spaces to allow other values at infinity. The first approach extends the gauge group to allow any unit complex number to be the value at infinity.

Let $\tilde{h}\co [-1,\infty) \to [0,1]$ be a smooth function equal to zero in a neighborhood of $0$ and equal to one on $[0,\infty)$. Set $h\co \hat{N}\to [0,1]$ to be $\tilde{h}$ composed with projection on $[-1,\infty)\times S^3$ extended as zero elsewhere. 
\begin{definition}
    The \emph{restricted cylindrical end gauge group} is 
    \[
    \mathcal{G}' = \{g + hz \,|\, g\in L^2_{3,\delta}(\Bigwedge^0\hat N\otimes\bc), \ 
    z \in S^1, |g + hz| = 1\}.
    \]
    The map $\partial_\infty\co\mathcal{G}\to S^1$, given by $\partial_\infty(g+hz) = g$ is the value at infinity. We often write $g$ for $g+hz$, and $\partial_\infty(g) = g_\infty$.
    The \emph{infinity-based cylindrical end gauge group} is 
    $\mathcal{G}^0 = \text{ker}(\partial_\infty)$.
    The corresponding quotients are $\mathcal{B} = \mathcal{G}\backslash\mathcal{C}$ and $\mathcal{E} = \mathcal{G}\backslash\overline{\mathcal{E}}$. The \emph{framed} spaces are the corresponding quotients by the based gauge transformations: $\widehat{\mathcal{B}} = \mathcal{G}^0\backslash\mathcal{C}$ and $\widehat{\mathcal{E}} = \mathcal{G}^0\backslash\overline{\mathcal{E}}$.  
\end{definition}

The second approach adds all gauge transformations over $S^3$ to the gauge group. This requires one to also add Seiberg-Witten configurations over $S^3$ to the space of gauge and spinor fields. 

\begin{definition}
    The \emph{cylindrical end gauge group} is 
    \[
    \mathcal{G} = \{g + hu \,|\, (g,u)\in L^2_{3,\delta}(\Bigwedge^0\hat N\otimes\bc) \times L^2_{3}(\Bigwedge^0S^3\otimes\bc), \ 
    |u| = 1, |g + hu| = 1\}.
    \]
The \emph{cylindrical end configuration space} is
\begin{small}
    \[
\mathcal{C} = \{(A+ia+ihb,\psi,\theta) \,|\, (ia,\psi)\in L^2_{2,\delta}(i\Bigwedge^1\hat N)\oplus S^+), \ b \in \text{ker}(d^*:L^2_2(\Bigwedge^1S^3) \to L^2_1(\Bigwedge^0S^3), \ \theta\in\Xi\}.
\]
\end{small}
We will omit `cylindrical end' when there is no chance of confusion.
\end{definition}

In both approaches the Seiberg-Witten equations take values in $L^2_{1,\delta}$ spinor fields and imaginary-valued self-dual $2$-forms.
It is not difficult to see that the two approaches to the cylindrical end moduli spaces are equivalent. There is a natural map from the moduli space defined in the restricted setting to the cylindrical end moduli space taking a restricted equivalence class to a gauge equivalence class. To see that the map is surjective, one just needs to note that any configuration may be put into coulomb gauge at infinity. To see that the map is injective, just notice that two gauge equivalent configurations that are in coulomb gauge at infinity  
are related by a gauge transformation that is constant at infinity.

\begin{definition}\label{regular}
    A configuration or gauge  orbit  of a configuration is called \emph{parameterized regular} if the vertical component of the derivative of  $\widehat{SW}_\data$ is surjective at that point. A set of data is called \emph{good} if all solutions are parameterized regular and irreducible. The space of framed configurations where $\widehat{SW}_\data$ vanishes is the \emph{framed parameterized moduli space} $\widehat{\mathcal{M}}(\{\data\})$. The space of configurations where $\widehat{SW}_\data$ vanishes is the \emph{parameterized moduli space} ${\mathcal{M}}(\{\data\})$.
\end{definition}

\begin{remark}
~\smallskip
\begin{enumerate}
    \item Nicolaescu uses $\widehat{\mathcal{M}}$ for  moduli spaces over $4$-manifolds and ${\mathcal{M}}$ for moduli spaces on $3$-manifolds. He appends a star as in $\widehat{\mathcal{M}}(*)$ to represent framed moduli spaces \cite{nicolaescu:swbook}. Since we do not need to consider other $3$-manifolds, we return to a standard, older notation. 
    \item A configuration is \emph{regular} if the vertical component of the derivative restricted to variations that are trivial in the parameter direction is surjective. Regular implies parameterized regular. For family invariants the typical situation will be parameterized solutions that are parameterized regular but not regular.
    \item There is also a notion of strongly parameterized regular to describe the situation when the vertical component of the derivative of the Seiberg-Witten map is surjective when restricted to variations that do not change the limiting value. See the discussion following \cite[Lemma 4.3.18]{nicolaescu:swbook}. In all of the situations we consider the cylindrical end will be modeled on $(0,\infty)\times S^3$ with the standard metric, so the two notions coincide. 
    \item If the parameter space is not contractible the configuration spaces of forms and connections will be replaced by bundle analogues. See \cite{konno:classes}. We will minimize the use of bundles in the gluing argument.
    \item When $\vdim(\mathcal{M}_{X,\spincs}) + \text{dim}(\Xi) \le 0$ and $\dim(\Xi) < b^2_+(X)$ generic data is good.
\end{enumerate}
\end{remark}
For good data the parameterized moduli space is a manifold of the expected dimension, which in turn is $\vdim(\mathcal{M}_{X,\spincs}) + \dim(\Xi)$.

\subsection{Special Data}\label{S:spec-data}

In this section we specify the assumptions on the data that we use in the proof of our gluing theorem. We also prove that such data may always be chosen for the numerical invariants that we consider.

\subsubsection{Locally constant metrics}\label{s:constant}
In this short section, we sketch the proof that for a $0$-dimensional family moduli space, we may assume that, after a small deformation, we have locally metric independent data as in Definition~\ref{D:local-g}.

\begin{proposition}\label{SPNMID}
    Let $\data:\Xi \to \Pi(X)$ be good data, so that $\vdim\mathcal{M}_{X,\spincs}(\{\data_\theta\}) = 0$ and $\calm_{X,\spincs}(\{\data_\theta\})$ consists of finitely many irreducible solutions sitting over exceptional points in top simplices of $\Xi$. There exists another set of good locally metric independent data $\data':\Xi \to \Pi(X)$ so that
    \[
\mathcal{M}_{X,\spincs}(\{\data_\theta\}) = \mathcal{M}_{X,\spincs}(\{\data'_\theta\}).
    \]
\end{proposition}

\begin{proof}[Proof of Proposition~\ref{SPNMID}] 
A version of this proposition is proved in~\cite[Appendix A]{mrowka-ruberman-saveliev:sw-index} for $1$-parameter families. The proof adapts readily to the setting where the parameter space $\Xi$ is a $k$-dimensional complex; compare also~\cite[Proposition 7.2]{baraglia-konno:gluing}. 

The argument in~\cite{mrowka-ruberman-saveliev:sw-index} is given in terms of the framework explained in~\cite{kronheimer-mrowka:monopole} for finding good perturbations as in Proposition~\ref{P:generic}. We briefly explain the setup and sketch the adaptations of the proof to the higher-parameter case. At the outset, we note one important difference;  whereas~\cite{mrowka-ruberman-saveliev:sw-index} is concerned with reducible configurations, we are dealing only with irreducible ones. Hence in~\cite{mrowka-ruberman-saveliev:sw-index} the proposition is proved for configurations in the boundary of the blown-up configuration space $\calb^\sigma$ and moduli space $\calm^\sigma$ from~\cite{kronheimer-mrowka:monopole} but for us there is no need for the blown-up configuration space.
% \danny{Replace $\calr$ and $\calr$ with $\met$}

Consider the space $\widetilde{\calz} \subset \met(X) \times \calb$ consisting of triples $(g,A,\phi)$ with $D^{g,+}_A(\phi) = 0$ where $D^{g,+}_A$ is the Dirac operator for the metric $g$. (This is an extension of the notation from~\cite{kronheimer-mrowka:monopole}, where $\calz$ denotes the fiber over a particular metric $g$.)
% \danny{Square this with notation elsewhere.} 
The proof of~\cite[Lemma 27.1.1]{kronheimer-mrowka:monopole} shows this to be a Hilbert submanifold of $\met(X) \times \calb$ for which the projection to $\met(X)$ is a submersion. 
The function 
\[
\sw_{(1)} = (g,\sqrt{2}(F^+_{A}  - \frac{1}{2}{\bf c}^{-1}(q(\psi) )))
\]
(as in Equation~\ref{E:SWcurv})
takes its values in $\Omega_{\met(X)}$, the subspace of $\met(X) \times \Omega^2(X;i \R)$ consisting of pairs $(g,\omega)$ with $\omega$ being $g$-self dual.  Projection onto the first factor makes $\Omega_{\met(X)}$ into a bundle over $\met(X)$, which can be trivialized by parameterizing $\met(X)$ in terms of automorphisms of the tangent bundle as in~\cite[Chapter 3]{freed-uhlenbeck}.
% \danny{The proof is commented out; do we want/need it?}

For a single metric, as in~\cite[Lemma 27.1.1]{kronheimer-mrowka:monopole}, the generic-data argument is completed by choosing $\eta$ so that $-i\sqrt{2}\eta^+$ is a regular value for the restriction of $\sw_{(1)}$ to $\calz$. In the family setting, one perturbs the data so that the map $\data:\Xi \to \Omega_{\met(X)}$ is transverse to $\sw_{(1)}\colon \widetilde{\calz}\to \Omega_{\met(X)}$; this establishes Proposition~\ref{P:generic}.  By our hypotheses that $\vdim\mathcal{M}_{X,\spincs}(\{\data_\theta\}) = 0$ and that $b_2^+(X) > k$ we can assume that and all of our metrics lie in $\met(X)$ and that there are no solutions over the $k-1$ skeleton.  Moreover, the parameterized moduli space consists of isolated points each lying over $\theta = 0$ for $\theta$ in a ball $U$ in the interior of a top-dimensional cell of $\Xi$.   We will change the metric component of  the map $\data$ on the interior of $U$, so that near $0$, the metric is independent of $\theta$. 

Write $\data_\theta = (g(\theta), \eta(\theta))$ and note that $\Omega_+^2$ and the projection from $\Omega^2 \to \Omega_+^2$ depend continuously on the metric $g$.  Hence there are $C, \epsilon >0$ such that for all $g'$ defined on a ball of radius $\epsilon$ with $|g'(\theta)-g(\theta)| <C$ the data $\data'(\theta) = (g'(\theta),\eta(\theta)$ is still transverse to $\sw_{(1)}\colon \widetilde{\calz}\to \Omega_{\met(X)}$.  Choose $\epsilon$ small enough so that $|g(\theta)-g(0)| <C$ for all $\theta\in B_\epsilon(0)$. Then it is easy to construct an extension $g'$ of $g|_{\partial B_\epsilon(0)}$ such that $g'(\theta) = g(0)$ for $|\theta| \leq \epsilon/2$ and $|g(\theta)-g(\theta)| <C$. Then $\data'$ is the desired good locally metric independent data.
\end{proof}

\subsection{Proof of the gluing theorem}\label{S:gluing} 
We now proceed through the approximation, linearization, and contraction mapping portions of the proof.

\subsubsection{Approximate gluing}
The structure of the gluing procedure can be illuminated by considering the disjoint union $\hat{N}_1\perp\!\!\perp\hat{N}_2$. The space of framed configurations
on the disjoint union is just the product of the spaces of framed configurations from the components. The gauge group on the disjoint union that corresponds to one on the connected sum is the following hybrid:
\[
\mathcal{G}_{\hat{N}_1\perp\!\!\perp\hat{N}_2} = \{(g,h) \in \mathcal{G}_{\hat{N}_1} \times \mathcal{G}_{\hat{N}_2}\,|\, \partial_\infty g = \partial_\infty h\}. 
\]
There is a well defined limiting value map $\partial_\infty\co \mathcal{G}_{\hat{N}_1\perp\!\!\perp\hat{N}_2}\to S^1$ with kernel $\mathcal{G}_{\hat{N}_1}^0\times\mathcal{G}_{\hat{N}_2}^0$.
It follows that
\begin{align*}
   \mathcal{G}_{\hat{N}_1\perp\!\!\perp\hat{N}_2}&\backslash \left(\widehat{SW}_{\data^1}^{-1}(0)\times \widehat{SW}_{\data^2}^{-1}(0)\right) \\ &\cong 
   \left( \left(\mathcal{G}_{\hat{N}_1}^0\times\mathcal{G}_{\hat{N}_2}^0\right)  \mathcal{G}_{\hat{N}_1\perp\!\!\perp\hat{N}_2}\right)\backslash\left(\mathcal{G}_{\hat{N}_1}^0\times\mathcal{G}_{\hat{N}_2}^0\backslash \left(\widehat{SW}_{\data^1}^{-1}(0)\times \widehat{SW}_{\data^2}^{-1}(0)\right)\right) \\
   &\cong 
S^1\backslash\left(\widehat{\mathcal{M}}_{\hat{N}_1,\spincs_1}(\data^1)\times \widehat{\mathcal{M}}_{\hat{N}_2,\spincs_2}(\data^2)\right) = \widehat{\mathcal{M}}_{\hat{N}_1,\spincs_1}(\data^1)\times_{S^1} \widehat{\mathcal{M}}_{\hat{N}_2,\spincs_2}(\data^2).
\end{align*}

To define the approximate gluing map, we use something called \emph{temporal gauge}. Recall that a connection, $A$, on a cylinder $I\times N$ is in temporal gauge when $\iota_{\partial_t}A = 0$. By definition of our configuration spaces $A \in L^2_{2,\delta}$, so 
$e^{\delta\tau}A\in L^2_2$. Here $\tau$ is a function equal to one on the cylindrical end. By a suitable trace-style embedding, \cite[Theorem 7.58]{Adams}, the restriction of $A$ to any ray $[0,\infty)\times\text{pt}$ is integrable. Define a based gauge transformation equal to $\text{exp}\left(-\frac12\int_t^\infty\iota_{\partial_t}A|_{(s,y)}\,ds \right)$ extended arbitrarily into the rest of $\hat{N}_1$. One checks that $g\cdot A = A - 2g^{-1}dg$ is in temporal gauge.  Restricted to $[0,\infty)\times S^3$ this is the unique based gauge transformation making $A$ temporal. Denote it by $g_A$.

Let $\chi_r\co [0,\infty) \to [0,1]$ be a smooth mollification of the characteristic function of $[0,r+\frac12]$ so that $\dot \chi_r$ has support in $[r,r+1]$. Use $\chi^1_r$ to denote the natural extension of this to $\hat{N}_1$. Define $\chi^2_r$ similarly. 
\begin{definition}
    The approximate sum of a pair of configurations is given by 
    \[
    (\theta_1,A_1,\psi_1)\cs_r (\theta_2,A_2,\psi_2) = 
    (\theta_1,\theta_2, g_{A_1}^{-1}\cdot(\chi_r^1 g_{A_1}\cdot A_1) + g_{A_2}^{-1}\cdot(\chi_r^1 g_{A_2}\cdot A_2), \chi_r^1\psi_1 + \chi_r^2\psi_2).
    \]
\end{definition}

The basic result about the approximate sum is that it is a well-defined embedding of the fibered product of the framed moduli spaces into the space of configurations on the connected sum mod the full gauge group.
\begin{proposition}\label{approx-sum}
The map $\cs_r\co \widehat{\mathcal{M}}_{\hat{N}_1,\spincs_1}(\data^1) \times_{S^1}\widehat{\mathcal{M}}_{\hat{N}_2,\spincs_2}(\data^2)\to \mathcal{B}_{N(r),\spincs_1\cs_r\spincs_2}$ is a well-defined embedding.     
\end{proposition}

The next task is to find all solutions close to an approximate configuration \hfill\newline $C_\approx = (\theta_\approx,\hat A_\approx, \hat \psi_\approx) = (\theta_1,A_1,\psi_1)\cs_r (\theta_2,A_2,\psi_2)$. This is split into two steps, capturing the linear behavior, and then the non-linear behavior.
Nearby configurations may be expressed as $C_\approx +(\vartheta, ia, \vphi)$. To address the gauge symmetry, introduce the gauge-fixing operator \cite[Section 2.2.2]{nicolaescu:swbook}:
\[
\mathcal{L}^*_{{C}_\approx}(\vartheta,ia,\vphi) = -2id^*a - i\text{Im}(\langle {\psi}_\approx,\vphi \rangle).
\]
Gauge-fixed solutions to the parameterized Seiberg-Witten equations close to ${C}_\approx$ are just the zeros of the non-linear bundle map given by 
\begin{equation}\label{eqN}
 \mathcal{N}(\vartheta,ia,\vphi) = \left(\widehat{SW}_{\data}({C}_\approx + (\vartheta,ia,\vphi)), \mathcal{L}^*_{{C}_\approx}(ia,\vphi)\right).   
\end{equation}

Call this the local, gauge-fixed, Seiberg Witten map.

When analyzing the equations on one fixed manifold one may pick a specific metric and thus fix the codomain of the local, gauge-fixed, Seiberg-Witten map. Even though this is not generally possible in families, Baraglia and Konno are able to reduce to this case on the reducible side by restricting to small neighborhoods of the parameter values where the perturbation crosses a wall, see \cite[Proposition~7.2]{baraglia-konno:gluing}. In fact, one may reduce to the case when the metric is independent of the parameter in sufficiently small neighborhoods of parameter values associated with elements of the parameterized moduli space whenever the parameterized moduli space is zero dimensional.  This result is established in Proposition~\ref{SPNMID} for irreducible configurations. 

In general, the domain of the gauge-fixed Seiberg-Witten map, $\mathcal{N}$ is the bundle over $\Xi$ with fiber ${\bf x}^+_2 = L^2_2(X,i\, \Bigwedge^1\oplus S^+)$, and the codomain is the bundle with fibers ${\bf x}^-_1 = L^2_1(X,i\, \Bigwedge^2_+\oplus S^-\oplus i\, \Bigwedge^0)$. The point of the local metric independent condition is that the domain and codomain will be trivial bundles. The domain is a trivial bundle so it is just $\mathbb{X}^+_2 = \Xi\times {\bf x}^+_2$. Similarly the codomain is just $\mathbb{X}^+_1 = \Xi\times {\bf x}^+_1$. 
% When we 
% change the regularity, we will use $\mathbb{X}^\pm_k$ to denote the analogous spaces defined using $L^2_k$ sections. 
When working on manifolds with cylindrical ends, we use weighted spaces and add a factor of $\R$ capturing the infinitesimal gauge transformations to $\mathbb{X}^+_1$. We suppress the weight from the notation.

\subsubsection{Linearized equations}\label{linearization}
Let $\mathcal{F}\co X\to Y$ be a $G$-equivariant map between linear spaces, and assume that $0\in Y$ is fixed by the action. Given a point $x_0\in \mathcal{F}^{-1}(0)$, one has a map $R\co G\to X$ given by multiplication. In this case, one may take the derivatives and obtain a short complex:
\[
\begin{tikzcd}\tag{{\textbf{def}}}
0\arrow{r} & T_1G \arrow{r}{D_{1}R} & T_{x_0}X \arrow{r}{D_{x_0}\mathcal{F}} &  T_0Y \arrow{r} &0.\\
\end{tikzcd}
\]
This complex is called the \emph{deformation complex} and is denoted $\textbf{def}$. The cohomology $H^0(\textbf{def})$ represents the tangent space to $\text{Stab}(x_0)$. The cohomology $H^2(\textbf{def})$ is called the \emph{obstruction space}. When the obstruction space vanishes the point $[x_0] \in G\backslash \mathcal{F}^{-1}(0)$ has a neighborhood isomorphic to $\text{Stab}(x_0)\backslash H^1(\textbf{def})$.

We now work in the setting of Theorem~\ref{pIRglue}, where the data is assumed to be irreducible-reducible good as in Definition~\ref{irr-red-good}. This means that we know the cohomology of the deformation complexes associated to the Seiberg-Witten moduli spaces over $\hat{N}_k$, for $k= 1, 2$. Our task is to compute the cohomology of the deformation complex associated to the Seiberg-Witten moduli space over $N(r)$. In the case without parameters \cite[Proposition 4.3.17]{nicolaescu:swbook} the deformation complex associated to the moduli space is denoted by $\widehat{\mathcal{K}}_{\widehat{C}_0}$. We will simplify notation and simply use $C$ for this deformation complex, even in the presence of parameters. The deformation complex, $F$, associated to the framed moduli space is defined just below \cite[Lemma 4.3.18]{nicolaescu:swbook}. Since the framed moduli space arises by restricting to the based gauge group, $F$ is a subcomplex of $C$. 

Notice that Nicolaescu defines various gauge groups in the setting where the $3$-manifold on the end has non-trivial $H^1$ \cite[Lemma 4.3.4]{nicolaescu:swbook}. In general, there are many cases to consider when analyzing the relationship between $F$ and $C$. However when the end is $[0,\infty)\times S^3$, the relation becomes particularly simple. The cohomology of the quotient complex is simply a copy of $\R$ concentrated in degree zero. This represents the tangent space to the stabilizer of the trivial connection on $S^3$ which is the limiting value of the connections considered in these moduli spaces. Note that even though the infinitesimal parameter space $T\Xi$ appears in both $F^1$ and $C^1$, it cancels in the quotient. The long exact sequence associated to  $0\to F \to C \to C/F \to 0$ implies that $H^2(F) = H^2(C)$. Furthermore, $H^0(F) = 0$ because the action of the based gauge group is always free, and the cohomology of $C/F$ is a copy of $\R$ concentrated in degree zero. We will show that for sufficiently large $r$, $H^0(C) = H^2(C) = 0$. The index computation and the long exact sequence will then imply that $H^1(C) = 0$ and $H^1(F) \cong \R$.  This corresponds to the fact that ${\mathcal{M}}_{{N}(r),\spincs}(\data)$  is a disjoint collection of points and   $\widehat{\mathcal{M}}_{{N}(r),\spincs}(\data)$  is a disjoint union of circles. 

We now continue the analysis for $\hat{N}_1$ and $\hat{N}_2$ separately. 
Any parameterized monopoles on $\hat{N}_1$ is irreducible, isolated, and non-degenerate. This implies that $H^0(C) = 0$, $H^1(C) = 0$, and $H^2(C) = 0$. It follows that $H^0(F) = 0$, $H^1(F) \cong \R$, and $H^2(F) = 0$. This corresponds perfectly with the fact that the moduli space is a finite set of points and the framed moduli space is the disjoint union of a finite collection of circles with a free $S^1$-action.

For $\hat{N}_2$ every monopole is reducible ($H^0(C) \cong \R$), and by assumption $H^1(C) = 0$. By the index condition of irreducible-reducible good data (Definition~\ref{irr-red-good}), we have $H^2(C) = 0$. This implies that $H^0(F) = 0$, $H^1(F) = 0$, and $H^2(F) = 0$.

The linear gluing theory from \cite{clm:I,Nicolaescu:CLM} describes how the linearized equations over  
$\hat{N}_1$, $\hat{N}_2$, and $N(r)$ relate in the non-parameterized case. The same paper states that it works for families of Dirac operators as well. However our situation adds the derivative of the perturbation $i\eta(\theta)^+$. The linearization of this is not a Dirac operator. Thus we need a slight modification of this theory.
We review this theory before providing a slight generalization.

Let $\hat{E} = \hat{E}^+\oplus\hat{E}^-$ be a cylindrical bundle over a cylindrical manifold $\hat{N}$ with limiting bundle $E\to N$.
Define the extended $L^2$ norm on 
$L^2(\hat{E})_{ex} = L^2(\hat{E})\oplus L^2(E)$ 
by
\[
\|(u,v)\|_\text{ex}^2 = \|u\|_{L^2(\hat{E})}^2 + \|v\|_{L^2(E)}^2.
\]
Given a positive smooth cut-off function $h:\hat{N} \to [0,1]$ with $h([0,\infty)\times N) = \{1\}$ and  $h(\hat{N}\setminus [-1,\infty)\times N) = \{0\}$, one may view elements of $L^2(\hat{E})_{ex}$ as sections of $\hat{E}$ via $(u,v)\mapsto u + hv$. In this context it is natural to write $u$ for the first component and $\partial_\infty u$ for the second component. 
Consider $\mathbb{D}:L^2(\hat{E}^+) \to L^2(\hat{E}^-)$ a generalized Dirac operator (viewed as a densely defined unbounded operator), and set 
\[
\hat{D} = \begin{bmatrix} 0 & \mathbb{D}^*\\ \mathbb{D} & 0
\end{bmatrix}. 
\]
On the neck use the symbol to define an almost complex structure by $J = \sigma_{\hat{D}}(dt)$ so that $\hat{D}$ has the form $J(\partial_t-D)$ on the neck. The main theorem of \cite{Nicolaescu:CLM} considers a pair of such operators defined over the cylindrical manifolds $\hat{N}_1$ and $\hat{N}_2$ that agree on the neck. These define an operator $D(r)$ over the associated closed manifolds $N(r)$. In \cite{Nicolaescu:CLM} the span of the eigenfunctions of $D(r)$ having small eigenvalues is denoted by $\mathcal{K}_r(c)$ and small means in the range $(-c(r),c(r))$ for suitable function $c$. In \cite{baraglia-konno:gluing}, $H_r$ is used to denote the span of eigenfunctions of $D(r)$ in the range $(-r^{-2},r^{-2})$. We use $\mathcal{K}(r)$ to denote this space. Since eigenfunctions are smooth, there are many ways to spit an eigenfunction $u$ defined over $N(r)$ into functions defined over $\hat{N}_1$ and $\hat{N}_2$. In \cite{Nicolaescu:CLM} this map $S$ is defined by setting $S^1_ru(t,y) = u(r,y)$ for $(t,y)$ in the neck with $t > r$ and equal to $u$ elsewhere with a similar formula for the component $S^2_ru$. In \cite{nicolaescu:swbook} an averaging procedure is employed. The main conclusion is that there is an $R$ so that for all $r>R$ the following is an exact sequence.
\[
0 \to \mathcal{K}(r) \stackrel{p_rS_r}{\longrightarrow}\text{ker}_{\text{ex}}(\hat{D}_1)\oplus\text{ker}_{\text{ex}}(\hat{D}_2) \stackrel{\Delta}{\longrightarrow}
\partial_\infty\text{ker}_{\text{ex}}(\hat{D}_1) + \partial_\infty\text{ker}_{\text{ex}}(\hat{D}_2) \to 0,
\]
where $p_r$ is orthogonal projection to the image, $\text{ker}_{\text{ex}}(\hat{D}_1) = \text{ker}(\hat{D}_1)\cap L^2_{\text{ex}}(\hat{E}^+\oplus \hat{E}^-)$, and $\Delta(u,v) = \partial_\infty u - \partial_\infty v$.
That this sequence is exact follows by applying \cite[Lemma 2.2]{Nicolaescu:CLM} to \cite[Remark 4.3]{Nicolaescu:CLM}.

We need a generalization of this to cover the case of the sum of a finite rank linear operator with a generalized Dirac operator. In particular, let $\eta_i$, $i=1,\dots,n$ be a finite collection of smooth compactly supported sections of $\hat{E}^-$, and for $\theta\in\R^n$ set ${\bf n}(\theta) = \theta^i\eta_i$. We now consider the operator $\widetilde{D}:\R^n\oplus L^2_{\text{ex}}(\hat{E}^+\oplus \hat{E}^-) \to \R^n\oplus L^2_{\text{ex}}(\hat{E}^+\oplus \hat{E}^-)$ given by $\widetilde{D}(\theta,u,v) = ({\bf n}^*(v),\mathbb{D}^*v,\mathbb{D}u + {\bf n}(\theta))$. 

\begin{lemma}[Parameterized linear gluing]
If $\widetilde{D}_1$ and $\widetilde{D}_2$ are parameterized, cylindrical, Dirac operators over $\hat{N}_i$, and $\widetilde{D}(r)$ is the corresponding operator defined over $N(r)$, there is an $R$ so that for all $r>R$ following is an exact sequence:
\[
0 \to \mathcal{K}(r) \stackrel{p_rS_r}{\longrightarrow}\text{ker}_{\text{ex}}(\widetilde{D}_1)\oplus\text{ker}_{\text{ex}}(\widetilde{D}_2) \stackrel{\Delta}{\longrightarrow}
\partial_\infty\text{ker}_{\text{ex}}(\widetilde{D}_1)+\partial_\infty\text{ker}_{\text{ex}}(\widetilde{D}_1) \to 0.
\]
\end{lemma}

\begin{proof}
Following the proof from \cite{Nicolaescu:CLM} one starts with a bounded sequence in the image of the splitting map and assumes that it is bounded away from the sum of the kernels. Establishing the existence of a convergent subsequence gives a contradiction and allows one to conclude that the image of the splitting map approaches the sum of the kernels. The norm is given by $\|(\theta,V)\|^2 = \|\theta\|^2 + \|V\|_\text{ex}^2$. Thus the corresponding sequence of $\theta$s is bounded implying that there is a convergent subsequence. One notices that a bound on $\|S_r(\theta_r,V_r)\|$ gives a bound on $\|(\theta_r,V_r)\|$. It is easy to see that one has $\widetilde{D} = \hat{D}$ away from the neck (define the neck to start after the support of $\eta_i$), uniform bounds $\|{\bf n}(\theta)\|_{L^2_k} \le C|\theta|$, $|{\bf n}^*v| \le C\|v\|_{L^2}$, and the inequality
\[
\|[\hat{D},\widetilde{D}]V\|_{L^2} \le C\left(\|[\hat{D},\widetilde{D}]V\|_{L^2(N(r)\setminus\text{neck})} + \|[\hat{D},\widetilde{D}]V\|_{L^2(N(r)\setminus\text{neck})}\right).
\]
One quickly obtains the bounds
\[
\|\hat{D}^pV\|^2 \le C\sum_{k=0}^p \|\widetilde{D}^kV\|^2,
\]
with the constant independent of $r$.  The splitting part of the proof then follows the argument from \cite{Nicolaescu:CLM}, which is a bootstrapping argument to obtain a convergent subsequence which is show to converge to an element of $\text{ker}(\Delta)$. We remark that it is in this first step that the eigenvalue bound is required as one of the terms that is required to vanish in the limit is $r\lambda$ where $\lambda$ is one of the eigenvalues corresponding to an element of $\mathcal{K}(r)$. The second step is to define an approximate gluing $V_1\#_rV_2$ so that $\text{sup}\frac{\|\widetilde{D}(V_1\#_rV_2)\|}{\|V_1\#_rV_2\|} < r^{-2}$. Since $\widetilde{D}(V_1\#_rV_2)$ vanishes outside of the neck and ${\bf n}(\theta)$ vanishes inside the neck, the argument here is unchanged from \cite{Nicolaescu:CLM}.   \end{proof}

Since the operator $\widetilde{D}$ maps $\R^n\oplus L^2_\text{ex}(\hat{E}^+)$ to $L^2_\text{ex}(\hat{E}^-)$ and vice-versa, one may decompose $\mathcal{K}(r) = \mathcal{K}^+(r)\oplus \mathcal{K}^-(r)$. The abstract linear gluing result uses information about $\text{ker}_{\text{ex}}$ as input, but 
we only currently understand the cohomology of the $C$ and $F$ complexes associated to $\hat{N}_k$, $k=1,2$. It is useful to collapse a general Hilbert complex $(C^j,d^j)$ into a single map $d+d^*:C^* = \oplus\, C^j \to C^*$. One obtains $\text{ker}(d+d^*) = \oplus H^j(C)$.

Recall that 
\begin{gather*}
    F_k^* = T_{\theta_0}\Xi\oplus L^2_{\0,\delta}(i(\Bigwedge^0\oplus\Bigwedge^1\oplus\Bigwedge^2_+)\hat{N}_k\oplus S) \\
    C_k^* = T_{\theta_0}\Xi\oplus L^2_{\0,\delta}(i(\Bigwedge^0\oplus\Bigwedge^1\oplus\Bigwedge^2_+)\hat{N}_k\oplus S)\oplus i\R \\
    (C_{ex})_k^* = T_{\theta_0}\Xi\oplus L^2_{\0,\delta}(i(\Bigwedge^0\oplus\Bigwedge^1\oplus\Bigwedge^2_+)\hat{N}_k\oplus S) \oplus L^2(i(\Bigwedge^0\oplus\Bigwedge^1)S^3\oplus S).
\end{gather*}
Since we are collapsing the complexes, the notation here does not include the number of derivatives used in each summand.  The collapsed differential in each case is just $\widetilde{D} = D_0\mathcal{N}+D_0\mathcal{N}^*$ restricted to the relevant space.
These complexes are defined and studied in the unparameterized setting in detail in \cite[Chapter 4]{nicolaescu:swbook}. 
Notice that the last summand of the extended complex $C_{ex}^*$ is just the deformation complex of the configuration at infinity. When the configuration at infinity is non-degenerate, $\text{ker}(\partial_\infty\widetilde{D})$ projects to a subspace of $L^2(i\Bigwedge^0S^3)$ and a subspace of $L^2(i\Bigwedge^1S^3\oplus S)$ where $S$ is the spinor bundle on $S^3$.
The map taking values at infinity has a corresponding decomposition $\partial_\infty = \partial_\infty^0 + \partial_\infty^c$ described just before \cite[Remark 4.3.27]{nicolaescu:swbook}.

This leads to \cite[Proposition 4.3.28]{nicolaescu:swbook} which asserts the existence of the exact sequence
\[
0\to H^1(C) \to \text{ker}_{ex}(\mathbb{D}+{\bf n}) \to T_1\left(\text{Stab}(\partial_\infty C)/\partial_\infty\text{Stab}(C))\right) \to 0.
\]
The proof takes place on the neck. Since the perturbation vanishes on the neck, the same result holds in the parameterized case. The same decomposition leads to  
\[
0\to H^2(F)\to \text{ker}_{ex}(\mathbb{D}^*+{\bf n}^*) \to \text{im}(T_1\text{Stab}(C) \to T_1\text{Stab}(\partial_\infty C)) \to 0.
\]
The proof is based on an analysis of the relation between gauge fixing conditions arising from different weights. Since the gauge fixing conditions do not depend on the perturbation, the proof carries over to the parameterized case. It is useful to set 
\[
\mathfrak{G}^+ = T_1\left(\text{Stab}(\partial_\infty C)/\partial_\infty\text{Stab}(C))\right) \ \text{and} \ \mathfrak{G}^- =\text{im}(T_1\text{Stab}(C) \to T_1\text{Stab}(\partial_\infty C)).
\]

The conclusion is that in the obstructed case, the tangent space diagram \eqref{tangent} and obstruction space diagram  \eqref{obstruction} from before Theorem 4.5.7 in \cite{nicolaescu:swbook} both continue to commute and have exact rows and columns.
% \danny{It used to say `hold'  here. Is my interpretation OK?}
\begin{equation}\label{tangent}\tag{{\bf T}}
\begin{tikzcd}[column sep = small]
0\arrow{r} & \ker(\Delta_+^c) \arrow{r}\arrow{d} &  H^1(C_1)\oplus H^1(C_2) \arrow{r}\arrow{d} &  \partial_\infty^cH^1(C_1) \oplus \partial_\infty^cH^1(C_1) \arrow{r}\arrow{d} &  0 \\
0 \arrow{r} &  \mathcal{K}(r)^+ \arrow{r}{p_rS_r}\arrow{d} & \ker_{\text{ex}}^+(\widetilde{D}_1)\oplus\ker_{\text{ex}}^+(\widetilde{D}_2) 
\arrow{r}{\Delta}\arrow{d}
 &\partial_\infty\ker_{\text{ex}}^+(\widetilde{D}_1) + \partial_\infty\ker_{\text{ex}}^+(\widetilde{D}_2)\arrow{r}\arrow{d} & 0\\
0 \arrow{r}   &  \text{ker}(\Delta_+^0) \arrow{r} &   \mathfrak{G}^+_1 \oplus \mathfrak{G}^+_2 \arrow{r} &  \mathfrak{G}^+_1 + \mathfrak{G}^+_2 \arrow{r} &  0 
\end{tikzcd}
\end{equation}
\vspace*{1ex}

\begin{equation}\label{obstruction}\tag{{\bf O}}
\begin{tikzcd}[column sep = small]
0\arrow{r} & \ker(\Delta_-^c) \arrow{r}\arrow{d} &  H^1(F_1)\oplus H^1(F_2) \arrow{r}\arrow{d} &  \partial_\infty^cH^1(F_1) \oplus \partial_\infty^cH^1(F_1) \arrow{r}\arrow{d} &  0 \\
0 \arrow{r} &  \mathcal{K}(r)^- \arrow{r}{p_rS_r}\arrow{d} & \ker_{\text{ex}}^-(\widetilde{D}_1)\oplus\ker_{\text{ex}}^-(\widetilde{D}_2) 
\arrow{r}{\Delta}\arrow{d}
 &\partial_\infty\ker_{\text{ex}}^-(\widetilde{D}_1) + \partial_\infty\ker_{\text{ex}}^+(\widetilde{D}_2)\arrow{r}\arrow{d} & 0\\
0 \arrow{r}   &  \text{ker}(\Delta_-^0) \arrow{r} &   \mathfrak{G}^-_1 \oplus \mathfrak{G}^-_2 \arrow{r} &  \mathfrak{G}^-_1 + \mathfrak{G}^-_2 \arrow{r} &  0
\end{tikzcd}
\end{equation}

Since $H^1(C_k)$ and $H^2(F_k)$ are all zero, the first row of each diagram is zero. Furthermore, as $\partial_\infty C_k$ is reducible for each $k$, $C_1$ is irreducible and $C_2$ is reducible we have $\mathfrak{G}^+_1\cong \R$, $\mathfrak{G}^+_2 = 0$, $\mathfrak{G}^-_1 = 0$, and $\mathfrak{G}^-_2 \cong \R$. This implies that $\mathcal{K}(r) = 0$. Thus once $r$ is larger than some threshold, $R_L$, the operator $D_0\mathcal{N}$ is an isomorphism and $\widetilde{D} = D_0\mathcal{N}+D_0\mathcal{N}^*$ has no eigenvalues in $(-r^{-2},r^{-2})$. It follows that $D_0\mathcal{N}:L^2_2 \to L^2_1$ has a bounded inverse, $S:L^2_1\to L^2_2$, with operator norm satisfying $\|S\|_{op}\le r^2$.

\subsubsection{Contraction Mapping}\label{contraction}
Now that the linearized equations are understood, we can apply the standard contraction mapping argument to show that approximate solutions may be converted into actual solutions. This follows the argument from \cite{nicolaescu:swbook} as presented in \cite{baraglia-konno:gluing} exactly. Recall that the gauge-fixed Seiberg-Witten map from Equation~\ref{eqN} is a map
\[
\mathcal{N}\co \mathbb{X}^+ \to \mathbb{X}^-.
\]
 We will often write the argument $(\vartheta,ia,\varphi)$ as $x$ or $y$. In this notation, our task is to show that there is a unique solution to $\mathcal{N}(x) = 0$ in a sufficiently small neighborhood of $0$.
The Taylor expansion with error reads
\[
0 = \mathcal{N}(0) + D_0\mathcal{N}(x) + Q(x).
\]
Applying the operator $S$ transforms this to
\[
x = F(x) = - S\mathcal{N}(0) - SQ(x).  
\]
% \danny{Comment that notation slight difference from Liviu; he uses R for our S; his Q means something related to obstructions that we don't have?}
We will show that once $r$ is larger than a threshold $R_c$, the map $F$ is a contraction on $D_r = \{x \in L^2_2\,|\, \|x\| \le r^{-4} \}$.

Here is is useful to divide $\mathcal{N}(\vartheta,ia,\varphi) = \mathcal{N}_0(ia,\varphi) +i\eta(\theta+ \vartheta)^+$. The function $\mathcal{N}_0$ was analyzed in the unparameterized case. Since $\eta$ is a smooth function with a compact parameter space, it is not difficult to estimate terms associated to $\eta$. This decomposition leads to a corresponding decomposition of the quadratic error term
\[
Q(\theta,A,\psi) = Q_0(A,\psi) + q_\eta(\theta).
\]
We use the following standard calculus lemma.
\begin{lemma}\label{L:calculus}
 Let $f:\R^n\to H$, H a Hilbert space, have bounded second derivatives, and let $R(\theta) = f(\theta) - [f(0) +f'(0)\theta]$ be the error in the linear approximation. There is a constant $C$ to that
\[
\|R(\theta) - R(\phi)\| \le C(\|\theta\|+\|\phi\|)\|\theta - \phi\|.
\]   
\end{lemma}

In this case, the Hilbert space depends on the parameter $r$. However, this is really not an issue because the support of $\eta$ is disjoint from the neck which is the only part that depends upon $r$. It follows that 
\begin{align*}
  \|Sq_\eta(\theta_1) - Sq_\eta(\theta_2)\|_{L^2_2} &\le r^2\|q_\eta(\theta_1) - q_\eta(\theta_2)\|_{L^2_1} \\
&\le Cr^2(|\theta_1| + |\theta_2|)|\theta_1 - \theta_2|  \\
&\le Cr^2(\|(\theta_1,A_1,\psi_1)\|+\|(\theta_2,A_2,\psi_2)\|)|\theta_1 - \theta_2|.
\end{align*}

The estimate 
\[
\|Q_0(A_1,\psi_1) - Q_0(A_2,\psi_2)\|_{L^2_1} \le Cr^{3/2}(\|(A_1,\psi_1)\|_{L^2_2} + \|(A_2,\psi_2)\|_{L^2_2})\|(A_1,\psi_1) - (A_1,\psi_1)\|_{L^2_2}
\]
is established in \cite[Lemma 4.5.6]{nicolaescu:swbook}.
It follows that 
\[
\|SQ_0(x) - SQ_0(y)\|_{L^2_1} \le Cr^{7/2}(\|(A_1,\psi_1)\|_{L^2_2} + \|(A_2,\psi_2)\|_{L^2_2})\|(A_1,\psi_1) - (A_1,\psi_1)\|_{L^2_2}.
\]
Combining these gives
\[
\|F(x) - F(y)\|_{L^2_2} \le Cr^{7/2}(\|x\|_{L^2_2}+\|y\|_{L^2_2})\|x - y\|_{L^2_2}.
\]
Using $x, y\in D_r$ then implies the contraction property $\|F(x) - F(y)\|_{L^2_2} \le Cr^{-1/2}\|x - y\|_{L^2_2}$ with say $Cr^{-1/2}\le \frac12$ once $r$ is larger than a suitable threshold. Substituting $y=0$ gives $\|F(x) - F(0)\|_{L^2_2} \le Cr^{-1/2}\|x\|_{L^2_2}$

The expression $\mathcal{N}(0)$ just measures the failure of the approximate solution $C_1\cs_rC_2$ from being an actual solution to the gauge-fixed Seiberg-Witten equations. This failure only takes place in the neck where the perturbation vanishes. \cite[Theorem 4.2.33]{nicolaescu:swbook} states that for any $\lambda < \delta$ there is a constant $C$ so that any solution to the unperturbed perturbed gauge-fixed Seiberg-Witten equations satisfies $\|x(t)\|_{L^2_2} \le C e^{-\lambda t}$. This gives a bound.
\[
\|S\mathcal{N}(0)\|_{L^2_2}\le Cr^2e^{-\lambda r}.
\]
Thus for $x\in D_r$ we have
\begin{align*}
\|F(x)\| &\le \|F(0)\| + \|F(x) - F(0)\|\\
&\le  Cr^2e^{-\lambda r} +  Cr^{-1/2}\|x\|_{L^2_2} \le C(r^6 e^{-\lambda r} + r^{-1/2})r^{-4}.
\end{align*}
For $r$ larger than a suitable threshold this implies that $\|F(x)\| \le r^{-4}$ so that $F$ is a contraction mapping and we conclude that it has a unique fixed point.

  \subsubsection{Checking that the gluing map is an isomorphism}
\begin{definition}
    The gluing map 
    $\cs\co \widehat{\mathcal{M}}_{\hat{N}_1,\spincs_1}(\data^1) \times_{S^1}\widehat{\mathcal{M}}_{\hat{N}_2,\spincs_2}(\data^2)\to {\mathcal{M}}_{N(r),\spincs_1\cs\spincs_2}(\data^1\cs\data^2)$
    is given by
    \[
    (\theta_1,A_1,\psi_1)\cs (\theta_2,A_2,\psi_2) = (\theta_1,A_1,\psi_1)\cs_r (\theta_2,A_2,\psi_2) + x,
    \]
    where 
    \[
    F_{(\theta_1,A_1,\psi_1)\cs_r (\theta_2,A_2,\psi_2)}(x) = x.
    \]
\end{definition}
Our work here is simplified by the fact that the parameterized moduli spaces that we are considering are all finite, thus each solution to the family equations has a specific parameter. Furthermore, the $I\times S^3$ neck that we consider does not admit tunnel solutions; compare~\cite[\S 4.5.4]{nicolaescu:swbook}. The construction of a local slice to the gauge group so that each gauge orbit in a sufficiently small neighborhood of the solution hits the slice in exactly one point proceeds exactly as in the unparameterized case. 

The assumption that 
\[
(\theta_1,A_1,\psi_1)\cs (\theta_2,A_2,\psi_2) = (\theta_1',A_1',\psi_1')\cs (\theta_2',A_2',\psi_2')
\]
implies that $(\theta_1,A_1,\psi_1)\cs_r (\theta_2,A_2,\psi_2)$
and $(\theta_1',A_1',\psi_1')\cs (\theta_2',A_2',\psi_2')$ are close (within $2r^{-4}$). 
This in turn implies that $(\theta_k,A_k,\psi_k)$ and $(\theta_k',A_k',\psi_k')$ are close. 
This follows because the size of the solution on the end $[a,\infty)\times S^3$ is bounded by the size of the solution of the rest of $\hat{N}_k$ and decays exponentially in $a$. 
Once we conclude that $(\theta_k,A_k,\psi_k)$ and $(\theta_k',A_k',\psi_k')$ are sufficiently close, it follows that they must be equal since the parameterized moduli spaces are discrete. It follows that the gluing map is injective for all sufficiently large $r$.

To prove that the the gluing map is surjective is to make a cut-off to turn any solution on $N(r)$ into a pair of approximate solutions on $\hat{N}_k$. There are two main points. The size of a solution in the center of the neck is bounded  by a factor that decays exponentially in the length of the neck times the size on the of the solution on the ends of the neck. Up to gauge, there is a unique solution on $I\times S^3$. Thus once $r$ is sufficiently large, a solution, $C$ on $N(r)$ gives rise to an approximate solution $C^\approx_1$ on $\hat{N}_1$, which may be turned into an actual solution $C_1$. 
This is done so that $C$ is within a distance of $r^{-4}$ of $C_1\cs_r C_2$ and hence is in the image of the gluing map. Since the gluing map is a bijection, we have established Theorem~\ref{pIRglue}.

\subsection{Applications of the irreducible-reducible gluing theorem}\label{1st-apps}
With Theorem~\ref{pIRglue} proved, in hand, we deduce several applications that have been used earlier in the paper. 
When family invariants are well-defined for a family of manifolds $\hat{N}_1$ with cylindrical ends, Proposition~\ref{SPNMID} establishes that there is data on the family satisfying the conditions required for $\hat{N}_1$ in Theorem~\ref{pIRglue}. Thus the main items to establish before applying the theorem are the conditions associated with the manifold $N_2$. Generally, when the parameter space is compact, the usual arguments in Seiberg-Witten theory will imply that the parameterized moduli space is compact. If the virtual dimension of the parameterized moduli space is negative and the data is generic, the parameterized moduli space will contain only reducible solutions. In the negative virtual dimension case with reducibles, the natural restriction to impose is that the first cohomology of the deformation complex vanishes at a reducible point. Establishing this is the most delicate part of applying this theorem.  In many applications, one may take positive scalar curvature metrics on $N_2$ to establish many of the properties of an irreducible-reducible good pair associated with $N_2$.

There are various situations where one would like to apply Theorem~\ref{pIRglue}. One of the most basic is when $\hat N_2$ is $\R^4$ with scalar curvature bounded below by a positive constant, a cylindrical end modeled on $(0,\infty)\times S^3$, $\Xi_2$ a point, and $\eta = 0$. This allows one to pass between a manifold or family of manifolds with cylindrical end modeled on $S^3$ and a closed manifold or family of closed manifolds.  One can easily construct such a metric on $\R^4$ for example, working in spherical coordinates  take a metric of the form 
\[
g = dr^2 + (\rho d\alpha^1)^2 + (\rho \sin\alpha^1 d\alpha^2)^2 + (\rho \sin\alpha^1\sin\alpha^2 d\alpha^3)^2,
\]
where $\rho$ is a function of $r$. A direct calculation shows that the scalar curvature of this metric is given by
$s = 6(\rho^{-2} - \rho^{-2}(\rho')^{2} - \rho^{-1}\rho^{\prime\prime})$. The choices $\rho = r$, $\rho = \sin(r)$, and $\rho = 1$ correspond to $\R^4$, $S^4$, and $\R\times S^3$ respectively.  It is not difficult to pick a function that interpolates between $\sin(r)$ near $r=0$ and $1$ for large values of $r$ so that the scalar curvature is strictly positive. Nicolaescu \cite[Example 4.3.40]{nicolaescu:swbook} argues that $H^1 = 0$ in this situation. This metric (and vanishing $2$-form) will be called the \emph{standard data} on $\hat{S^4}$. In this setting, the closed/deleted correspondence is straightforward.
\begin{corollary}\label{closed-deleted} With respect to a sufficiently stretched metric on $N$, the moduli space $\mathcal{M}_{{N},\spincs}({\data})$ is isomorphic to the cylindrical end moduli space $\mathcal{M}_{\hat{N},\spincs}(\hat{\data})$.
\end{corollary}
A second situation arises when $\hat N_2$ is the Euler class $-1$ bundle over $S^2$ with a cylindrical end modeled on $(0,\infty)\times S^3$, a metric of scalar curvature bounded below by a positive constant, $\Xi_2$ is a point, and $\eta = 0$. This leads to the usual blow-up formula (for manifolds of simple type) in Seiberg-Witten theory. Nicolaescu \cite[Example 4.1.27]{nicolaescu:swbook} argues that any complex line bundle over $S^2$ admits a PSC metric with cylindrical end modeled on a ray times the standard metric lens space that is the boundary of the corresponding disk bundle. That $H^1 = 0$ exactly for the $\Spinc$ structures with $c_1(\spincs) = \pm[\cpone]$ is proved in \cite[Example 4.3.39]{nicolaescu:swbook} and this establishes the usual Seiberg-Witten blow-up formula. 

A third application arises when $\hat N_2$ is a punctured copy of $\cptwo\cs 2\cptwobar$. Since $\cptwo$ admits a PSC metric independent of orientation and a once punctured $\cptwo$ may be identified with the Euler class one bundle over $S^2$, we see that $\hat N_2$ admits a PSC metric with cylindrical end modeled on $\R\times S^3$, \cite{gromov-lawson:psc}. For a parameterized moduli space corresponding to a path of metrics, one evaluates $H^0$ using the fact that the solution is reducible and $H^2$ with the assistance of the PSC condition, and then applies the index theorem to conclude that $H^1 = 0$. This then gives the wall crossing formula from~\cite{ruberman:swpos} for the diffeomorphism invariant. The application that is relevant to this paper arises when $\hat N_2$ is a punctured copy of $\sss$ and it is very similar to the example in~\cite{ruberman:swpos} of the wall crossing formula. Before describing the case of a punctured $\sss$ we remark that the gluing result of Baraglia-Konno~\cite{baraglia-konno:gluing} also fits into the framework of Theorem~\ref{pIRglue}. After reducing to the situation of families defined over cells, their gluing theorem corresponds to the case with $\Xi_1$ just a point. 

Now we turn to the case of the punctured $\sss$ that is relevant to this paper. 
\begin{lemma}\label{L:standardsssdata}
    There is data $\data\colon [0,1] \to \Pi(\sss)$ for which $\data(1) = (RR')^*\data(0)$ that is locally metric independent and such that the first homology of the deformation complex of the unique Seiberg-Witten monopole associated with this data vanishes.
\end{lemma}
The explicit diffeomorphism  $RR'$ on $\sss$ is defined in Remark~\ref{explicitRR}. We call the data constructed in the proof \emph{standard data} on $\sss$.
\begin{proof}
The map $RR'$ is an isometry of $\sss$ with the standard metric. The steps used to introduce a spherical cylindrical end in the constructions of a connected sum of positive scalar curvature 
metrics from~\cite{gromov-lawson:psc,rosenberg-stolz:psc} are equivariant with respect to the $RR'$ symmetry when performed at a point on the fixed-point set. Denote the resulting metric by $g_0$. When restricted to the cylindrical end of the punctured $\sss$, $RR'$ is just a rotation of the $S^3$-factor. In spherical coordinates $(\vphi^1,\vphi^2,\vphi^3)$ it just adds $\pi$ to the $\vphi^3$ coordinate. Use $t$ as a coordinate on the $\R$-factor of $(0,\infty)\times S^3$. To define the connected sum, we perform a small isotopy 
to make it restrict to the identity on the end of the punctured $\sss$. This may be accomplished by making the rotation a function of $t$, for instance via smooth function $h(t)$ that ramps from a constant value of $\pi$ down to a constant value of zero.
This allows one to write an explicit family of metrics interpolating between $g_0$ and $g_1 = RR^*g_0$. Since $g_0$ has positive scalar curvature, $g_1$ also has positive scalar curvature. Restricted to the cylindrical end, the metric $g_s$ described below arises as the pull-back of a diffeomorphism that rotates one end of the cylinder by $s\pi$, while fixing the other.
Actually we use a smooth monotone increasing function, $\sigma\co[0,1]\to[0,1]$, with $\sigma(s) = s$ on a neighborhood of $\{0,1\}$ and $\sigma(s) = \frac12$ in a neighborhood of $\frac12$.
As a pull-back this is still PSC. Even though rotations that are not multiples of $\pi$ do not extend across $\sss\setminus\text{pt}$, the induced metrics do match since rotation is an isometry of $S^3$.
\[
 g_s = (dt)^2 + (d\vphi^1)^2 + (\sin\vphi^1 d\vphi^2)^2 + (\sin\vphi^1\sin\vphi^2d\vphi^3 + \sigma(s)\dot h(t)\, dt)^2.
\]
This is the \emph{standard family of metrics} on $\sss$.
To complete this into a family of data, we need to add suitable perturbations. 

Following \cite[Proposition 4.3.14]{donaldson-kronheimer} one may pick a $2$-form $\eta_0$ with non-zero $g_0$-self-dual part. It follows that the pullback $\eta_1 = RR^*\eta_0$ has non-zero $g_1$-self-dual part. Since $RR$ acts as $-1$ on the one-dimensional positive part of the second homology any path joining $\eta_0$ to $\eta_1$ must pass the wall consisting of ASD perturbations. The argument of \cite[Proposition 4.3.14]{donaldson-kronheimer} shows that this intersection may be taken to be transverse. Transversality to the wall is an open condition \cite[Lemma 7.3]{baraglia-konno:gluing} and the space of compactly supported sections is dense in the weighted Sobolev space, so we may assume that these perturbations have support disjoint from the end. 
We then take the \emph{standard data} on $\sss$ to be $\data_s = (g_s,\eta_s)$. Without loss of generality, we will assume that this only crosses the wall at $s=1/2$. Note that by construction, it is locally metric independent. 

For any Seiberg-Witten configuration $(A,\psi,s)$ in the parameterized moduli space the PSC condition combined with the Weitzenb\"ock formula implies that $\psi = 0$ and $\ker(D^+_A) = 0$. Thus the configuration is reducible so $s=\frac12$ and $H^0_{A,0,\frac12} = \R$.
Using the round metric on $S^3$ as in \cite[4.3.4]{nicolaescu:swbook} one sees that the Atiyah-Patodi-Singer theorem gives $\text{index}(D^+_A) = 0$, so that $D^+_A$ has trivial cokernel. It follows that $H^2_{A,0,\frac12}$ may be identified with the cokernel of 
$\frac{d}{ds}\eta^+_s|_{s=0}:T_{1/2}[0,1] \to \mathcal{H}^2_+(\widehat{\sss},g_{1/2})$ which vanishes since $\eta_s$ is transverse to the wall. Now the virtual dimension satisfies
\begin{align*}
 \vdim\mathcal{M}_{\widehat{\sss},\spincs_0}(\{\data_s\}_{s\in [0,1]}) &= \dim(H^1_{A,0,\frac12}) - \dim(H^0_{A,0,\frac12}) - \dim(H^2_{A,0,\frac12})    \\ &= 1 + \vdim\mathcal{M}_{\widehat{\sss},\spincs_0}(\data_0) \\
 &= 1 + \frac14\left(c_1(\spincs_0)^2[\sss] - 2\chi(\sss) - 3\sigma(\sss)\right) = -1.
\end{align*}
It follows that $H^1_{A,0,0} = 0$ as required to use Theorem~\ref{pIRglue}. 
\end{proof}

\begin{proof}[Proof of Corollary~\ref{glue}] Note that we are assuming that the data on $Z$ is irreducible-reducible good, and that by Proposition~\ref{SPNMID} we can assume further that it is locally metric independent. Then the lemma is a direct consequence of Theorem~\ref{pIRglue}, which shows that the parameterized moduli space for $Z \cs \sss$ for the given $\Spinc$ structure and data is the product of the moduli space on $Z$  and the $1$-parameter moduli space for $\sss$ with standard data. But we showed in Lemma~\ref{L:standardsssdata} that the moduli space for the standard data on $\sss$ has one point at which the data is locally metric independent. 
% \danny{Proof of this corollary should point to where we verify conditions of def \ref{irr-red-good} } 
\end{proof}

\begin{remark}\label{R:final}
We close with a few comments on ideas related to the work in this paper. 
    \begin{enumerate}
        \item We have considered family invariants arising from $0$-dimensional family moduli spaces, but it would certainly be possible to consider numerical invariants arising from higher-dimensional. This would involve counting solutions on divisors corresponding to the so-called $\mu$-map. Much of the theory here would apply, including the proof of Theorem~\ref{pIRglue}, since the moduli spaces (intersected with those divisors) involved are $0$-dimensional. This was explicitly discussed in the paper of Baraglia-Konno~\cite{baraglia-konno:gluing}.

        The simple type~\cite[Definition 2.3.6]{nicolaescu:swbook} conjecture for the usual \SW invariants says that when $b^2_+ > 1$, only zero dimensional moduli spaces contribute to the Seiberg-Witten invariants. The analogous conjecture for family invariants would be that when $b^2_+ > 1+\text{dim}(\Xi)$, only the zero dimensional parameterized spaces contribute. 
        \item When $b^2_+ = 1$, the space of metrics splits into chambers, and the \SW invariant for a single manifold depends upon a choice of chamber. Furthermore, higher-dimensional moduli spaces do contribute to the numerical invariants. There should be a similar theory in the critical case when $b^2_+ = 1+\dim(\Xi)$ family case. The case when $\dim(\Xi) = 1$ applies to diffeomorphisms of manifolds with $b^2_+ =2$, and has been worked out by Haochen Qiu~\cite{qiu:b+=2}. An interesting feature is that the wall-crossing term that appears when crossing the codimension one walls when $b^2_+ =1$ is replaced by a winding number around a codimension two set. 
        \item Higher-dimensional moduli spaces also apppear in an essential way in the Bauer-Furuta invariant~\cite{bauer-furuta:cohomotopy} and its family refinements~\cite{baraglia-konno:bf}. To apply our methods, one would need to abandon the local metric independence assumption and refine the argument to take the bundle structure of the equations into account.
    \end{enumerate}
\end{remark}

\vfill\newpage

%%%%%%%%% bib
\def\cprime{$'$}
\providecommand{\bysame}{\leavevmode\hbox to3em{\hrulefill}\thinspace}

\bibliographystyle{alphaurl}
\bibliography{diff}

\end{document}